\documentclass[preprint]{elsarticle}       
\usepackage{lipsum}
\makeatletter
\def\ps@pprintTitle{%
 \let\@oddhead\@empty
 \let\@evenhead\@empty
 \def\@oddfoot{}%
 \let\@evenfoot\@oddfoot}
\makeatother

\usepackage{lscape}
\usepackage{lineno,hyperref}
\usepackage{rotating}
\usepackage{makeidx}
\usepackage{float}
\usepackage{epsfig}
\usepackage{amsmath,amssymb}
\usepackage[latin1]{inputenc}
\usepackage{epic,eepic}
\usepackage{overpic}
\usepackage{color}
\usepackage{amsfonts}
\usepackage{graphicx}
\usepackage{mwe} 
\usepackage{subcaption}

\usepackage{algorithm}
\usepackage{algorithmicx}
\usepackage{algpseudocode}
\usepackage{enumitem}

\usepackage{tikz}
\usetikzlibrary{arrows.meta,calc,decorations.markings,patterns}

\usepackage{listings}
            {\hfill\penalty10000\raisebox{-.09em}{\large\bf\rm $\blacksquare$}\par\medskip}

\newtheorem{theorem}{Theorem}[section]

\newtheorem{proposition}[theorem]{Proposition}

\newtheorem{remark}{Remark}
\newtheorem{proof}{Proof}
[section]

\journal{}

\begin{document}

\small

\begin{frontmatter}

\title{Data dependent Shepard approximation through and adaptive modification of the shape parameter}\tnotetext[label1]{The second and fourth authors have been supported through GVA project CIAICO/2024/089, the fourth author has also been funded by grant  PID2023-146836NB-I00 funded by MCIN/AEI/10.13039/501100011033. The third author has been supported by NSFC (Grant No. 12022104 and 12371394).}

\author[UPCT]{Jos\'e Kuruc}
\ead{Jose.kuruc@hotmail.com}
\author[UPCT]{Juan Ruiz-\'Alvarez}
\ead{juan.ruiz@upct.es}
\author[HNU]{Bo Wang}
\ead{bowang@hunnu.edu.cn}
\author[UV]{Dionisio F. Y\'a\~nez}
\ead{Dionisio.Yanez@uv.es}

\date{Received: date / Accepted: date}

\address[UPCT]{Departamento de Matem\'atica Aplicada y Estad\'istica. Universidad  Polit\'ecnica de Cartagena, Cartagena (Spain).}
\address[HNU]{School of Mathematics and Statistics, Institute of Interdisciplinary Studies, Hunan Normal University, Hunan 410081, (P. R. China).}
\address[UV]{Departamento de Matem\'aticas. Universidad de Valencia, Valencia (Spain).}


\begin{abstract}
In this article, we introduce a novel data-dependent Shepard interpolation method inspired by the adaptive strategies proposed in \cite{RBF_arxiv}. In this case, as Shepard interpolation does not produce oscillations, our approach has the core objective of reducing the smearing near jump discontinuities in the data in one and two dimensions. While the original work in \cite{RBF_arxiv} focuses in on Radial Basis Function (RBF) interpolation, we extend these ideas to the Shepard framework by incorporating a data-dependent adaptation mechanism. Specifically, we modify the classical Shepard interpolation by adaptively adjusting the influence weights based on local smoothness indicators that modify the shape parameter. These indicators, similar to those used in \cite{RBF_arxiv}, are designed to detect discontinuities: for grid-based data, we use squared undivided second-order differences, and for scattered data, we employ squared least-squares approximations of the Laplacian scaled by the square of the mean local separation of stencil points. The resulting data-dependent weighting scheme forces the kernels close to a discontinuity to behave like a local delta function, effectively reducing the smearing of the discontinuities introduced by the classical Shepard approach.
We establish the theoretical foundation of the method, including the properties of the new interpolation and we theoretically prove that the reduction of the smearing of discontinuities is possible. Numerical experiments in one and two dimensions confirm that the proposed data-dependent Shepard interpolation significantly reduces the smearing of jump discontinuities while maintaining high accuracy in smooth regions.
\end{abstract}
\end{frontmatter}

\section{Introduction and review}

Shepard interpolation is a classical method for approximating multivariate functions based on weighted averages. Originally introduced by Donald Shepard in 1968 \cite{Shepard1968}, it has been widely used due to its simplicity and robustness, especially for scattered data. In this work, we propose a data-dependent Shepard interpolation method inspired by recent advances in adaptive radial basis function (RBF) interpolation, particularly those involving shape parameter modification near discontinuities \cite{RBF_arxiv}.
Given a set of data points ${(\mathbf{x}_i, y_i)}_{i=1}^N$, the goal is to construct an interpolant $s(\mathbf{x})$ such that $s(\mathbf{x}_i) = y_i$ for all $i$. In the classical Shepard method, the interpolant is defined as a normalized weighted average:
\begin{equation}\label{shepard1}
s(\mathbf{x}) = \frac{\displaystyle\sum_{i=1}^N w_i(\mathbf{x}) y_i}{\displaystyle\sum_{i=1}^N w_i(\mathbf{x})},
\end{equation}
where $w_i(\mathbf{x})$ are positive weights that typically decay with the distance between $\mathbf{x}$ and $\mathbf{x}_i$. In the classical Shepard method, the weights can be defined using radial basis functions with a constant shape parameter:
\begin{equation}\label{shepard2}
w_i(\mathbf{x}) = \phi(\varepsilon |\mathbf{x} - \mathbf{x}_i|),
\end{equation}
where $\phi$ is a radial basis function and $\varepsilon$ is a constant shape parameter.  In our data-dependent extension, we define a locally varying shape parameter that increases near discontinuities $\tilde\varepsilon_i$, effectively sharpening the kernel and reducing the influence of kernels crossing the discontinuity.
As in the RBF framework, the choice of $\phi$ significantly affects the behavior of the interpolant. Table \ref{tabla1nucleos} lists several commonly used radial basis kernels that can be used within the Shepard formulation. These kernels exhibit different degrees of smoothness and localization, and their behavior as $\varepsilon \to \infty$ mimics that of a discrete delta function, which has been found to be desirable near discontinuities in other contexts \cite{Jung2007, RBF_arxiv}.
\begin{table}[h!]
\centering
\begin{tabular}{lll}
\hline
$\phi(r,\varepsilon)$& RBF &                                                          \\ \hline
$e^{-  (\varepsilon r)^2}$ & Gaussian $\mathcal{C}^\infty$ & G                                          \\
$e^{-  \varepsilon r} \left( 1 +   \varepsilon r \right)$ & Mat\'ern $\mathcal{C}^2$ & M2 \\
$e^{-  \varepsilon r} \left( 3 + 3  \varepsilon r +   (\varepsilon r)^2 \right)$ & Mat\'ern $\mathcal{C}^4$ & M4 \\
$(1 -  \varepsilon r)^4_+ \left( 4 \varepsilon r + 1 \right)$ & Wendland $\mathcal{C}^2$ & W2 \\
$(1 -  \varepsilon r)^6_+ \left( 35 (\varepsilon r)^2 + 18 \varepsilon r + 3 \right)$ & Wendland $\mathcal{C}^4$ & W4 \\ \hline
\end{tabular}
\caption{Examples of RBFs that resemble a discrete delta function when $\varepsilon\to\infty$.}\label{tabla1nucleos}
\end{table}
In this data-dependent Shepard framework, the adaptivity of the shape parameter is guided by smoothness indicators, allowing the method to preserve accuracy in smooth regions while reducing the smearing of jump discontinuities in the data. This approach combines the local averaging nature of Shepard interpolation with the flexibility of RBF kernels, resulting in a robust interpolant for both smooth and piecewise-smooth functions.

In this work, we focus on the accurate approximation of functions exhibiting jump discontinuities, a well-known and challenging problem in numerical analysis. Classical global approximation techniques and high-order linear methods, including those relying on Radial Basis Functions (RBFs), often perform poorly in this setting. In particular, they typically generate spurious oscillations near discontinuities (the Gibbs phenomenon) and tend to smear sharp transitions, ultimately reducing the overall quality of the reconstruction. Over the years, numerous strategies have been proposed to overcome these difficulties (see, for example, \cite{Arandiga2005,Bozzini2013,Bozzini2014,Crampton2005,Gout2008,Lenarduzzi2017}). Although many of these contributions address general interpolation problems or surface reconstruction, only a limited number specifically target RBF-based approaches. An important contribution in this direction is due to Jung \cite{Jung2007}, who showed that multiquadric RBFs with shape parameters that decrease adaptively in the vicinity of discontinuities can substantially mitigate oscillatory behavior. This approach effectively produces a locally linear approximation, thereby enhancing both stability and accuracy. A more sophisticated class of techniques is based on variably scaled RBFs, where discontinuous scaling functions are incorporated around the locations of the jumps. As a consequence, the resulting interpolants are allowed to be discontinuous, which improves their ability to reproduce the true behavior of the underlying function \cite{DeMarchi2019,Romani2018,Erb2019}. However, these strategies usually rely on prior information about the position of discontinuities, making an initial edge-detection procedure unavoidable. 
In parallel, nonlinear reconstruction techniques such as Essentially Non-Oscillatory (ENO) and Weighted ENO (WENO) schemes \cite{Liu1994,JiangShu1996} have demonstrated remarkable success in capturing sharp gradients and discontinuities while maintaining high-order accuracy. These methods rely on adaptive weighting mechanisms based on local smoothness indicators, a principle that has inspired extensions to meshfree frameworks, including RBF-based ENO and WENO formulations. In particular, Guo and Jung \cite{Guo2017a,Guo2017b} introduced RBF-based ENO and WENO schemes in which the shape parameter is tuned to improve performance near discontinuities. More recently, Ar\`andiga et al. \cite{Arandiga2020} proposed a multiquadric RBF-WENO framework that combines locally several RBF kernels (typically two or three) in order to achieve high-order accuracy while effectively suppressing oscillations. At the same time, significant advances have been achieved in the development of high-order finite difference schemes for problems involving discontinuities or high gradients in other contexts. In particular, compact finite difference methods for the solution of PDEs have been shown to retain high-order accuracy even in the presence of strong irregularities in the solution or its coefficients \cite{LiIto2006, Feng2021,Feng2022}. Although these approaches are typically formulated on structured meshes, they highlight the importance of carefully designed approximation mechanisms when dealing with non-smooth data.

This manuscript is organized as follows. Section 2 introduces the data-dependent modification applied to the shape parameter, analyzes its impact on the affected kernel functions, and discusses the required properties of both the smoothness indicators and the composing function used to adapt the interpolation near discontinuities. It also presents the final form of the interpolation expression that achieves this adaptivity. Section 3 provides a theoretical analysis of the weights and examines the number of intervals affected by the discontinuity, considering both the classical Shepard interpolation and the data-dependent modification that we propose. Section 4 discusses the design of smoothness indicators in one and two dimensions, for both grid-based and mesh-free data. Section 5 presents numerical experiments supporting the theoretical findings and demonstrating the effectiveness of the proposed method. Finally, Section 6 summarizes the main conclusions.

\section{A new Shepard interpolation based on the data-dependent modification of the shape parameter}

In this section, we present an automatic procedure to locally adjust the shape parameter $\varepsilon$ in the weighting functions used in Shepard interpolation. The main objective is to reduce the influence radius of the kernels near discontinuities, where classical Shepard methods often produce undesirable smearing of jump discontinuities, while maintaining their standard behavior in smooth regions. This adaptive strategy enhances the interpolation quality close to the discontinuities, without compromising the properties of Shepard interpolation in regular areas.
Unlike the linear Shepard formulation, our approach introduces a data-dependent shape parameter that varies with the evaluation point, resulting in spatially varying weights that alter the structure of the interpolation process. To formalize these ideas, we consider the data-dependent Shepard interpolant defined as:
\begin{equation}\label{int_shepard}
\tilde{s}(\mathbf{x}) = \frac{\displaystyle\sum_{i=1}^N \phi(\tilde{\varepsilon}_i \|\mathbf{x} - \mathbf{x}_i\|) y_i}{\displaystyle\sum_{i=1}^N \phi(\tilde{\varepsilon}_i \|\mathbf{x} - \mathbf{x}_i\|)},
\end{equation}
where $\tilde{\varepsilon}_i$ is a point-dependent shape parameter that increases near detected discontinuities. This causes the kernel $\phi$ to become more localized, effectively reducing the contribution of distant data points to the final interpolation. The resulting interpolant adapts its behavior based on the local smoothness of the data, and we will show that offers improved performance for the interpolation of piecewise-smooth data. Thus, we can say that the approach introduced in this work relies on a modified form of Shepard interpolation in which the influence of the weighting functions is deliberately reduced in the vicinity of discontinuities. The key mechanism behind this strategy is a data-dependent adjustment of the shape parameter $\varepsilon$, which is forced to grow near non-smooth regions while remaining almost unchanged in smooth areas. When this occurs, the corresponding weights become extremely concentrated around their associated data points, behaving almost like delta functions and interacting only weakly with neighboring nodes. This strong localization alters the overall interpolation procedure: each modified weight essentially affects only the point at which it is centered, while contributions from the rest of the unaltered weights are the same. The concept applies uniformly to all kernels listed in Table \ref{tabla1nucleos}.

Following the ideas exposed previously, we define the adaptive shape parameter as:
\begin{equation}\label{tildegamma_new}
\tilde{\varepsilon}_i = \varepsilon \cdot \frac{1}{c + g(I_i)},
\end{equation}
where $c$ is a small constant introduced to prevent division by zero, and $g$ is a decaying function given by:
\begin{equation}\label{gx}
g(x) = e^{- (C x)^t},
\end{equation}
with $C > 0$ and $t \ge 1$ being tunable parameters that control the sensitivity of the adaptation. This formulation ensures that $\tilde{\varepsilon}_i$ becomes large near discontinuities (where $I_i$ is large), resulting in sharply localized kernels that approximate a discrete delta function. In smooth regions, $g(I_i)$ remains close to 1, and $\tilde{\varepsilon}_i$ stays near its original value $\varepsilon$.
In our numerical tests, we use the values $c = 10^{-16}$, $C = 1$, and $t = 1$. The quantity $I_i$ is positive and represents a smoothness indicator, a localized numerical measure of how regular the data is around the point $\mathbf{x}_i$. For the indicator to be effective, it should satisfy the following properties:
\begin{enumerate}[label={\bfseries P\arabic*}]
\item\label{P1sm1d_new} In regions where the function is smooth, the indicator should scale with the grid spacing as $I_i = \mathcal{O}(h^m)$ for some $m > 0$.
\item\label{P2sm1d_new} If a discontinuity lies within the stencil $\mathcal{S}_i$, then $I_i = \mathcal{O}(1)$ as $h \to 0$.
\end{enumerate}
These properties ensure that the shape parameter adapts appropriately to the local behavior of the data.
We now prove the following result, Prop. \ref{prop1},  using Taylor's expansion. 
\begin{proposition}\label{prop1}
Let $\tilde{\varepsilon}_i$ be defined as in Equation \eqref{tildegamma_new}, and assume the smoothness indicator $I_i$ satisfies properties \ref{P1sm1d_new} and \ref{P2sm1d_new}, then
$$
\tilde{\varepsilon}_i=\begin{cases}
\varepsilon(1+O(h^{mt}))),& \text{in smooth regions.}\\
\frac\varepsilon2 H(c^{-1},e^{(CI_i)^t}), & \text{near discontinuties}.
\end{cases}
$$
where $H$ is the Harmonic mean. 
\end{proposition}

Since $H$ is a mean, the following property holds  $\min(x,y)\leq H(x,y)\leq \max(x,y)$, for all $x,y\in(0,+\infty)$. Thus, near discontinuities, if the smoothness indicator is large, then 
$\tilde{\varepsilon}_i$ is also large, and we obtain the following property: If we use radial basis functions $\phi$ with the behavior described in Table \ref{tabla1nucleos},  we have that
\begin{itemize}
\item In smooth regions, $\phi(r, \tilde{\varepsilon}_i) \approx \phi(r, \varepsilon)$.
\item Near discontinuities, $\phi(r, \tilde{\varepsilon}_i) \approx \delta(r)$, where $\delta(r)$ is a discrete delta function defined by:
\[
\delta(r) =
\begin{cases}
1, & r = 0, \\
0, & r \neq 0.
\end{cases}
\]
\end{itemize}
The kernels listed in Table \ref{tabla1nucleos} satisfy these conditions and are discussed in detail in \cite{Fasshauer2007}. For some of these functions, we use the cutoff operator $(\cdot)_+ : \mathbb{R} \to [0,+\infty)$, defined as:

 \begin{equation*}
 (x)_+=\begin{cases}
 x, & x\geq 0,\\
 0, & x<0.
 \end{cases}
 \end{equation*}

\section{Error around the discontinuity}
In this section we analyze the interpolation error in the vicinity of the discontinuity, focusing on its dependence on the shape parameter. In particular, we study the regime in which the shape parameter is allowed to grow unbounded within a thin belt surrounding the discontinuity, and we quantify how this scaling influences the magnitude and spatial extent of the error.


As we mentioned in the previous section, we adapt the shape parameter of Shepard weights so that kernels whose stencils \emph{see} a jump
become very localized (delta-like). We want to proof that this strategy can reduce wrong-side influence near the interface to
\emph{thin the smearing belt} of classical Shepard interpolation. For this, we analyse the exact error identity across the interface. 

Let  $X=\{\mathbf{x}_i\}_{i=1}^N\subset\Omega\subset\mathbb{R}^d$ be the nodes with $\Omega$ an open bounded set, the data $y_i=f(\mathbf{x}_i)$ and let $\Sigma$ split $\Omega=\Omega^{+}\cup\Omega^{-}$ and set
$I^{+}=\{i:\ \mathbf{x}_i\in\Omega^{+}\}$, $I^{-}=\{i:\ \mathbf{x}_i\in\Omega^{-}\}$. We measure plus-side (equivalently minus-side) support pointwise with the
\emph{local plus-side  distance}
\[
h_{X^{+}}(\mathbf{x}) := \inf_{j\in I^{+}}\|\mathbf{x}-\mathbf{x}_j\|.
\]
We compare the modified Shepard method, $\tilde{s}$, introduced in \eqref{int_shepard}, with adaptive $\varepsilon_i$ defined in \eqref{tildegamma_new}, with $I_i$ being a smoothness indicator (introduced in the next section), satisfying properties \ref{P1sm1d_new} and \ref{P2sm1d_new} and  the nodewise Shepard method whose form is
\begin{equation}\label{shep}
s(\mathbf{x})
=
\frac{\sum_{i=1}^N \phi\!\bigl(\varepsilon\|\mathbf{x}-\mathbf{x}_i\|\bigr)\,f(\mathbf{x}_i)}
     {\sum_{i=1}^N \phi\!\bigl(\varepsilon\|\mathbf{x}-\mathbf{x}_i\|\bigr)},\qquad \varepsilon>0.
\end{equation}
 We suppose $\tilde\varepsilon_i\approx\varepsilon$ in smooth regions and $\tilde\varepsilon_i\gg\varepsilon$
on jump-crossing stencils (where kernels become highly localized).
 For $\mathbf{x}\in\Omega^{+}$ (the point where we want to approximate) we define
\[
W^{\pm}(\mathbf{x})=\sum_{i\in I^{\pm}}\phi(\varepsilon\|\mathbf{x}-\mathbf{x}_i\|),\quad
\omega^{\pm}(\mathbf{x})=\frac{W^{\pm}(\mathbf{x})}{W^{+}(\mathbf{x})+W^{-}(\mathbf{x})},\quad
\bar f^{\pm}(\mathbf{x})=\frac{1}{W^{\pm}(\mathbf{x})}\sum_{i\in I^{\pm}}\phi(\varepsilon\|\mathbf{x}-\mathbf{x}_i\|)\,f(\mathbf{x}_i).
\]
In a similar way, we can define:
\[
\tilde{W}^{\pm}(\mathbf{x})=\sum_{i\in I^{\pm}}\phi(\tilde{\varepsilon}_i\|\mathbf{x}-x_i\|),\quad
\tilde{\omega}^{\pm}(\mathbf{x})=\frac{\tilde{W}^{\pm}(\mathbf{x})}{\tilde{W}^{+}(\mathbf{x})+\tilde{W}^{-}(\mathbf{x})},\quad
\tilde{\bar f}^{\pm}(\mathbf{x})=\frac{1}{\tilde{W}^{\pm}(\mathbf{x})}\sum_{i\in I^{\pm}}\phi(\tilde{\varepsilon}_i\|\mathbf{x}-\mathbf{x}_i\|)\,f(\mathbf{x}_i).
\]
The error at the plus-side (equivalent for the minus side) of the discontinuity is then given by the identity
\begin{equation}\label{eq:error-split}
E(\mathbf{x}) := s(\mathbf{x})-f^{+}(\mathbf{x})
= \underbrace{\omega^{+}(\mathbf{x})(\bar f^{+}(\mathbf{x})-f^{+}(\mathbf{x}))}_{\text{smooth-side interpolation}}
+ \underbrace{\omega^{-}(\mathbf{x})(\bar f^{-}(\mathbf{x})-f^{+}(\mathbf{x}))}_{\text{wrong-side leakage}}.\
\end{equation}

We now focus on $s$ (\ref{shepard1}), as the analysis for $\tilde{s}$ (\ref{int_shepard}) is analogous. To prove that the reduction of the smearing region around the discontinuity is possible, we make the following assumptions:
\begin{itemize}
\item[(A1)] We need to assume an exponential kernel decay: $\phi$ is nonincreasing and
$\phi(\lambda)\le C_1 e^{-C_2\lambda^p}$ for some $C_1,C_2>0$, $p\ge 1$.
\item[(A2)] Near/far split on the wrong side ($-$ side, in the case we are considering). Fix $R>0$ and define
\[
I^{-}_{\mathrm{near}}=\{i\in I^{-}:\ \operatorname{dist}(\mathbf{x}_i,\Sigma)\le R\},\qquad
I^{-}_{\mathrm{far}}=\{i\in I^{-}:\ \operatorname{dist}(\mathbf{x}_i,\Sigma)> R\}.
\]
\item[(A3)] Wrong-side trigger. On the wrong-side of the discontinuity, in a near belt around it, $\tilde\varepsilon_i\ge \kappa\,\varepsilon$
for some $\kappa>1$ (this follows from \ref{P2sm1d_new} assuming a small value of $c$ in (\ref{tildegamma_new})).
\end{itemize}

In the proof we use three elementary bounds. Let $d(\mathbf{x})=\operatorname{dist}(\mathbf{x},\Sigma)$:
\begin{itemize}
  \item \textbf{(B1) Near wrong-side bound (constant/adaptive).}
  For $d(\mathbf{x})=\operatorname{dist}(\mathbf{x},\Sigma)$ and $A:=|I^{-}_{\mathrm{near}}|\,C_1$,
  \begin{equation}\label{eq:B1}
    W^{-}_{\mathrm{near}}(\mathbf{x})
    \;\le\; A\,\exp\!\big(-C_2(\eta\, d(\mathbf{x}))^p\big),
    \qquad
    \eta=
    \begin{cases}
      \varepsilon, & \text{constant-}\varepsilon \text{ method},\\[1mm]
      \kappa\,\varepsilon, & \text{adaptive method}.
    \end{cases}
  \end{equation}
  Here we used that $\|\mathbf{x}-\mathbf{x}_i\|\ge d(\mathbf{x})$ for all $i\in I^{-}_{\mathrm{near}}$ and the exponential decay in (A1).

  \item \textbf{(B2) There is a far wrong-side floor for the error that is distance independent.}
  \begin{equation}\label{eq:B2}
    W^{-}_{\mathrm{far}}(\mathbf{x}) \;\le\; B_{\mathrm{far}},
  \end{equation}
  with 
  \[
B_{\mathrm{far}}
:= |I^{-}_{\mathrm{far}}|\,
C_1 \exp\!\bigl(- C_2 (\varepsilon_{\min} R)^{p} \bigr),
\qquad
\varepsilon_{\min}=\frac{\varepsilon}{1+c}.
\]
  This is $d$-independent because far stencils lie in smooth regions (so \ref{P1sm1d_new} rules), hence
  $\tilde\varepsilon_i\approx \varepsilon/(1+c)$, and (A1) gives a uniform bound.

  \item \textbf{(B3) Plus-side denominator bound (pointwise).}
  \begin{equation}\label{eq:B3}
    W^{+}(\mathbf{x})
    \;\ge\; \phi\!\bigl(\varepsilon_{\mathrm{den}}(\mathbf{x})\,h_{X^{+}}(\mathbf{x})\bigr)
    \;=:\; \phi_{\mathrm{den}}(\mathbf{x}).
  \end{equation}
  Here $\varepsilon_{\mathrm{den}}(\mathbf{x})=\varepsilon$ for the constant method and
  $\varepsilon_{\mathrm{den}}(\mathbf{x})$ is the value of the shape parameter at the nearest plus-side node $\mathbf{x}_j$
  for the adaptive method. 
The bound follows since $W^{+}$ is a sum of
nonnegative terms, $W^{+}(\mathbf{x})\ge \phi(\varepsilon_{\mathrm{den}}(\mathbf{x}) h_{X^{+}})>0$.
\end{itemize}

Now we can define two distance-independent floors for the errors. The first one is $\delta_{\mathrm{far}}(\mathbf{x})$, which is the contribution of distant kernels at the wrong side (wrong-side leakage in (\ref{eq:error-split}) for far away data). The second one is $\delta_{\mathrm{smooth}}$, that is given by the accuracy of the method considering smooth data:
\[
\delta_{\mathrm{far}}(\mathbf{x}):=\frac{\|f\|_{d}}{\phi_{\mathrm{den}}(\mathbf{x})}\,B_{\mathrm{far}},
\qquad
\delta_{\mathrm{smooth}}(\mathbf{x}):=\sup_{\text{near }\Sigma}|\omega^{+}(\mathbf{x})(\bar f^{+}(\mathbf{x})-f^{+}(\mathbf{x}))|\le C\,h^{q}.
\]
Here, we have assumed that,
\begin{equation}\label{salto}
\|f\|_{d}
:= \max\{\bigl| f^+(\mathbf{x}_i) - f^-(\mathbf{x}_j) \bigr|: i\in I^+_{\mathrm{near}},\,\,j \in I^-_{\mathrm{near}}\}.
\end{equation}
represents the maximum size of the jump in the data.

To analyse the reduction of the smearing belt fixing a tolerance for the error $\delta$, we need to assure that
\begin{equation}\label{eq:Rdelta-local}
\delta\ >\ \delta_{\mathrm{far}}(\mathbf{x})+\delta_{\mathrm{smooth}}(\mathbf{x})
\quad\text{for all $\mathbf{x}$ in the near belt}.
\end{equation}
This ensures that the error tolerance selected is inside the smearing belt boundary, and is governed by the \emph{near wrong-side} contribution of the data. Now we are prepared to give the following main result.

\begin{theorem}[The smearing belt around the discontinuity can be reduced]
\label{thm:thinning-local}
Under \textnormal{(A1)}-\textnormal{(A3)} and a tolerance \eqref{eq:Rdelta-local}, we define the
distance thresholds
\[
d_{\star}^{\mathrm{const}}(\mathbf{x})
:=\frac{1}{\varepsilon}\left[
\frac{1}{C_2}\log\!\Big(
\frac{A\,\|f\|_{d}}{
\delta\,\phi(\varepsilon\,h_{X^{+}}(\mathbf{x}))-\|f\|_{d}\,B_{\mathrm{far}}}
\Big)\right]^{\!1/p},
\]
\[
d_{\star}^{\mathrm{adap}}(\mathbf{x})
:=\frac{1}{\kappa\,\varepsilon}\left[
\frac{1}{C_2}\log\!\Big(
\frac{A\,\|f\|_{d}}{
\delta\,\phi(\varepsilon_{\mathrm{den}}(\mathbf{x})\,h_{X^{+}}(\mathbf{x}))
-\|f\|_{d}\,B_{\mathrm{far}}}
\Big)\right]^{\!1/p}.
\]
Then for each method (constant/adaptive), \emph{if} $d(\mathbf{x})\ge d_{\star}^{\mathrm{method}}(\mathbf{x})$
we have $|E(\mathbf{x})|\le \delta$. Consequently, the $\delta$-smearing belt in $\Omega^{+}$ is contained in
a tube of half-width $\sup_{\mathbf{x}\text{ on the belt}} d_{\star}^{\mathrm{method}}(\mathbf{x})$.
Moreover,
\[
\frac{\operatorname{width}(B_{\delta}^{\mathrm{adap}})}
     {\operatorname{width}(B_{\delta}^{\mathrm{const}})}
\ \le\
\frac{1}{\kappa}\left(
\frac{
\log\Big(\!\dfrac{A\,\|f\|_{d}}{\delta\,\phi(\varepsilon_{\mathrm{den}}\,h_{X^{+}}(\mathbf{x}))-\|f\|_{d}\,B_{\mathrm{far}}}\Big)
}{
\log\Big(\!\dfrac{A\,\|f\|_{d}}{\delta\,\phi(\varepsilon\,h_{X^{+}}(\mathbf{x}))-\|f\|_{d}\,B_{\mathrm{far}}}\Big)
}
\right)^{\!1/p},
\]
so the adaptive belt can be thinner by about a factor $1/\kappa$ (the logarithmic ratio is close to one).
\end{theorem}

\begin{proof}
We start from the exact identity and isolate the leakage.
From \eqref{eq:error-split} and $|\omega^{+}(\mathbf{x})(\bar f^{+}(\mathbf{x})-f^{+}(\mathbf{x}))|\le \delta_{\mathrm{smooth}}(\mathbf{x})$,
\[
|E(\mathbf{x})|\ \le\ \delta_{\mathrm{smooth}}(\mathbf{x})
+ \omega^{-}(\mathbf{x})\,|\bar f^{-}(\mathbf{x})-f^{+}(\mathbf{x})|
\ \le\ \delta_{\mathrm{smooth}}(\mathbf{x})
+ \frac{W^{-}(\mathbf{x})}{W^{+}(\mathbf{x})}\,\|f\|_{d},
\]
as by definition of $\|f\|_{d}$ in (\ref{salto}),
$|\bar f^{-}(\mathbf{x})-f^{+}(\mathbf{x})|\le \|f\|_{d}$ in the near belt.

\smallskip
Now, we insert the denominator lower bound and split $W^{-}$.
Using (B3), $W^{+}(\mathbf{x})\ge \phi_{\mathrm{den}}(\mathbf{x})
=\phi(\varepsilon_{\mathrm{den}}\,h_{X^{+}}(\mathbf{x}))$,
and the near/far split $W^{-}(\mathbf{x})=W^{-}_{\mathrm{near}}(\mathbf{x})+W^{-}_{\mathrm{far}}(\mathbf{x})$,
\[
|E(\mathbf{x})|\ \le\ \delta_{\mathrm{smooth}}(\mathbf{x})
+ \frac{\|f\|_{d}}{\phi_{\mathrm{den}}(\mathbf{x})}
\Big(W^{-}_{\mathrm{near}}(\mathbf{x})+W^{-}_{\mathrm{far}}(\mathbf{x})\Big).
\]
\smallskip
By \eqref{eq:Rdelta-local}, the \emph{gap}
\[
\Gamma(\mathbf{x})
:=\delta\,\phi_{\mathrm{den}}(\mathbf{x})
-\|f\|_{d}\,B_{\mathrm{far}}
-\phi_{\mathrm{den}}(\mathbf{x})\,\delta_{\mathrm{smooth}}(\mathbf{x})
\]
is strictly positive (the tolerance exceeds the floors). Therefore a \emph{sufficient} condition for
$|E(\mathbf{x})|\le \delta$ is
\[
  \delta_{\mathrm{smooth}}(\mathbf{x})
+ \frac{\|f\|_{d}}{\phi_{\mathrm{den}}(\mathbf{x})}
\Big(W^{-}_{\mathrm{near}}(\mathbf{x})+W^{-}_{\mathrm{far}}(\mathbf{x})\Big)\le\delta,
\]
and
\[
W^{-}_{\mathrm{near}}(\mathbf{x})\ \le\ \frac{\Gamma(\mathbf{x})}{\|f\|_{d}}
\ \le\ \frac{\delta\,\phi_{\mathrm{den}}(\mathbf{x})-\|f\|_{d}\,B_{\mathrm{far}}}{\|f\|_{d}}.
\]
where we have dropped $\phi_{\mathrm{den}}(\mathbf{x})\delta_{\mathrm{smooth}}(\mathbf{x})$, that is a positive quantity.

\smallskip
Now we use the near bound and solve for the distance $d$. By (B1), in the constant-$\varepsilon$ case,
$W^{-}_{\mathrm{near}}(\mathbf{x})\le A\,e^{-C_2(\varepsilon d)^p}$, hence it suffices that
\[
A\,e^{-C_2(\varepsilon d)^p}
\ \le\
\frac{\delta\,\phi(\varepsilon\,h_{X^{+}}(\mathbf{x}))-\|f\|_{d}\,B_{\mathrm{far}}}{\|f\|_{d}}.
\]
Taking logs and solving for $d$ gives
\[
d\ \ge\
\frac{1}{\varepsilon}\left[
\frac{1}{C_2}\log\!\Big(
\frac{A\,\|f\|_{d}}{
\delta\,\phi(\varepsilon\,h_{X^{+}}(\mathbf{x}))-\|f\|_{d}\,B_{\mathrm{far}}}
\Big)\right]^{\!1/p}
=:d_{\star}^{\mathrm{const}}(\mathbf{x}).
\]
The adaptive case is identical, replacing $\varepsilon$ by $\kappa\varepsilon$ (by (A3)) and
$\phi(\varepsilon h_{X^{+}})$ by $\phi(\varepsilon_{\mathrm{den}}(\mathbf{x}) h_{X^{+}})$, yielding
$d_{\star}^{\mathrm{adap}}(\mathbf{x})$.

\smallskip

Now, if $d(\mathbf{x})\ge d_{\star}^{\mathrm{method}}(\mathbf{x})$ then $|E(\mathbf{x})|\le \delta$, so the belt
is contained in the tube $\{d(\mathbf{x})<\sup d_{\star}^{\mathrm{method}}(\mathbf{x})\}$. This gives the
stated width bound and the ratio inequality, with the explicit $1/\kappa$ factor.
\end{proof}

\begin{remark}[The result is stable under refinement]
As the sampling refines and $h_{X^{+}}(\mathbf{x})\to 0$, continuity of $\phi$ at $0$ implies
$\phi(\varepsilon\,h_{X^{+}}(\mathbf{x}))\to \phi(0)$ and
$\phi(\varepsilon_{\mathrm{den}}(\mathbf{x})\,h_{X^{+}}(\mathbf{x}))\to \phi(0)$ (because $c>0$ in (\ref{tildegamma_new})
keeps $\varepsilon_{\mathrm{den}}h_{X^{+}}\to 0$). Hence the two logarithms in the width ratio
have the same limit and the leading thinning factor is $1/\kappa$.
\end{remark}

\begin{remark}[We need to choose the near/far radius $R$ coherently with the indicator]
Let $S(\mathbf{x}_i)$ be the stencil used to compute $I_i$ (next section) and define its radius
$a_i:=\max_{\mathbf{x}_j\in S(\mathbf{x}_i)}\|\mathbf{x}_j-\mathbf{x}_i\|$. If $\operatorname{dist}(\mathbf{x}_i,\Sigma)>a_i$, the stencil cannot
cross the jump, so $i\in I^{-}_{\mathrm{far}}$. A single guard $R:=\sup_{i\in I^{-}} a_i$  guarantees (A2).
\end{remark}

\begin{remark}[For compact support kernels, the far floor vanishes inside a safe belt]
If $\phi(\varepsilon r)=\psi(\varepsilon r)$ with $\psi(s)=0$ for $s\ge \rho$, then choosing $R$ so that
$R-d\ge \rho/\varepsilon_{\min}$ with $\varepsilon_{\min}=\varepsilon/(1+c)$ implies
$W^{-}_{\mathrm{far}}(\mathbf{x})=0$ throughout the belt. The same proof applies, now with
$B_{\mathrm{far}}=0$ (which implies a slightly thinner belt).
\end{remark}

\section{Discussion about the smoothness indicators in one and several dimensions, and for regular and mesh-free data}\label{si_sec}

A smoothness indicator must satisfy the two conditions P1 and P2 described in Section 2, which ensure that the indicator vanishes for smooth data and grows appropriately near discontinuities. A valid choice for uniform univariate data in this context is the squared undivided second-order difference, given by
\begin{equation}\label{ud1}
\beta = \left( f(x+h) - 2f(x) + f(x-h) \right)^2,
\end{equation}
which has size $\mathcal{O}(h^4)$ for sufficiently smooth functions. In the same sense, for data in two dimensions using an uniform grid, we can use the squared five points stencil Laplacian as smoothness indicator, which is given by,
\begin{equation}\label{ud2}
\beta = \left( f(x+h, y) + f(x-h, y) + f(x, y+h) + f(x, y-h) - 4f(x, y) \right)^2.
\end{equation}
For non-uniform data in one or several dimensions, one can use a different strategy. 
Let $X = \{x_i\}_{i=1}^N \subset \Omega \subset \mathbb{R}^2$ be a set of scattered nodes and let $u : X \to \mathbb{R}$ denote pointwise samples of a sufficiently smooth function $u$. For each centre $x_0 \in X$ choose a local stencil of $K$ neighbours $S(x_0) = \{x_j\}_{j=1}^K$ with $x_1 = x_0$. We seek weights $w_j$ such that
\[
\Delta u(x_0) \approx \sum_{j=1}^K w_j u(x_j),
\]
and that this formula is exact for a finite polynomial space. A similar idea is used, for example, in \cite{LiIto2006} (see pages 68-69) to compute surface derivatives in several dimensions. Let $P_m$ denote the polynomials in two variables of total degree $\leq m$. Imposing exactness on a basis $\{p_\ell\}_{\ell=1}^M$ of $P_m$ yields the moment conditions
\[
\sum_{j=1}^K w_j p_\ell(x_j) = \Delta p_\ell(x_0), \quad \ell = 1, \dots, M,
\]
where $M = \dim P_m$. In matrix form this constraint reads
\[
V_w w = b,
\]
with $V_w \in \mathbb{R}^{M \times K}$, $(V_w)_{\ell j} = p_\ell(x_j)$, $w = (w_j)_{j=1}^K$ and $b_\ell = \Delta p_\ell(x_0)$. When $K = M$ and $V_w$ is nonsingular, the weights are obtained by direct solve $w = V_w^{-1} b$. More robustly, one typically takes $K > M$ and computes weights by solving the normal equations
\[
(V_w V_w^\top) \alpha = b, \quad w = V_w^\top \alpha,
\]
i.e., find $\alpha \in \mathbb{R}^M$ such that $V_w^\top \alpha$ is the weight vector whose moments match $b$. Small Tikhonov regularization may be added if $V_w V_w^\top$ is ill-conditioned (a tolerance for the condition number of $10^{-10}$ is used in the numerical experiments):
\[
(V_w V_w^\top + \lambda I) \alpha = b, \quad w = V_w^\top \alpha, \quad \lambda > 0.
\]
For $m = 2$ and univariate data, the monomial basis that we use is $\{1, x-x_0, (x-x_0)^2\}$. Then we have $M = 3$ and the right-hand side is
\[
b = ( 0, 0, 2)^\top.
\]
Analogously, for data in 2D and $m = 2$, we use the monomial basis $\{1, (x-x_0), (y-y_0), (x-x_0)^2, (x-x_0)(y-y_0), (y-y_0)^2\}$, we have $M = 6$ and the right-hand side is
\[
b = (0, 0, 0, 2, 0, 2)^\top,
\]
because $\Delta 1 = \Delta x = \Delta y = 0$, $\Delta x^2 = \Delta y^2 = 2$, $\Delta(xy) = 0$. The resulting weight vector $w$ is then applied to the data values $u(x_j)$ on the stencil to produce the discrete Laplacian approximation $\sum_j w_j u(x_j)$.

As we want to obtain something equivalent to the undivided second differences but for sparse data in several dimensions, we need to define a characteristic scale for the stencil. To do so, we compute a local fill-scale $h_{\text{loc}}(x_0)$. A simple and effective choice is the arithmetic mean of the Euclidean distances from $x_0$ to the other stencil nodes (excluding the zero self-distance):
\[
h_{\text{loc}}(x_0) = \frac{1}{K-1} \sum_{j=2}^K \|x_j - x_0\|.
\]
This $h_{\text{loc}}$ approximates the typical spacing of the stencil. Now, we can multiply the computed $\Delta$-approximation by $h_{\text{loc}}^2$ to obtain
\[
L[u] := h_{\text{loc}}(x_0)^2 \sum_{j=1}^K w_j u(x_j),
\]
which has the same units as a discrete undivided second-difference operator. We use $(L[u])^2$ as smoothness indicator.

A good remark here is that using this approach with 3 points in the univariate case or 5 points in the bivariate case with a uniform grid will return the expressions of the smoothness indicators in (\ref{ud1}) and (\ref{ud2}) respectively. Also, a similar approach can be used for more than 2 dimensions.

\section{Numerical experiments}
In this section, we carry out a collection of numerical tests aimed at assessing the performance of the data-dependent Shepard interpolation strategies introduced in this work. We begin with examples of smooth functions, which allow us to examine both the accuracy and the numerical reliability of the method under different kernel choices in the Shepard weights, and to contrast these results with those obtained using the standard, non-adaptive Shepard formulation. We present the results using the distribution of the point-wise absolute error.

After the smooth-function tests, we turn our attention to functions that present jump discontinuities. Here we explore how the data-dependent Shepard variants behave in the vicinity of the discontinuity, and how effectively they reduce smoothing artifacts compared with the classical formulation. Experiments are conducted in one and two spatial dimensions to illustrate the flexibility of the approach. For all tests, we include both regularly spaced grids and irregularly distributed point sets generated using Halton sequences, so that the robustness of the algorithms can be examined under a variety of sampling patterns. We also present plots of the distribution of the point-wise absolute error over the domain.

\subsection{Behaviour in smooth regions}

In this experiment we analyse the performance of the proposed data-dependent Shepard interpolation when the function is smooth. As a reference example, we consider
\begin{equation}\label{funcion_f}
f(x)=1+\sin(\pi x), \qquad x\in\mathbb{R},
\end{equation}
which is analytic in the whole real line. The accuracy is checked on the interval $[0,1]$.

Two interpolation strategies are compared, using the following notation:
\begin{itemize}
    \item $\textsc{Shepard}_{\mathcal H}$: classical Shepard method with kernel $\mathcal H$,
    \item $\textsc{NL\!-\!Shepard}_{\mathcal H}$: data-dependent variant incorporating smoothness-based adaptation of the weights.
\end{itemize}

The subscript $\mathcal H\in\{\text{G},\text{W2},\text{W4},\text{M2},\text{M4}\}$ identifies the radial kernel employed in the Shepard weights (Gaussian, Wendland $\mathcal{C}^2$ or $\mathcal{C}^4$, and Mat\'ern $\mathcal{C}^2$ or $\mathcal{C}^4$, respectively), following the list in Table~\ref{tabla1nucleos}.

Let \(\{z_j\}_{j=0}^{m}\) be an evaluation grid on the interval \([0,1]\), which may be uniform or non-uniform. At each evaluation point \(z_j\), we compute the values of the linear (\ref{shepard1}) and data-dependent (\ref{int_shepard}) Shepard approximations, denoted generically by \(\mathcal{S}(z_j)\). To analyze the pointwise distribution of the interpolation error, we define
\begin{equation}\label{error}
e_j = \big|\, f(z_j) - \mathcal{S}(z_j)\,\big|.
\end{equation}
For the classical method, the shape parameter $\varepsilon$ associated with the kernel is kept fixed and chosen so as to guarantee both numerical stability and good approximation properties.  
The values used in this test are
\begin{equation}\label{vareps}
\varepsilon = 
\begin{cases}
1/h, & \text{for Gaussian},\\[4pt]
0.3/h, & \text{for Wendland W2, W4},\\[4pt]
2/h, & \text{for Mat\'ern M2},\\[4pt]
3/h, & \text{for Mat\'ern M4},
\end{cases}
\end{equation}
where $h$ denotes the grid spacing for uniform nodes, or the fill distance in the Halton case.  
For the data-dependent Shepard version, the same baseline value is adopted, and then modified using the smoothness indicator described in previous sections.  

The Halton sets are generated with the Matlab instruction
\begin{lstlisting}	
p = haltonset(1,'Skip',0,'Leap',38);
\end{lstlisting}

The interpolation centers coincide with the original data sites. In the adaptive rules we set $c = 10^{-16}$ in (\ref{tildegamma_new}) and $C = 1$, $t = 1$ in (\ref{gx}). 

In the numerical experiments, we begin with an initial set of \(32\) data points on the interval \([0,1]\). Between each pair of consecutive initial points, we insert \(20\) additional evaluation points, which may be uniformly or non-uniformly distributed, obtaining in this way a refined grid. As mentioned before, at each evaluation point \(z_j\), we compute the values of both the classical Shepard interpolant and the data-dependent Shepard-type interpolant introduced in this work, and then compute the point-wise absolute value of the error through the expression in (\ref{error}) to obtain the physical distribution of the error over the domain.  Figure \ref{exp1_sm} presents experiments for gridded data and Figure \ref{exp2_sm} for Halton points. In both cases, the grids follow the descriptions given before. In all the experiments presented in Figures \ref{exp1_sm} and \ref{exp2_sm}, we display the classical Shepard approximation using a blue solid line, our data-dependent approximation using a black solid line, and the original sampled data using red points. The corresponding error plots show the point-wise absolute error for both methods: the blue curve represents the error of the classical Shepard approximation, while the red curve represents the error obtained by our algorithm. As it can be seen in Figures \ref{exp1_sm} and \ref{exp2_sm}, for smooth test functions, the errors produced by both approaches are virtually identical.

\begin{figure}[htbp!]
\begin{center}
		\begin{tabular}{cc}
	\includegraphics[width=5.2cm]{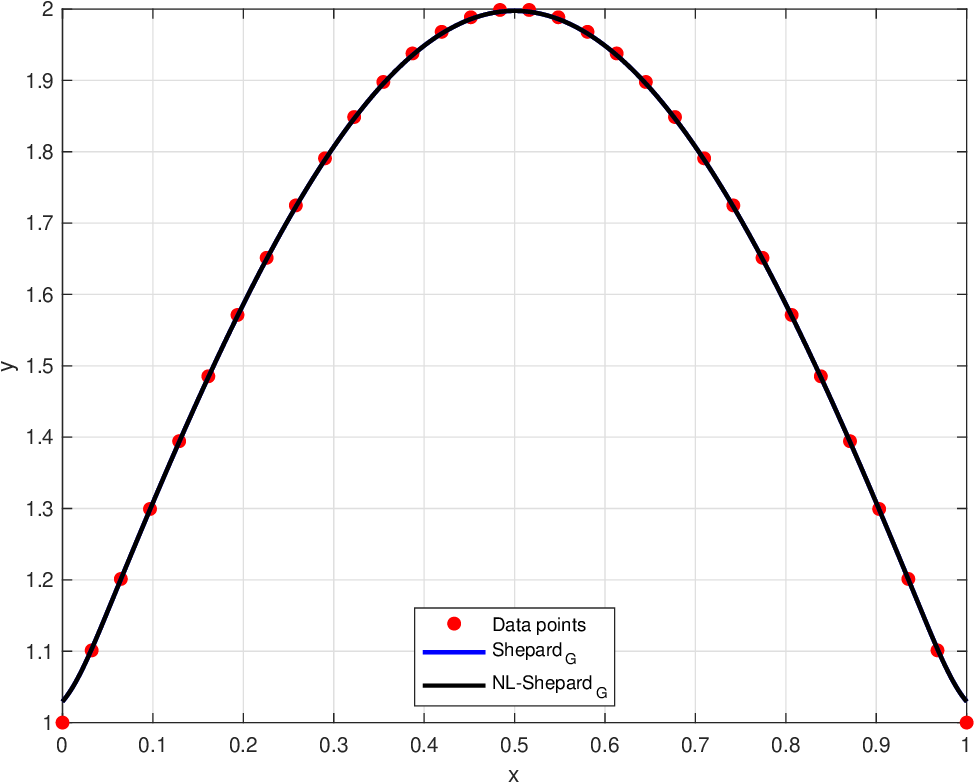} & 	\includegraphics[width=5.2cm]{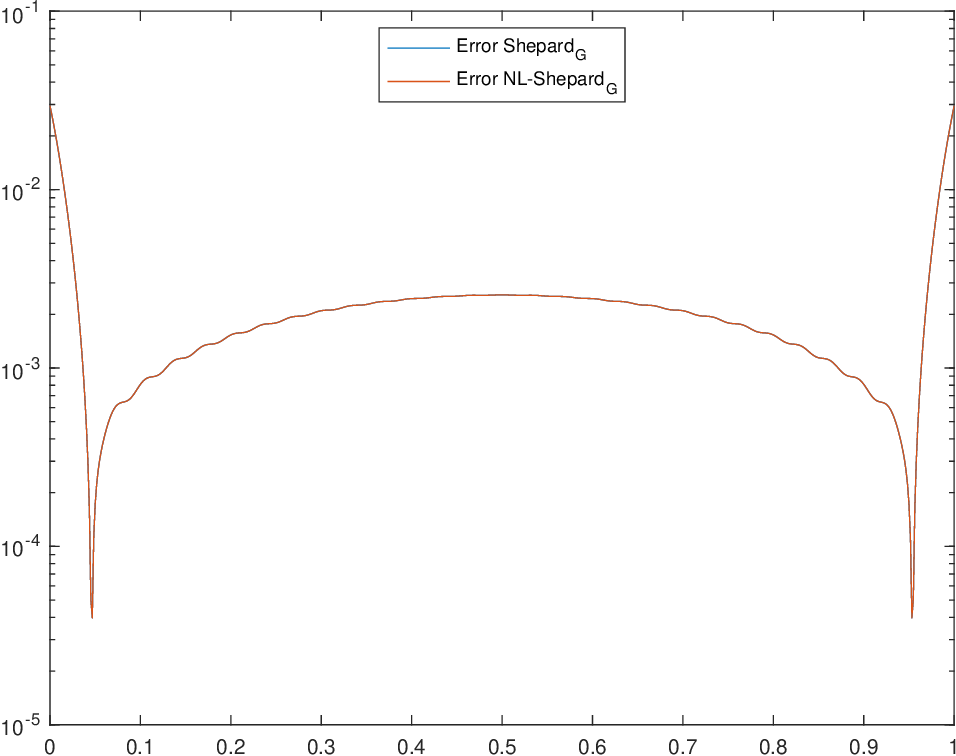}\\
	\includegraphics[width=5.2cm]{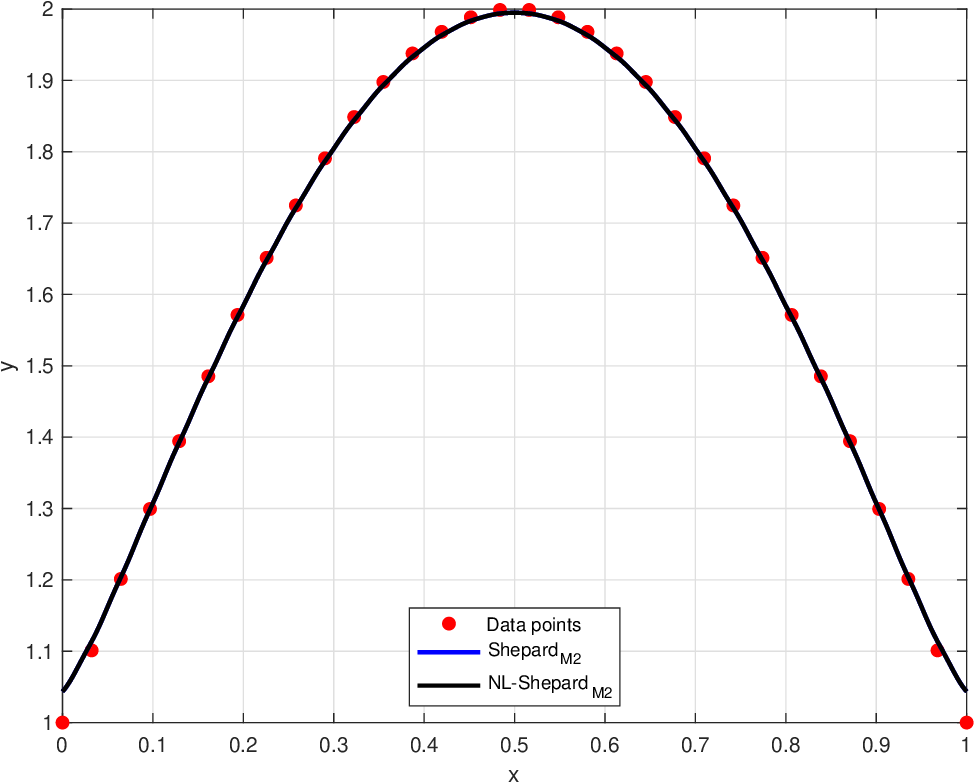} & 	\includegraphics[width=5.2cm]{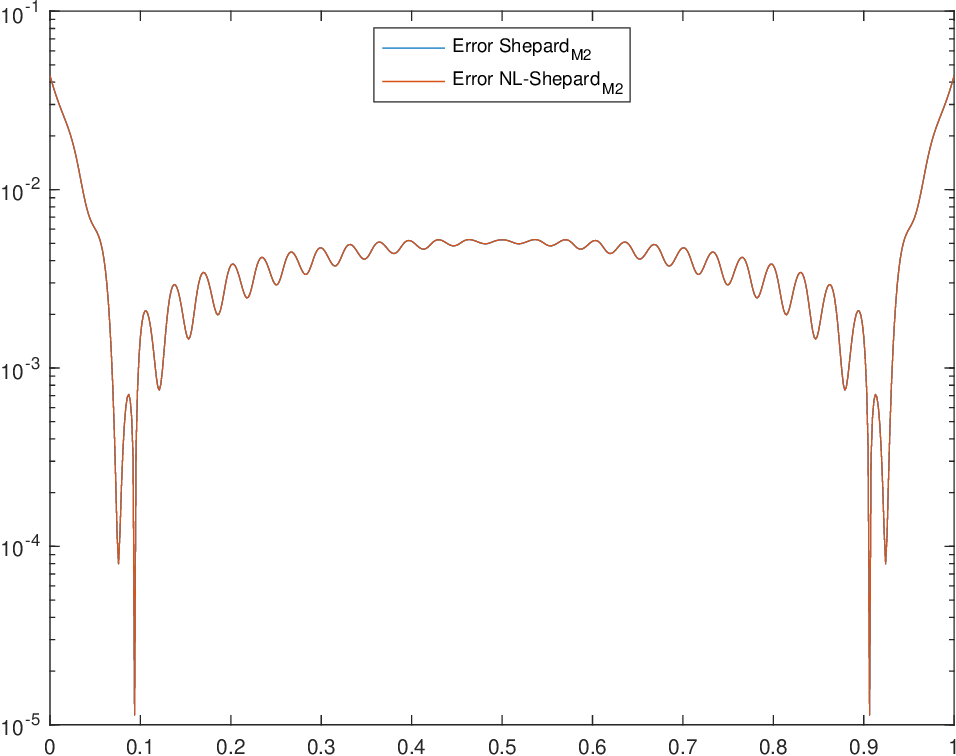}\\
	\includegraphics[width=5.2cm]{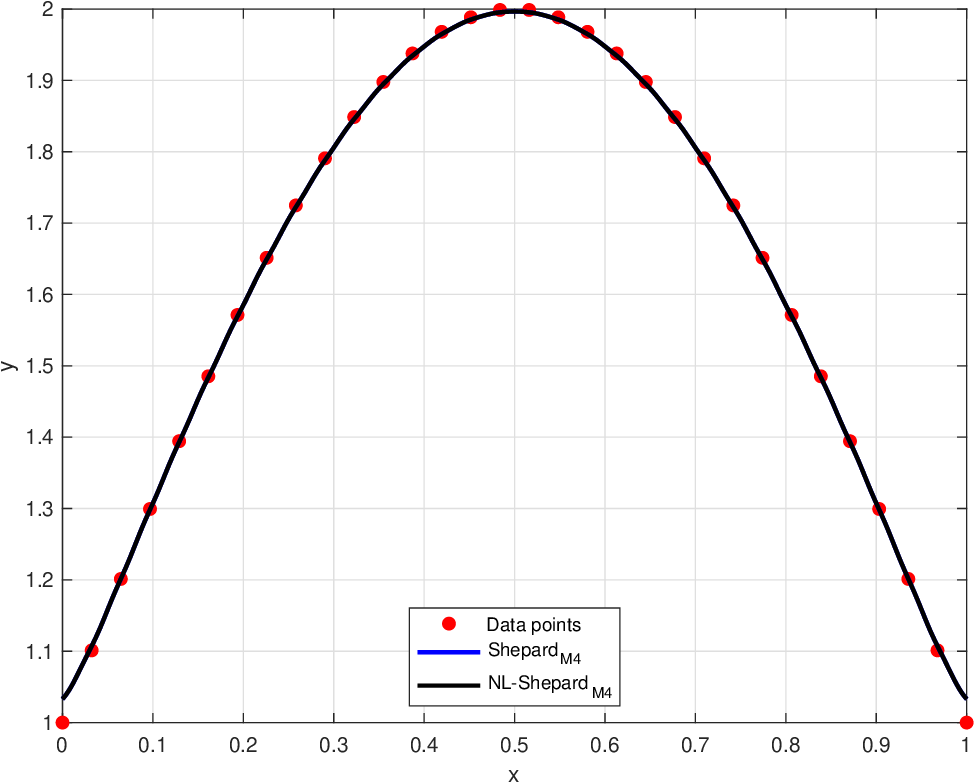} & 	\includegraphics[width=5.2cm]{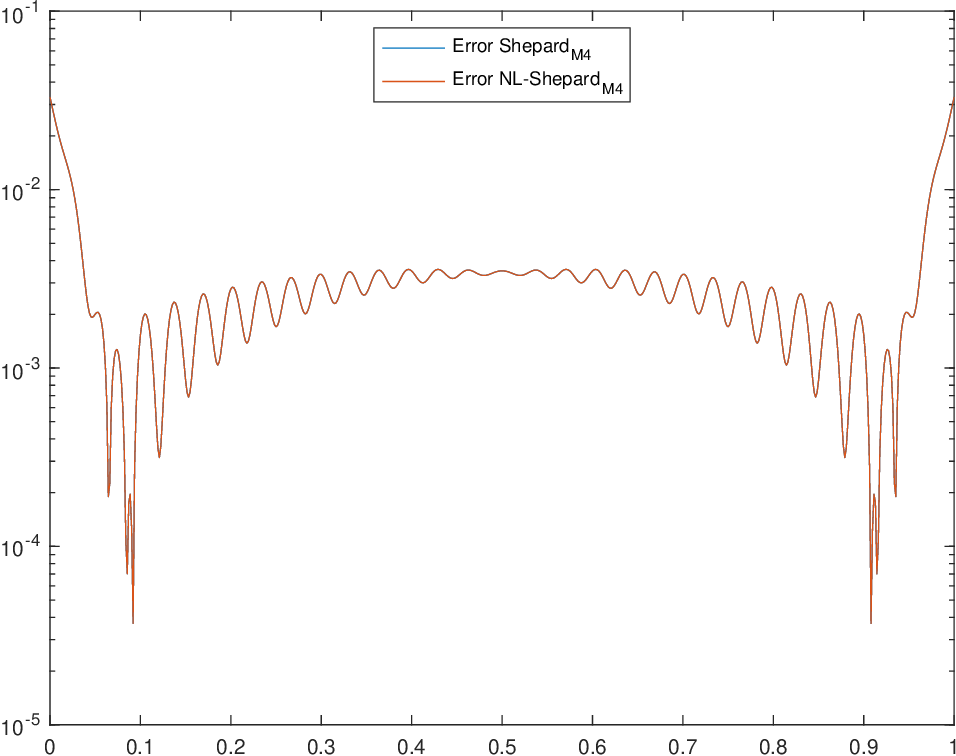}\\
	\includegraphics[width=5.2cm]{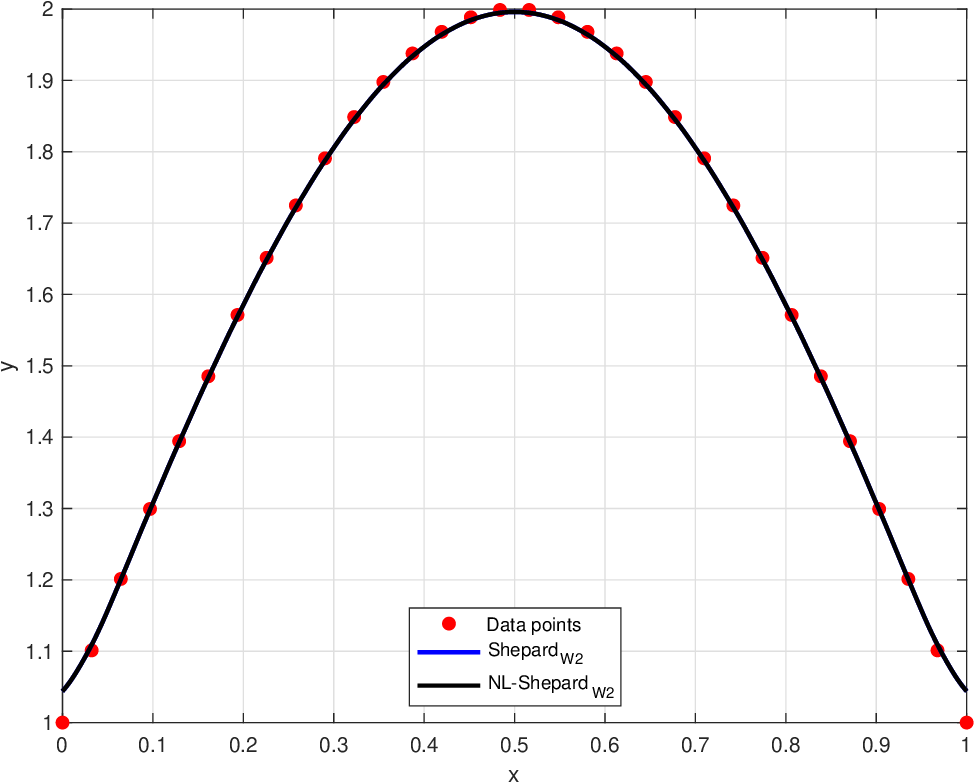} & 	\includegraphics[width=5.2cm]{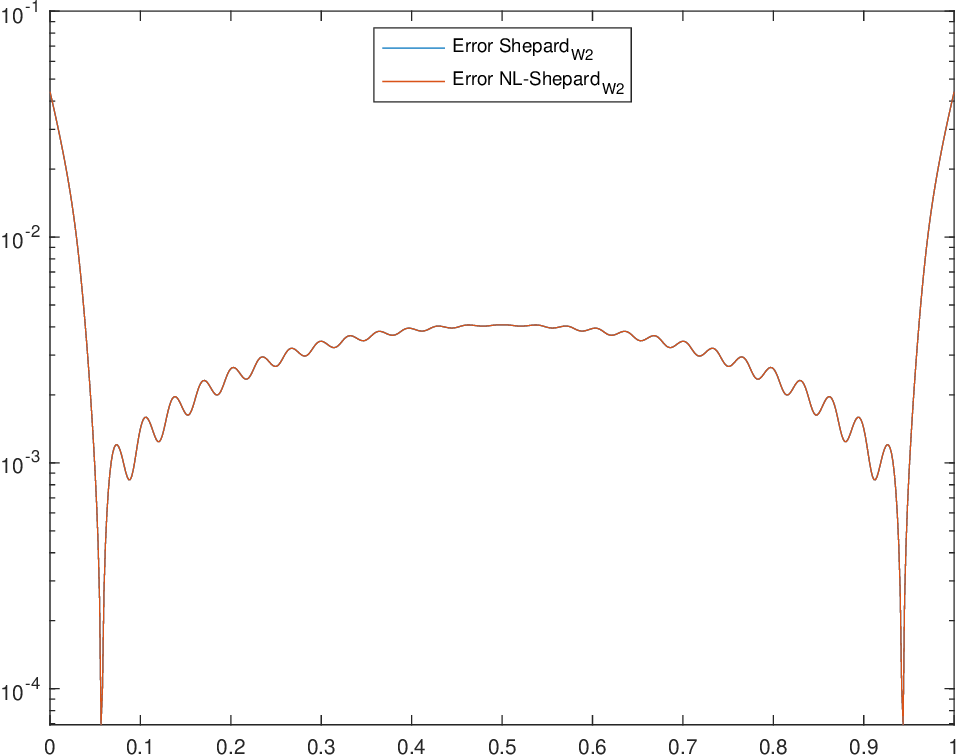}\\
		\includegraphics[width=5.2cm]{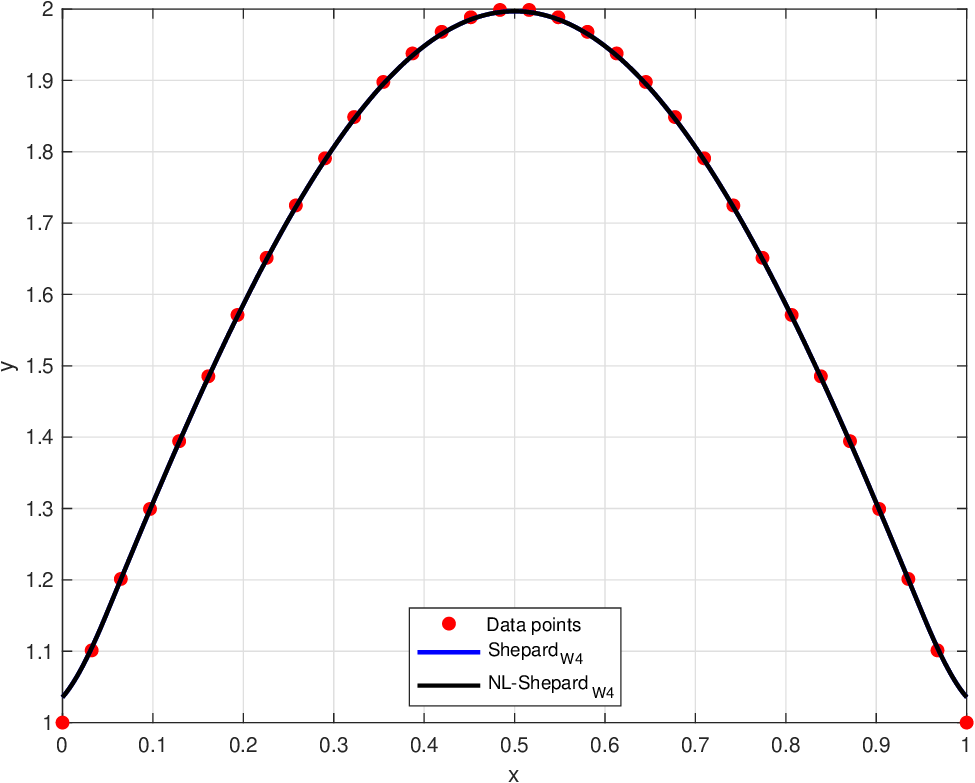} & 	\includegraphics[width=5.2cm]{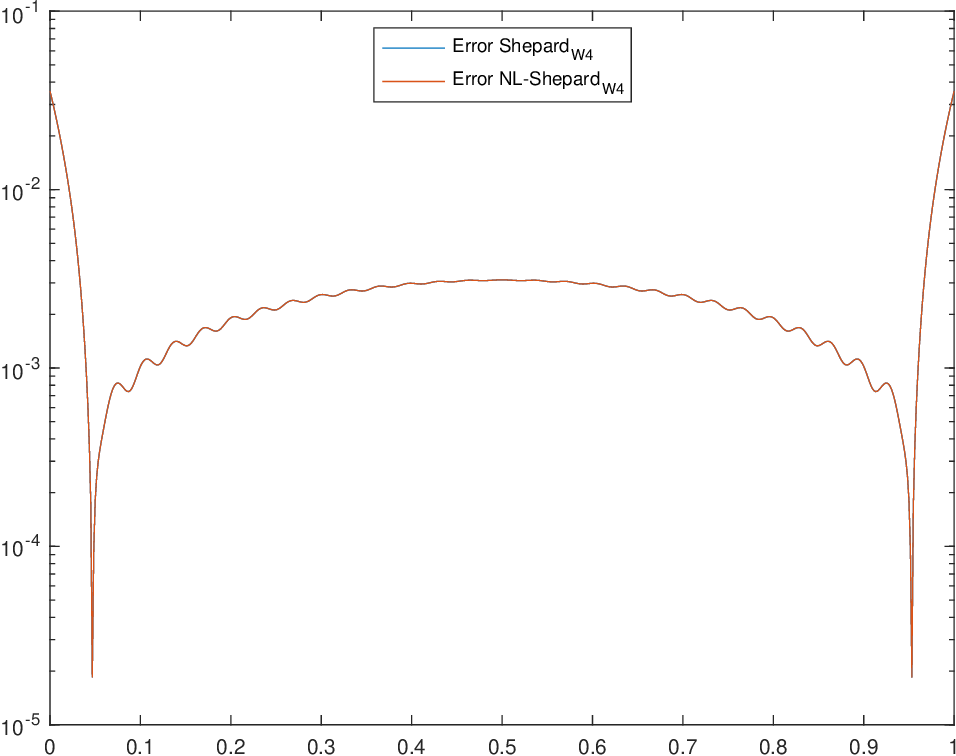}
		\end{tabular}
\end{center}
			\caption{Approximation to the function $f$, Eq. \eqref{funcion_f}, using a uniform grid of $n=32$ points (red dots) and generating 20 evaluation points between each of them. In each plot, the classical Shepard (blue solid line) and data-dependent (black solid line) Shepard algorithms have been used. The first column presents the approximation obtained by the algorithms. The second column presents the point-wise absolute error obtained in a semilogarithmic scale. From top to bottom these algorithms are: Shepard$_{\text{G}}$ and NL-Shepard$_{\text{G}}$, Shepard$_{\text{W2}}$ and  NL-Shepard$_{\text{W2}}$,  Shepard$_{\text{W4}}$ and  NL-Shepard$_{\text{W4}}$,  Shepard$_{\text{M2}}$ and NL-Shepard$_{\text{M2}}$, Shepard$_{\text{M4}}$ and NL-Shepard$_{\text{M4}}$.}
		\label{exp1_sm}
	\end{figure}

	\begin{figure}[htbp!]
\begin{center}
		\begin{tabular}{cc}
	\includegraphics[width=5.2cm]{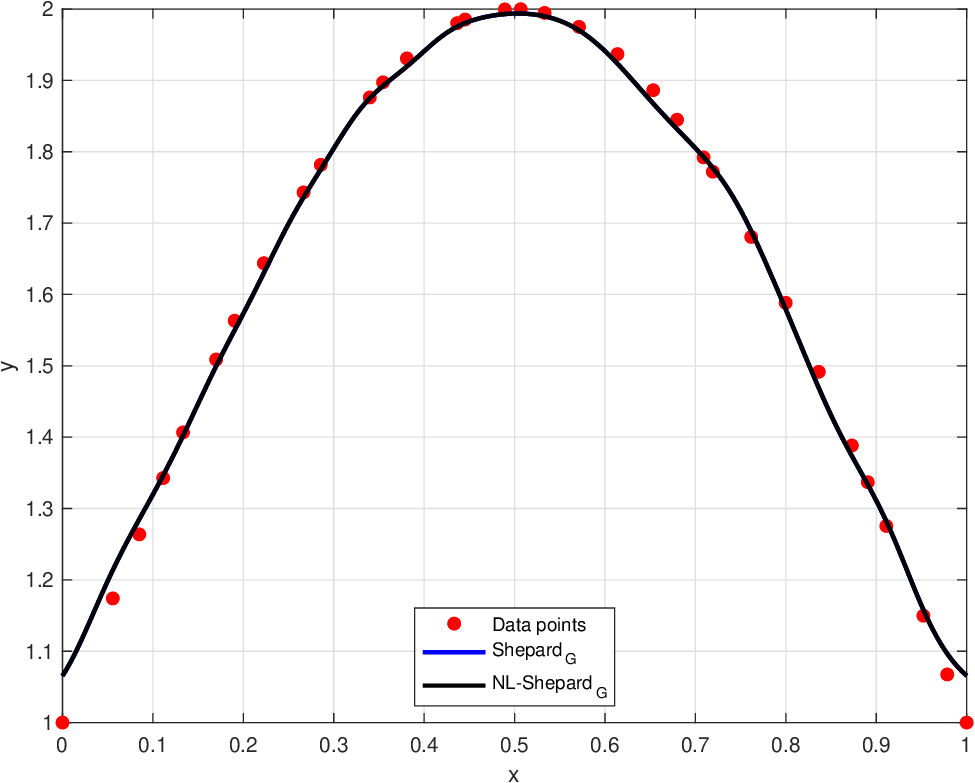} & 	\includegraphics[width=5.2cm]{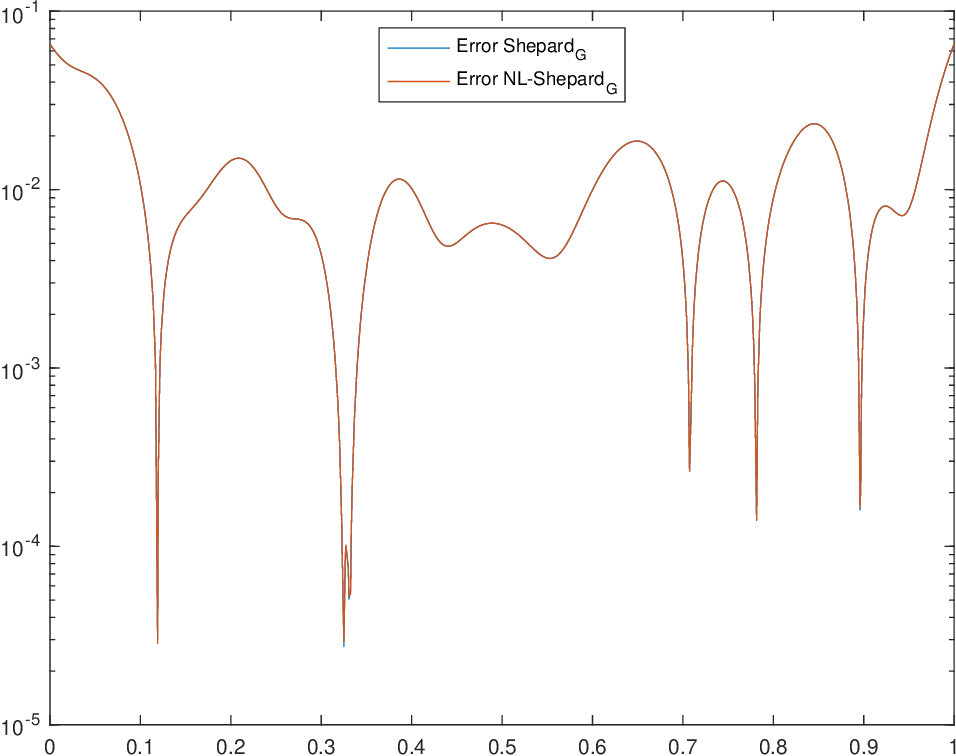}\\
	\includegraphics[width=5.2cm]{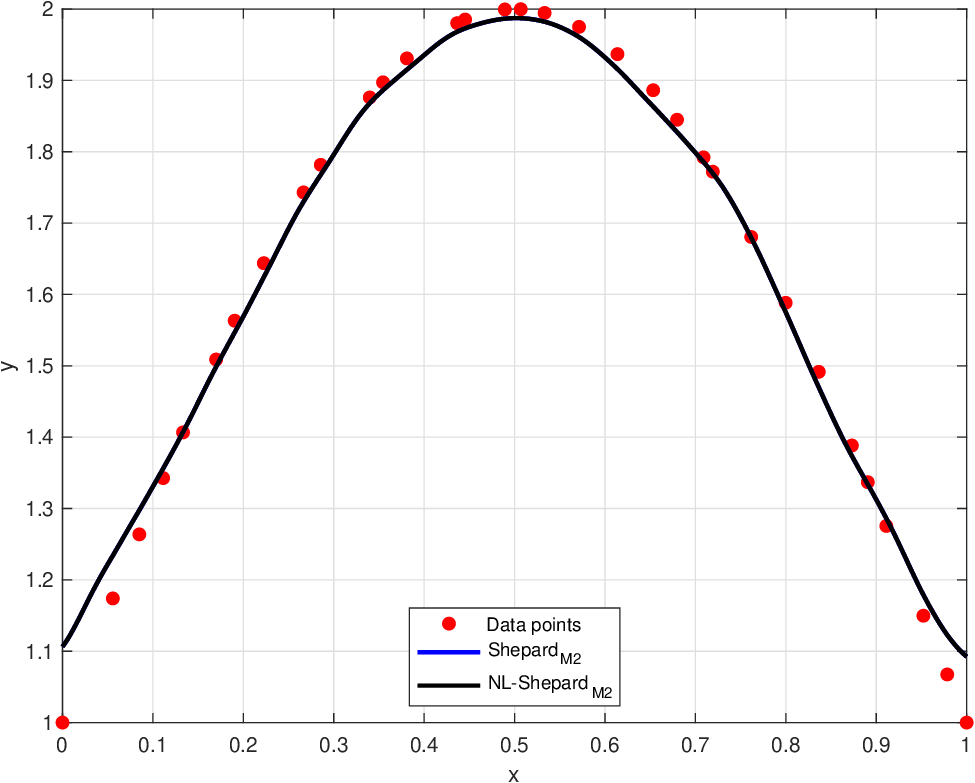} & 	\includegraphics[width=5.2cm]{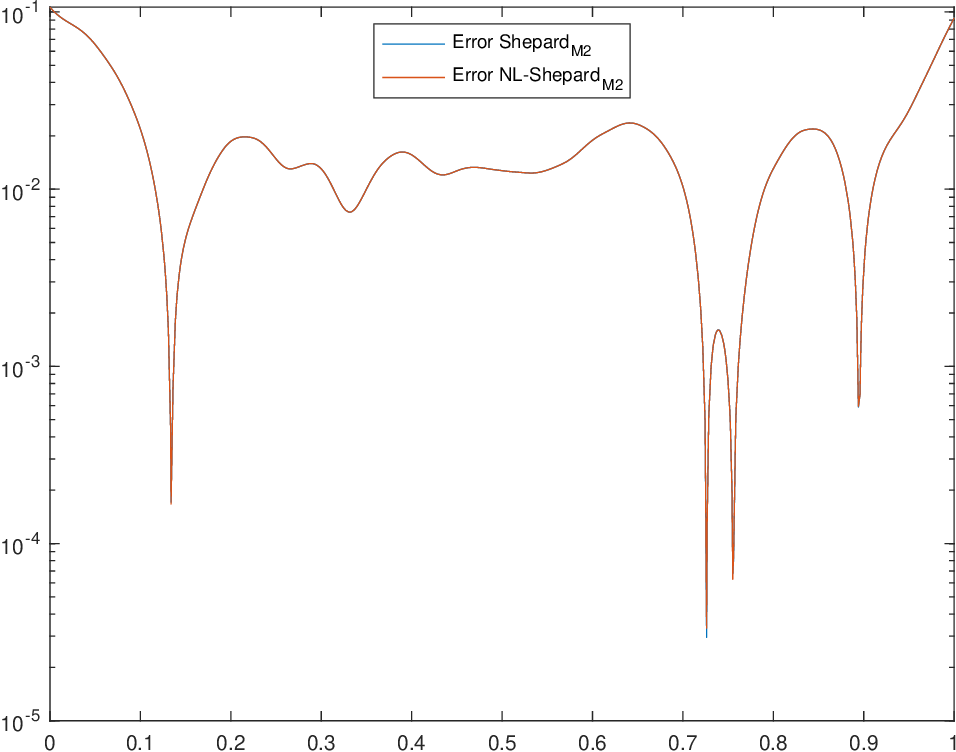}\\
	\includegraphics[width=5.2cm]{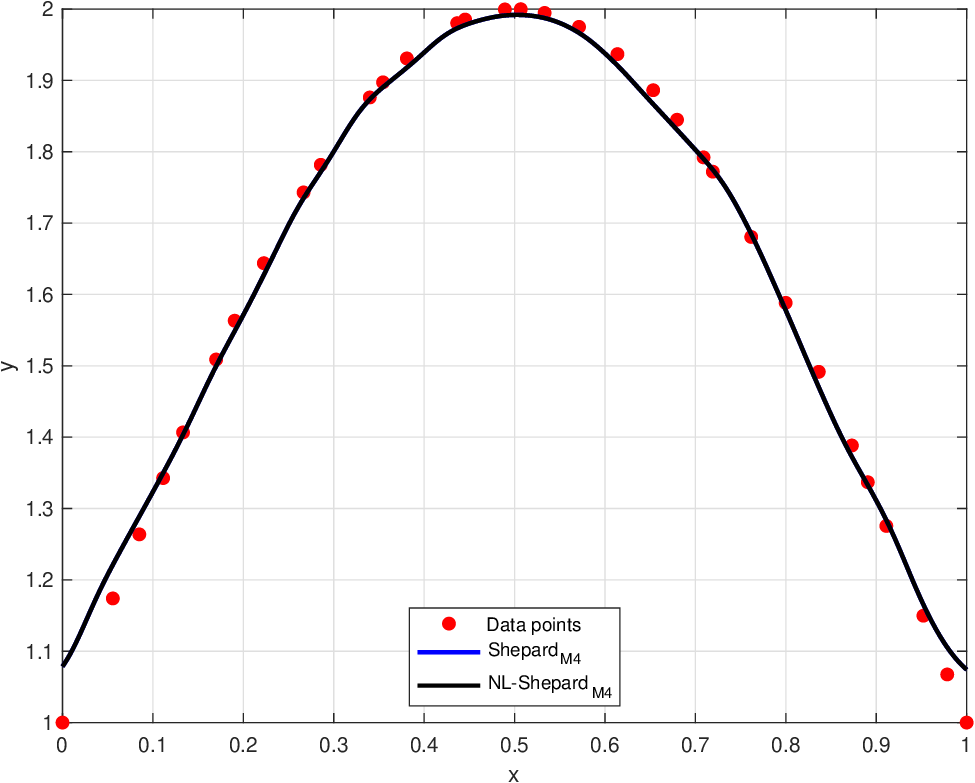} & 	\includegraphics[width=5.2cm]{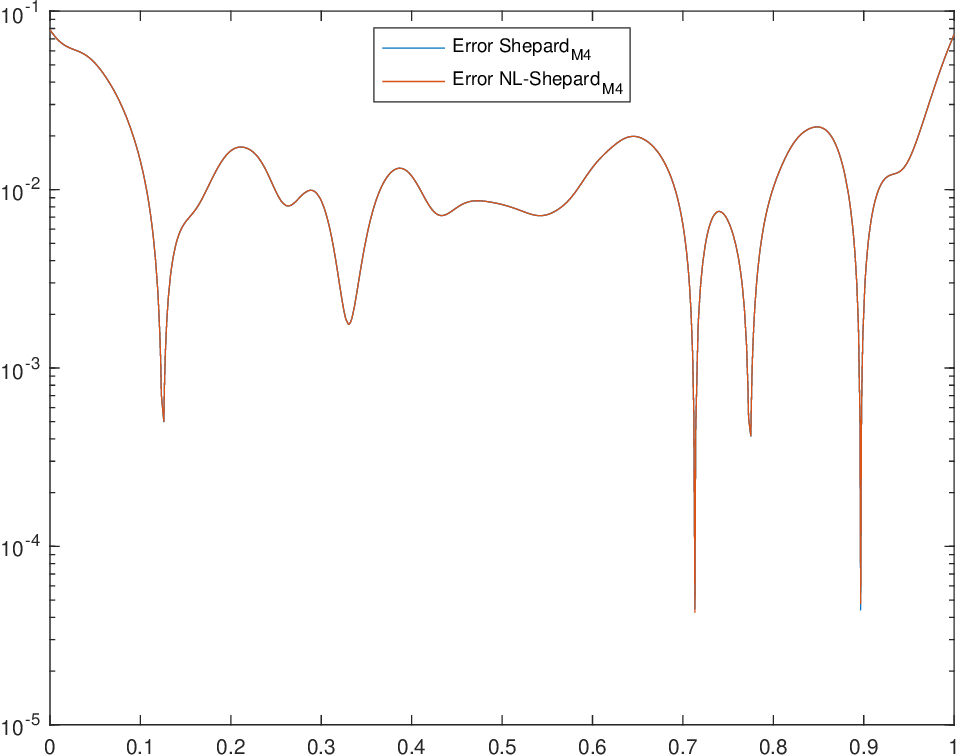}\\
	\includegraphics[width=5.2cm]{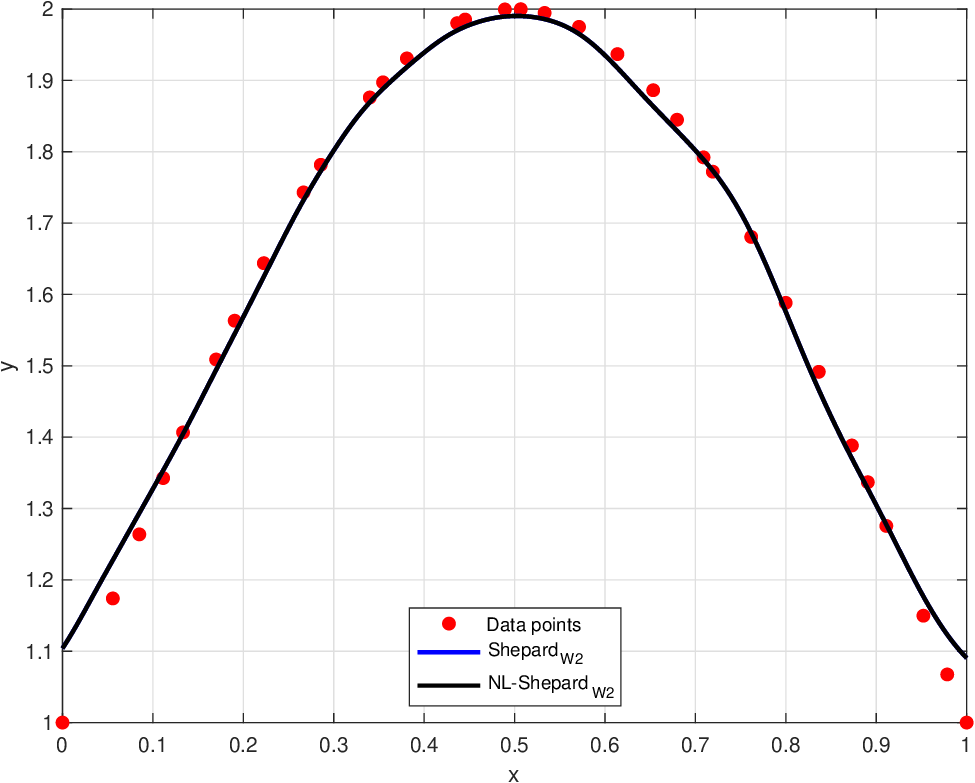} & 	\includegraphics[width=5.2cm]{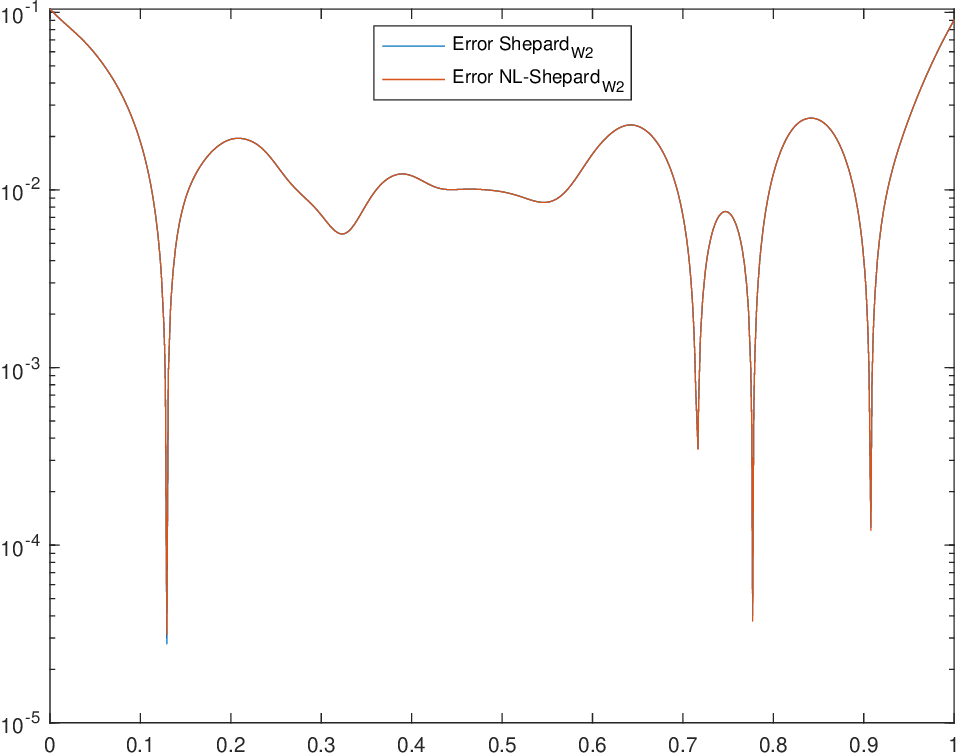}\\
		\includegraphics[width=5.2cm]{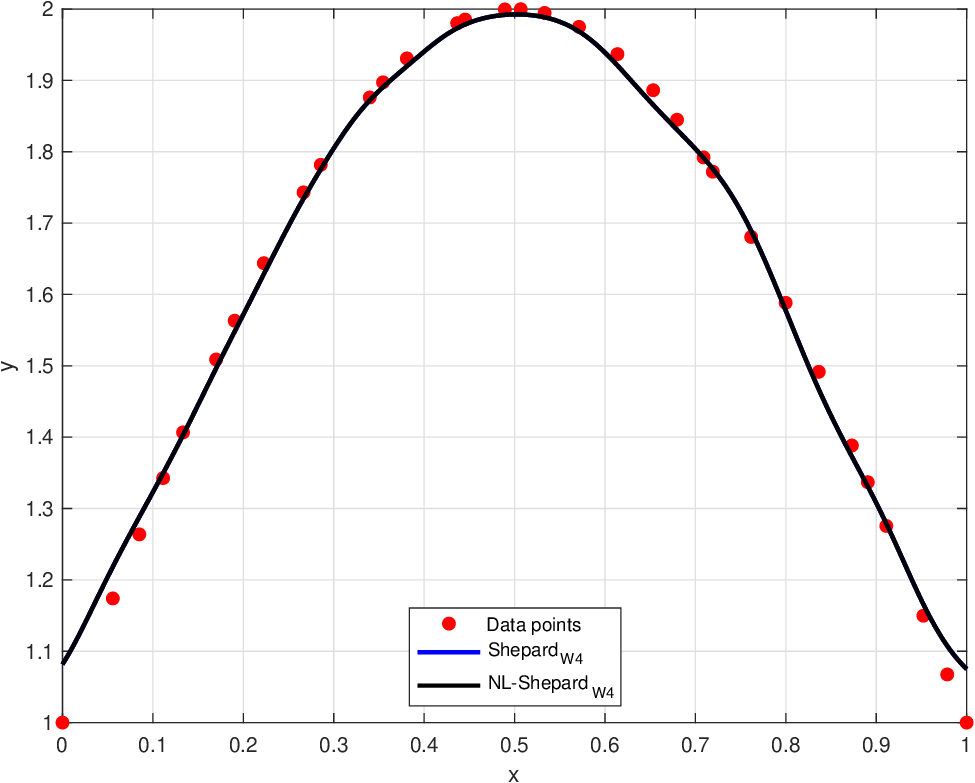} & 	\includegraphics[width=5.2cm]{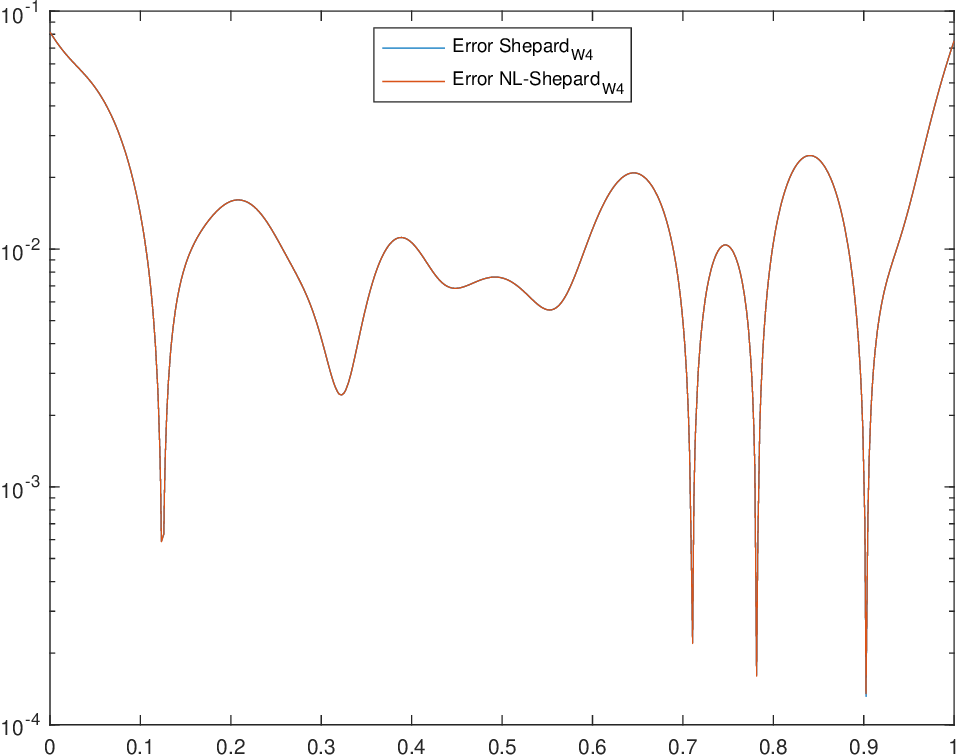}
		\end{tabular}
\end{center}
			\caption{Approximation to the function $f$, Eq. \eqref{funcion_f}, using a grid of $n=32$ Halton points (red dots) and generating 20 Halton evaluation points between each of them. In each plot, the classical Shepard (blue solid line) and data-dependent (black solid line) Shepard algorithms have been used. The first column presents the approximation obtained by the algorithms. The second column presents the point-wise absolute error obtained in a semilogarithmic scale. From top to bottom these algorithms are: Shepard$_{\text{G}}$ and NL-Shepard$_{\text{G}}$, Shepard$_{\text{W2}}$ and  NL-Shepard$_{\text{W2}}$,  Shepard$_{\text{W4}}$ and  NL-Shepard$_{\text{W4}}$,  Shepard$_{\text{M2}}$ and NL-Shepard$_{\text{M2}}$, Shepard$_{\text{M4}}$ and NL-Shepard$_{\text{M4}}$.}\label{exp2_sm}
	\end{figure}

\subsection{Mitigation of oscillations near jump discontinuities for univariate data}

In this subsection, we study the behavior of Shepard-type interpolation when the target function exhibits a discontinuity. The test function is given by
\begin{equation}\label{funciong}
g(x)=
\begin{cases}
\sin(\pi x), & x \le 2/3,\\[2mm]
1 - \sin(\pi x), & x > 2/3,
\end{cases}
\end{equation}
which is smooth on each subinterval but presents a jump at $x=2/3$. The approximations are computed on the whole interval $[0,1]$. As initial data sites we select $N = 32$ nodes, either equally spaced or generated via Halton points, as in the previous subsection. Between any two consecutive data sites we introduce 20 auxiliary evaluation points.

As in the previous section, the shape parameter $\varepsilon$ controlling the Shepard weights is chosen so that the \emph{classical} Shepard scheme yields a reasonable approximation. In the experiments based on uniform grids and on Halton points (see Figures~\ref{exp1} and~\ref{exp2}), we use the parameter values specified in (\ref{vareps}). The data-dependent Shepard formulations employ the same initial shape parameter, which is subsequently modified following the adaptive procedure introduced in the preceding sections, relying on local smoothness indicators. The interpolation centers coincide again with the original data sites. As before, in the adaptive method we set $c = 10^{-16}$ in (\ref{tildegamma_new}) and $C = 1$, $t = 1$ in (\ref{gx}). The Halton nodes are identical to those used in previous experiments.

Figure~\ref{exp1} shows the results obtained with both the \emph{classical} (i.e., standard) Shepard interpolation and its \emph{data-dependent} adaptive counterpart for all kernels listed in Table~\ref{tabla1nucleos}, using uniformly spaced data. Again, the data sites are highlighted with red dots. The output of the classical Shepard method is depicted with a blue curve (notation following Table~\ref{tabla1nucleos}), while the data-dependent method is represented by a black curve.

Across all kernels, we can see that the data-dependent strategy decreases the smearing belt around the discontinuity that the classical Shepard interpolant exhibits near the jump. In the adaptive approach, a mild smoothing of the discontinuity is visible. 

Figure~\ref{exp2} presents the analogue of the previous experiment performed with non-uniform Halton data. The conclusions are consistent with the uniformly spaced case: the data-dependent Shepard scheme effectively reduces the smearing that appears in the classical method when the target function is discontinuous. This improvement is robust across all kernels in Table~\ref{tabla1nucleos}. As before, a small amount of smoothing at the discontinuity is unavoidable but limited.

	\begin{figure}[htbp!]
\begin{center}
		\begin{tabular}{cc}
	\includegraphics[width=5.2cm]{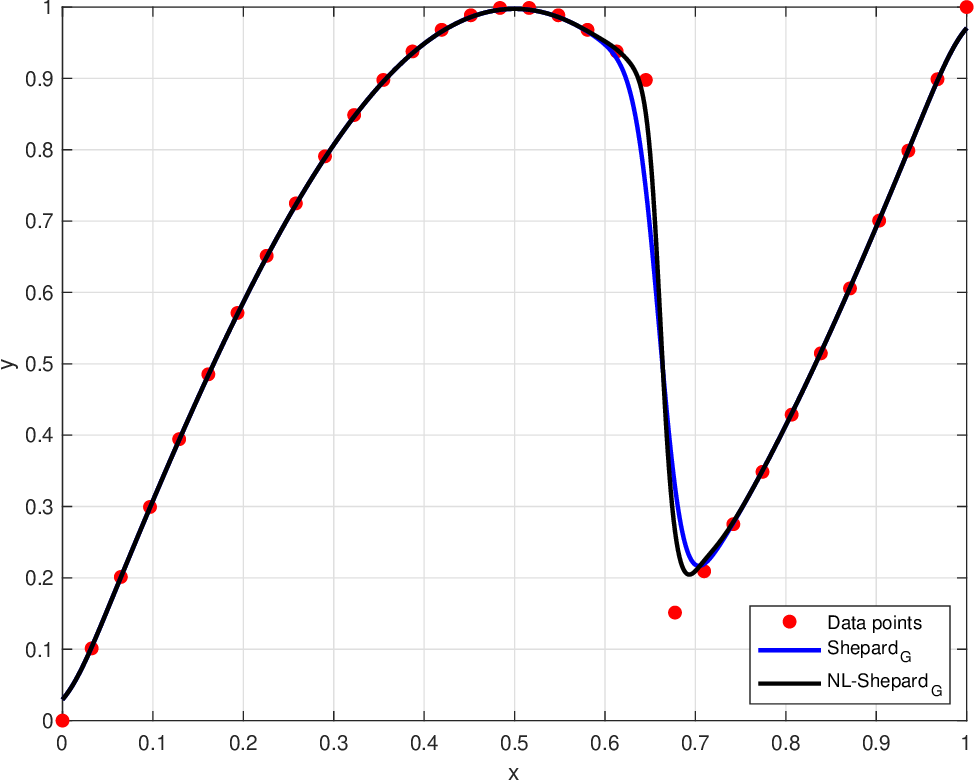} & 	\includegraphics[width=5.2cm]{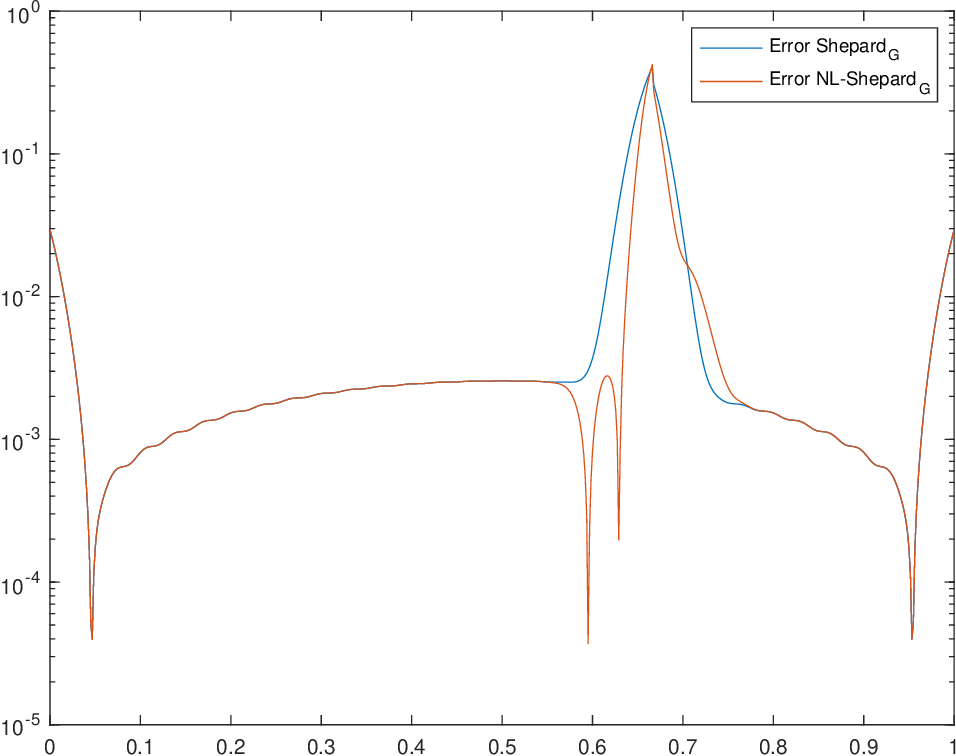}\\
	\includegraphics[width=5.2cm]{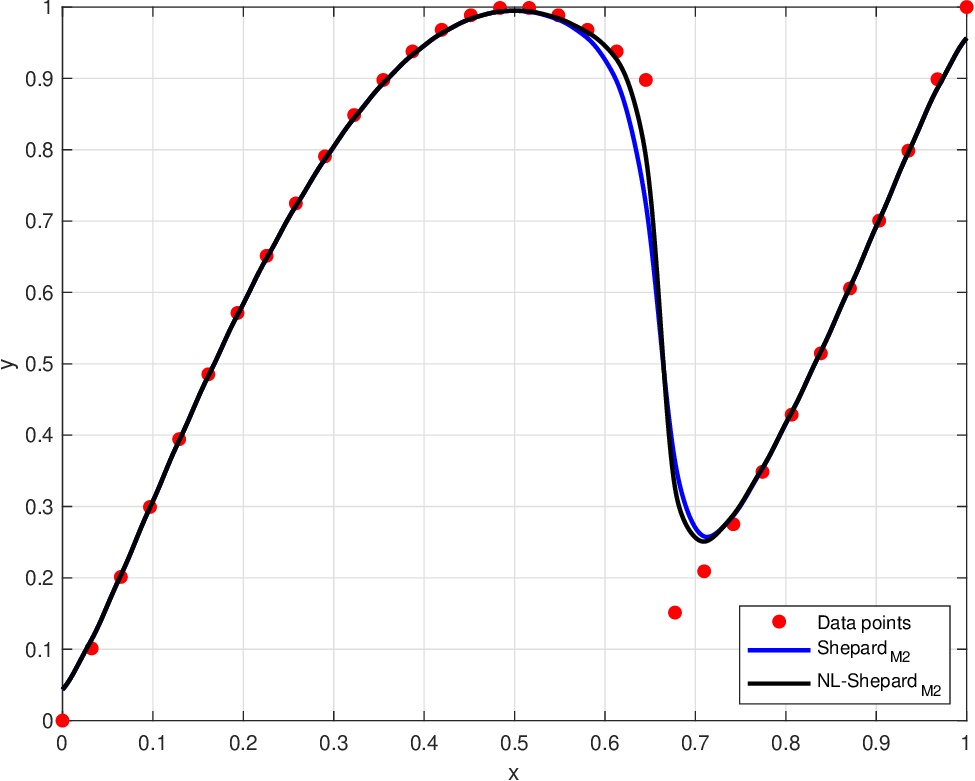} & 	\includegraphics[width=5.2cm]{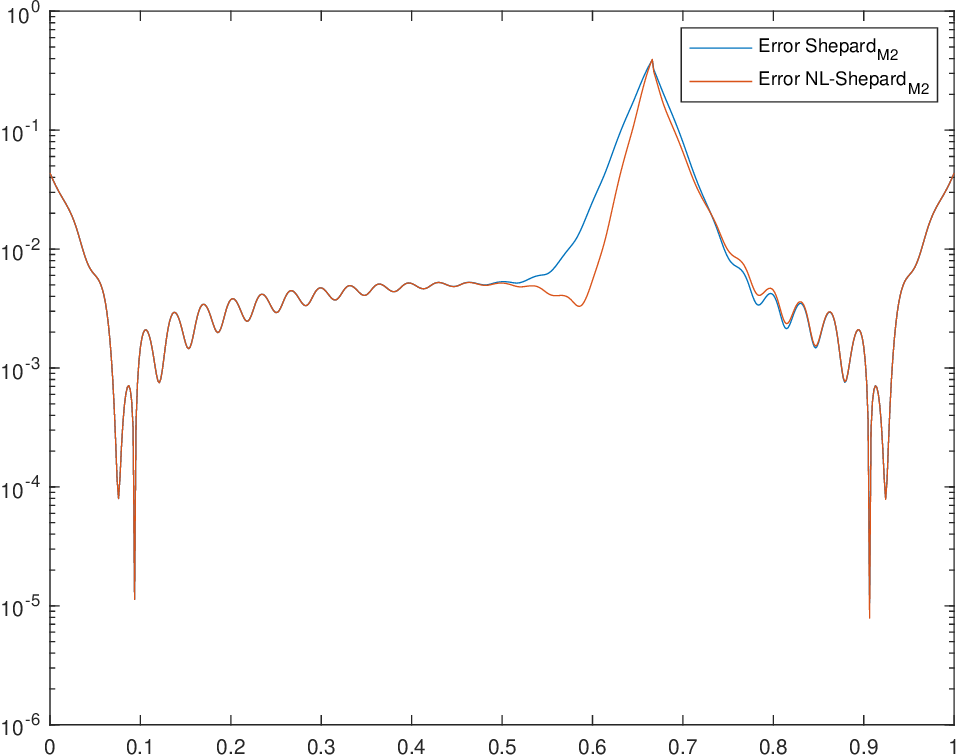}\\
	\includegraphics[width=5.2cm]{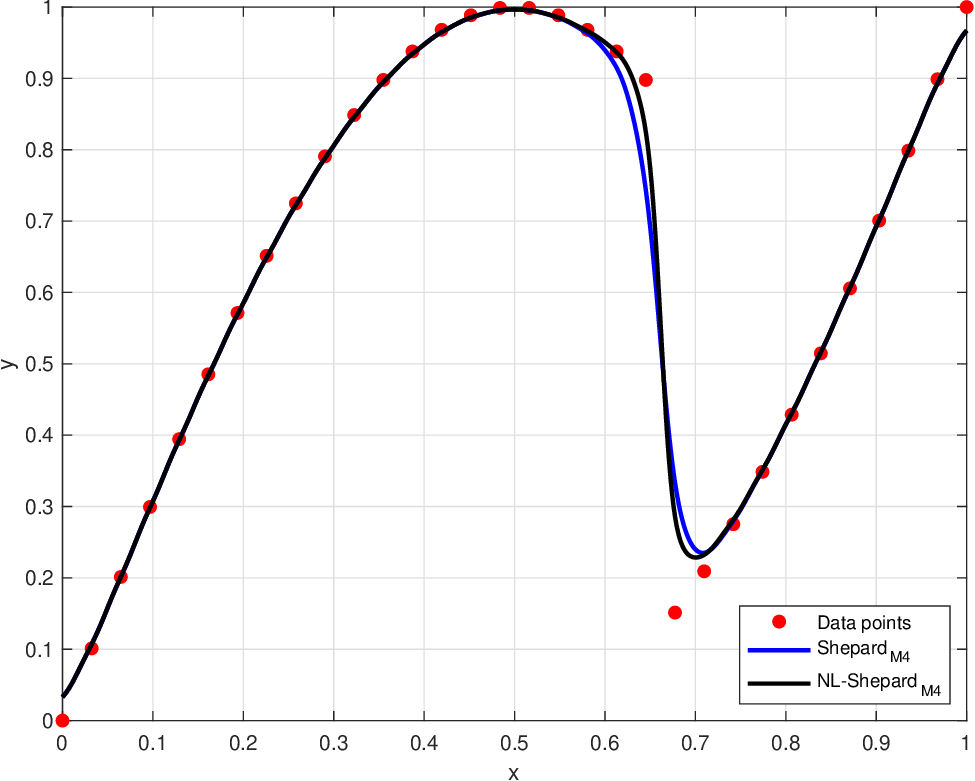} & 	\includegraphics[width=5.2cm]{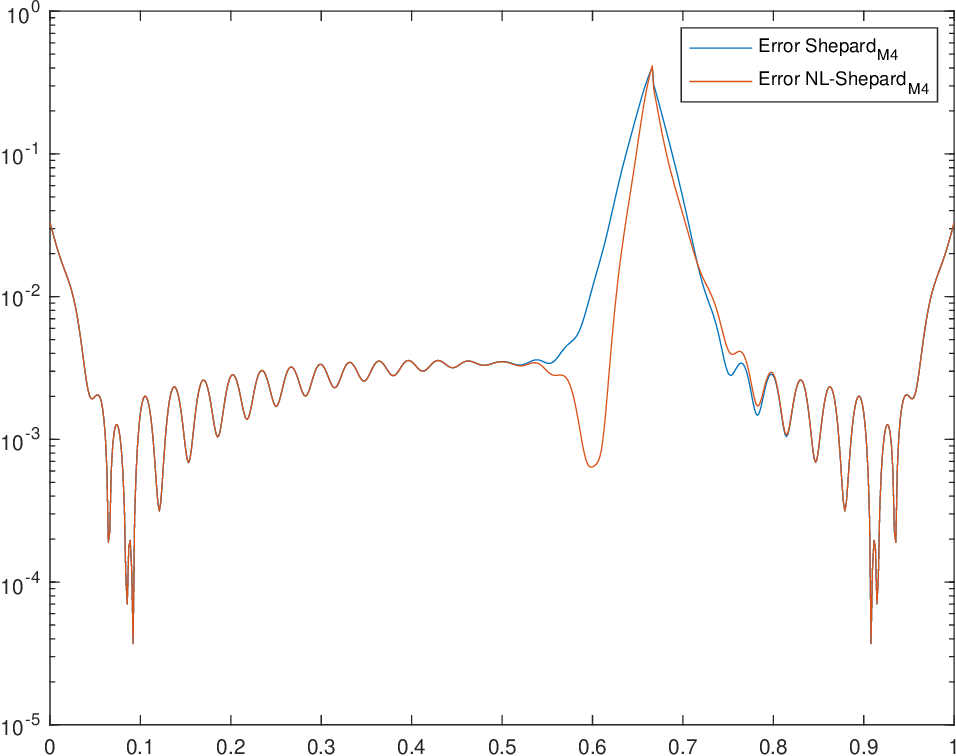}\\
	\includegraphics[width=5.2cm]{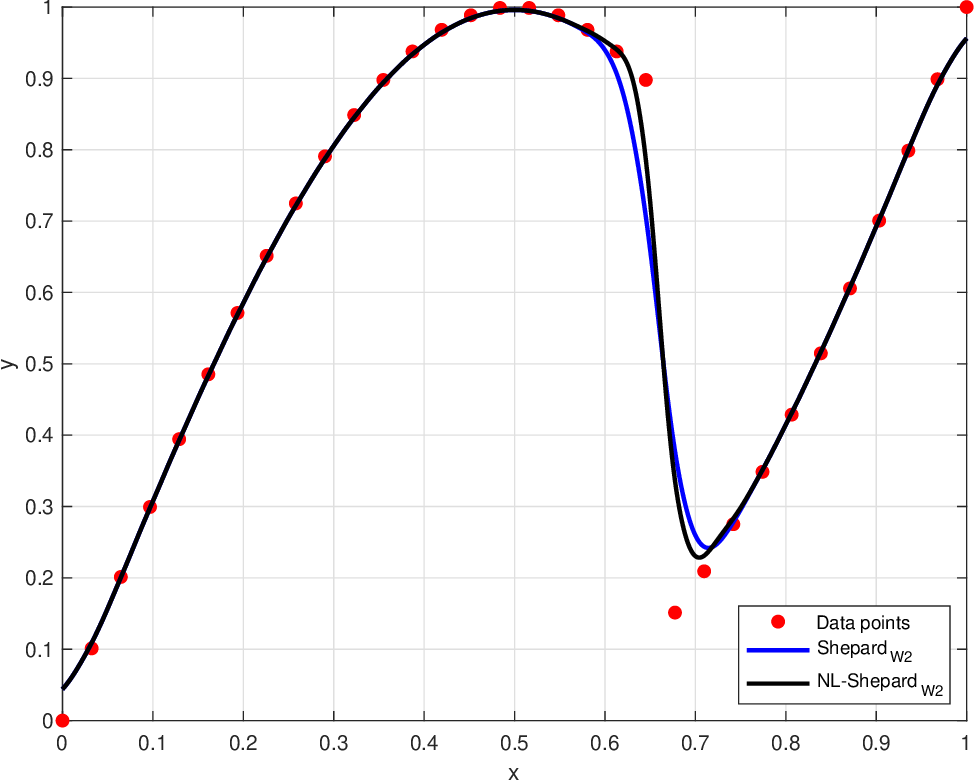} & 	\includegraphics[width=5.2cm]{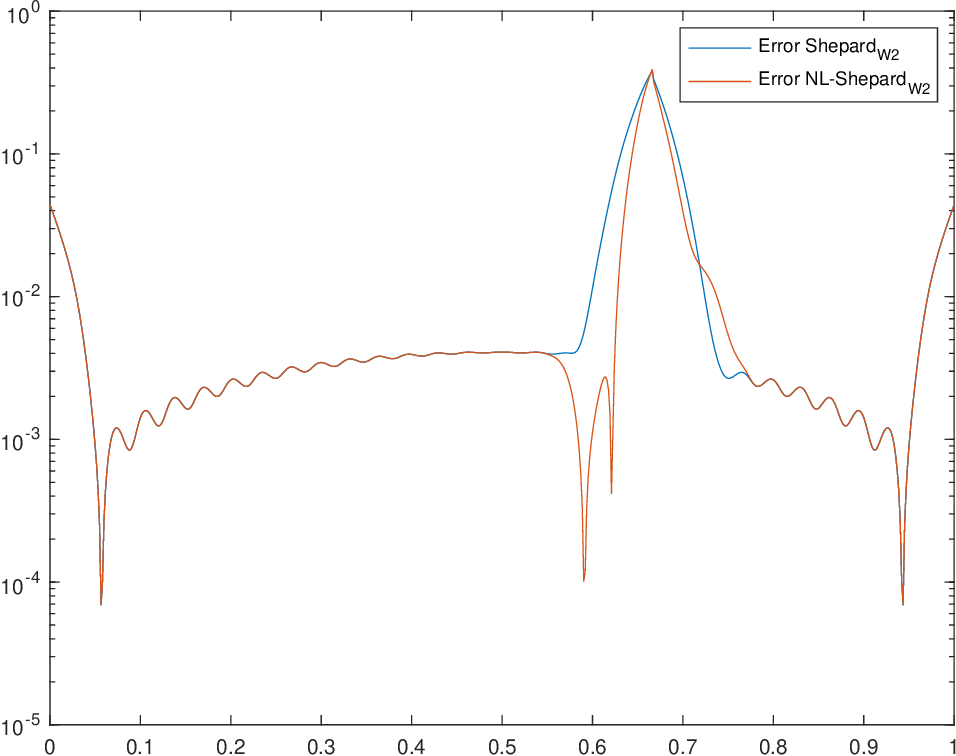}\\
		\includegraphics[width=5.2cm]{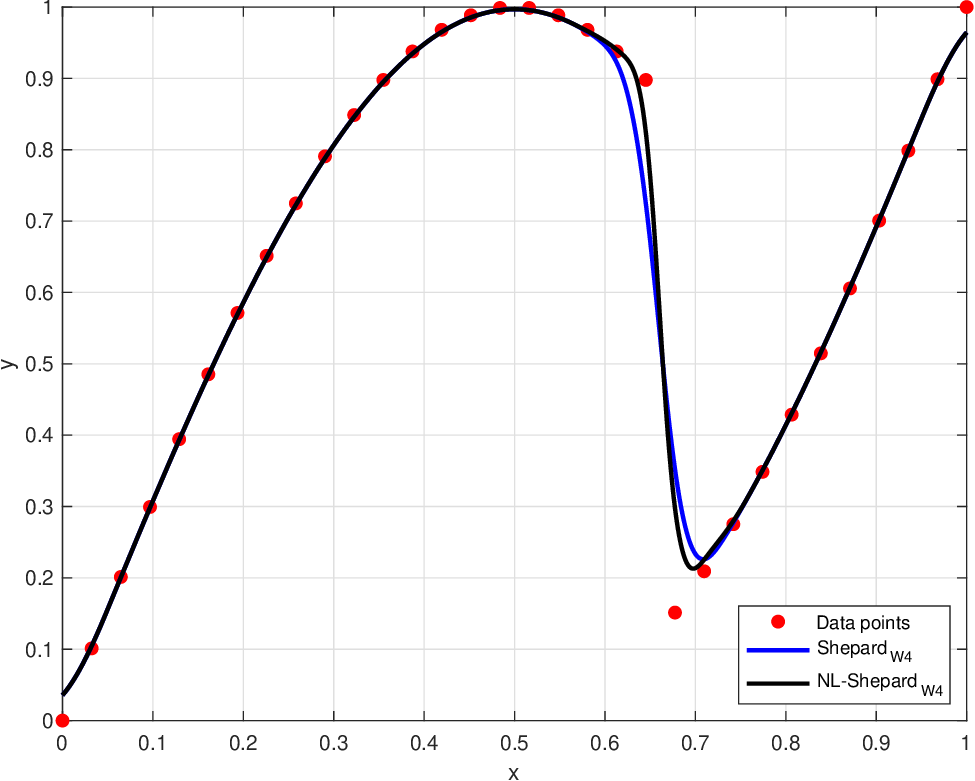} & 	\includegraphics[width=5.2cm]{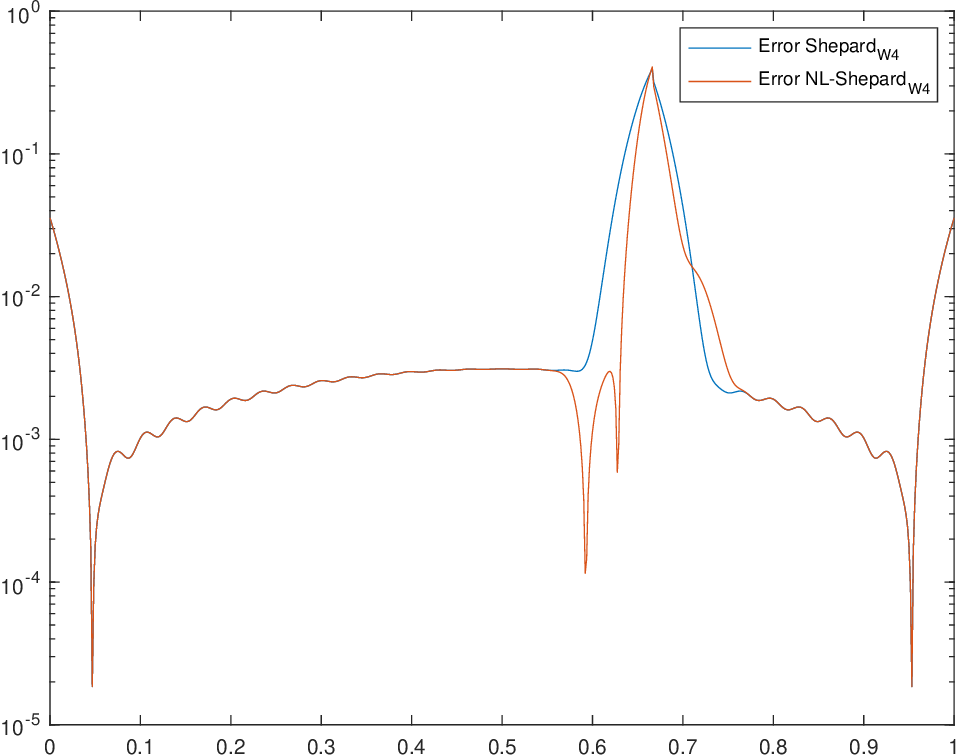}
		\end{tabular}
\end{center}
			\caption{Approximation to the function $g$ (red solid dots), Eq. \eqref{funciong}, using a uniform grid of $n=32$ points (red dots) and generating 20 evaluation points between each of them. In each plot, the classical (blue solid line) and data-dependent (black solid line) Shepard algorithms have been used. The first column presents the approximation obtained by the algorithms. The second column presents the error obtained in a semilogarithmic scale. In this case, the blue line is for the classical Shepard and the red line for the data-dependent. From top to bottom these algorithms are: Shepard$_{\text{G}}$ and Shepard$_{\text{G}}$, Shepard$_{\text{W2}}$ and  NL-Shepard$_{\text{W2}}$,  Shepard$_{\text{W4}}$ and  NL-Shepard$_{\text{W4}}$,  Shepard$_{\text{M2}}$ and NL-Shepard$_{\text{M2}}$, Shepard$_{\text{M4}}$ and NL-Shepard$_{\text{M4}}$.}
		\label{exp1}
	\end{figure}

	\begin{figure}[htbp!]
\begin{center}
		\begin{tabular}{cc}
	\includegraphics[width=5.2cm]{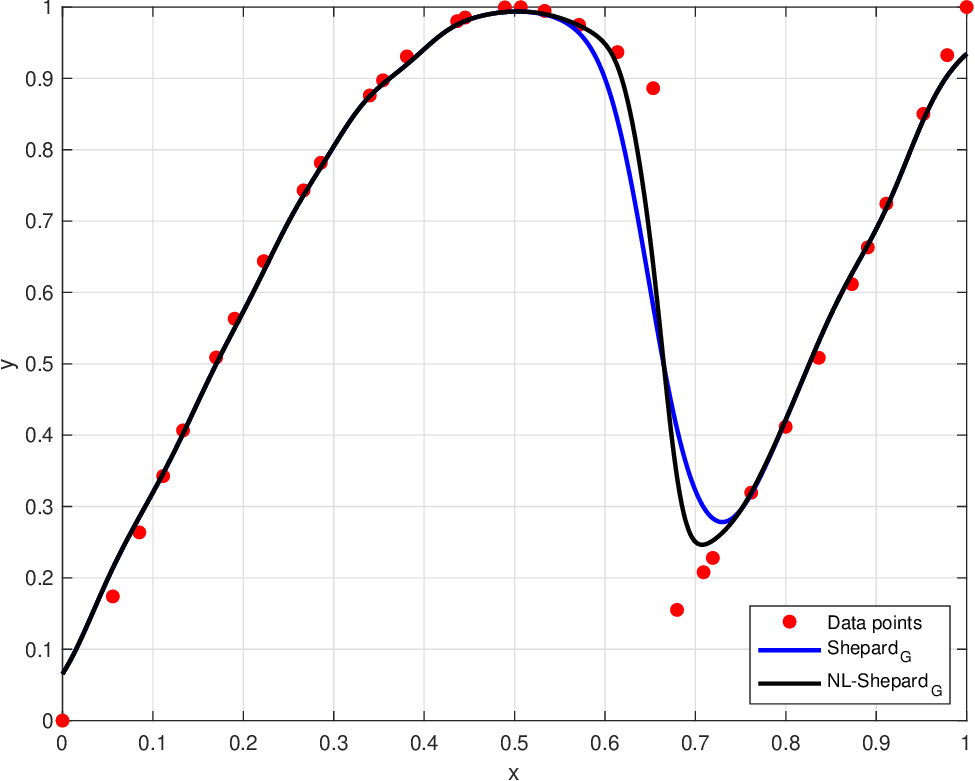} & 	\includegraphics[width=5.2cm]{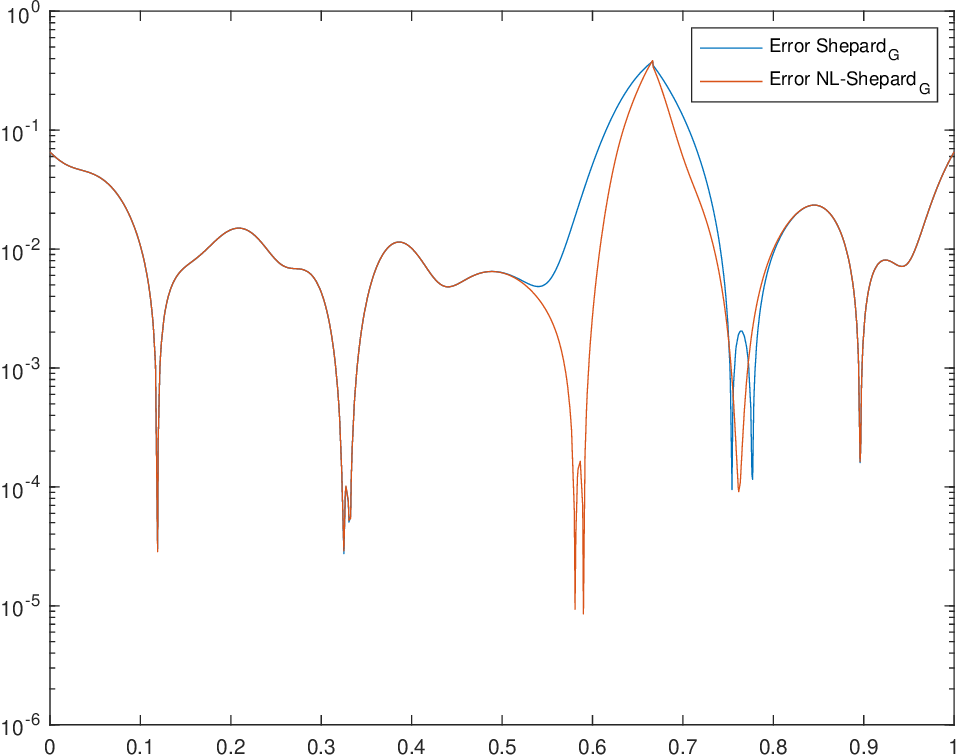}\\
	\includegraphics[width=5.2cm]{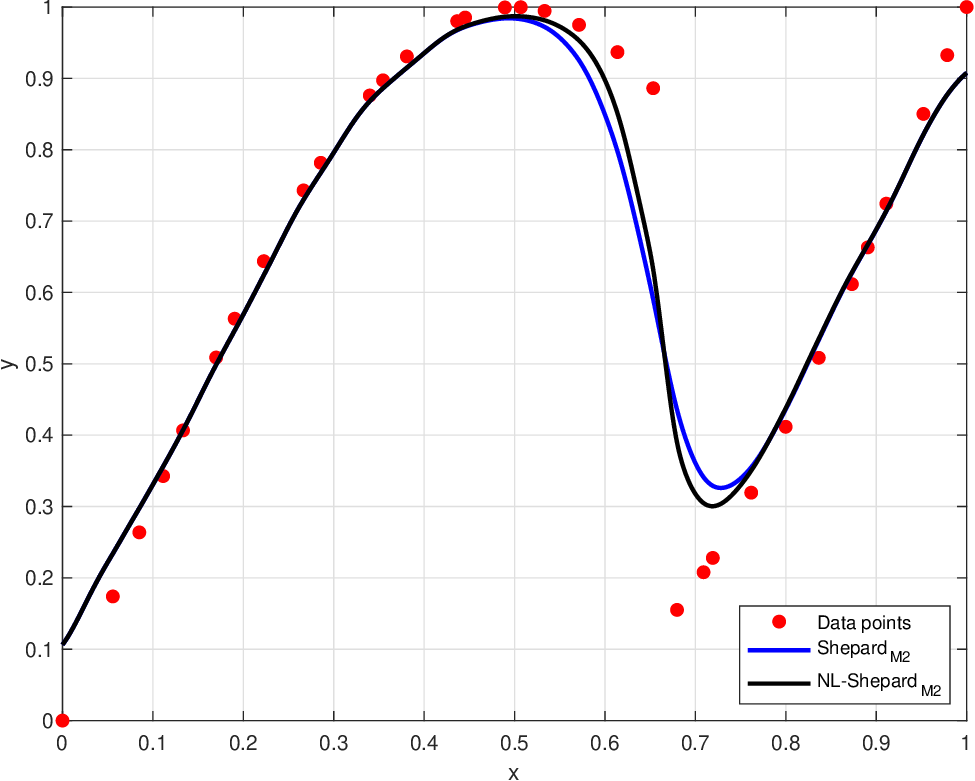} & 	\includegraphics[width=5.2cm]{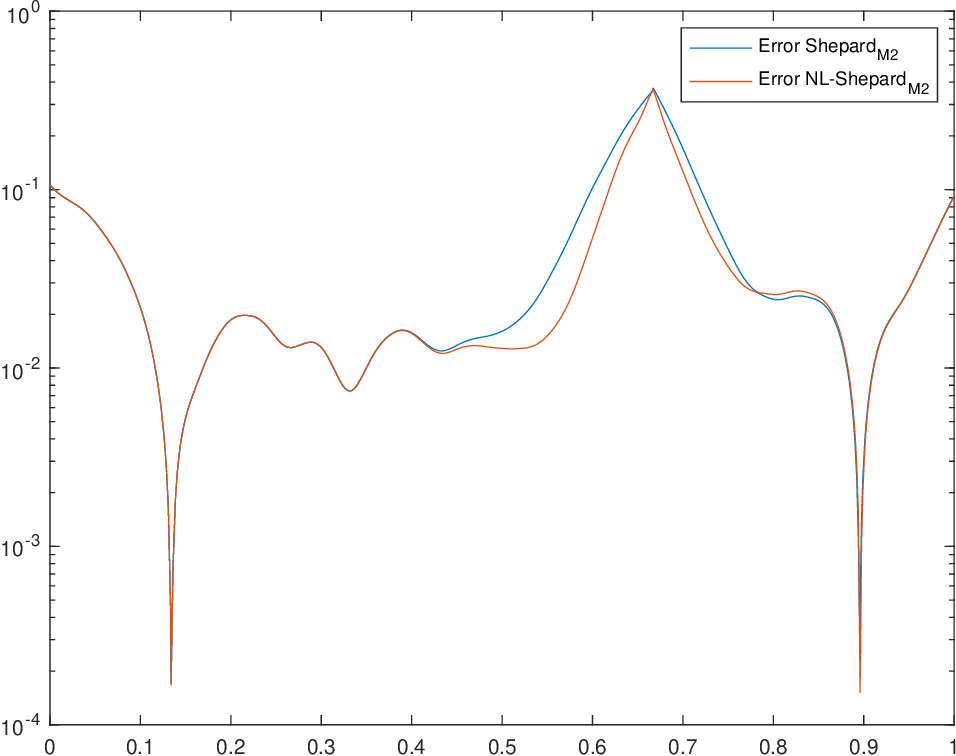}\\
	\includegraphics[width=5.2cm]{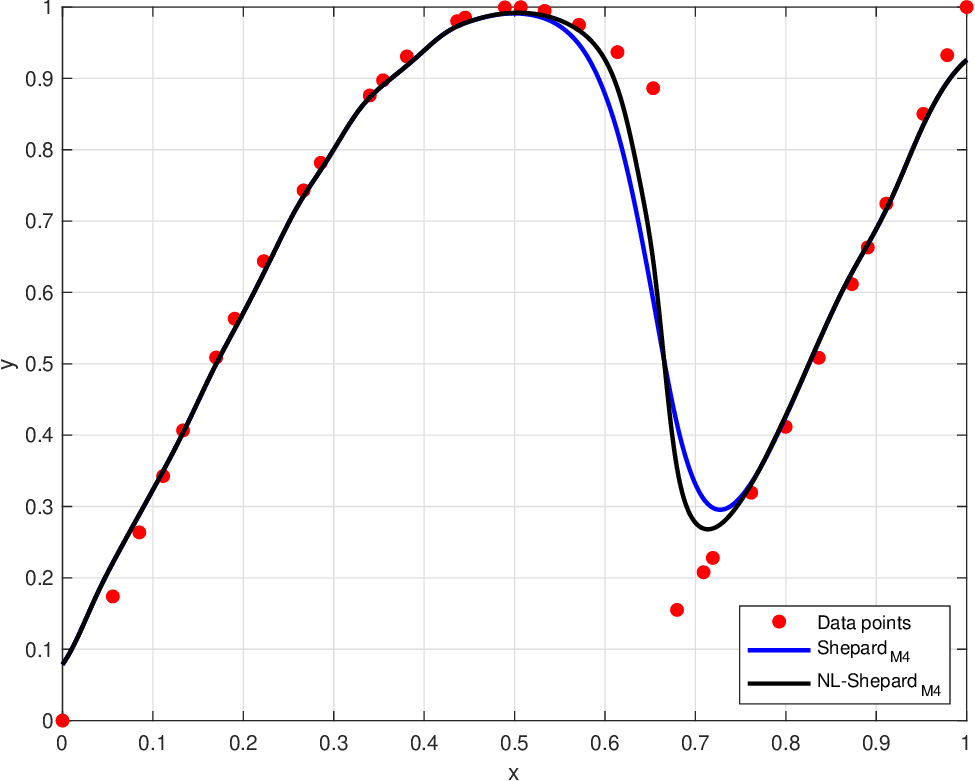} & 	\includegraphics[width=5.2cm]{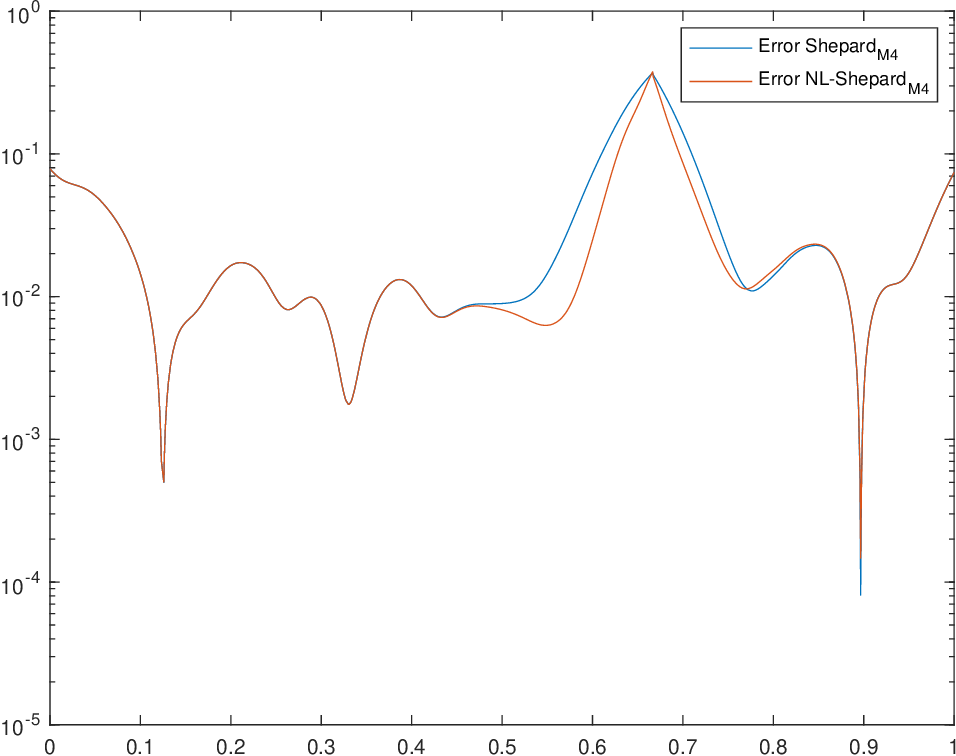}\\
	\includegraphics[width=5.2cm]{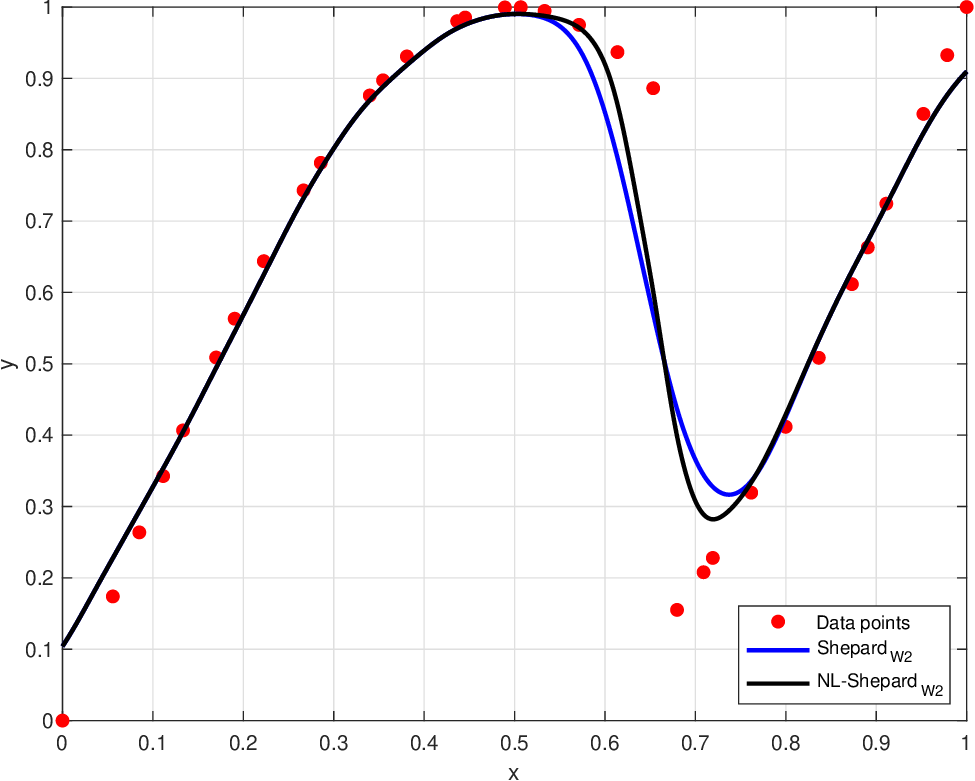} & 	\includegraphics[width=5.2cm]{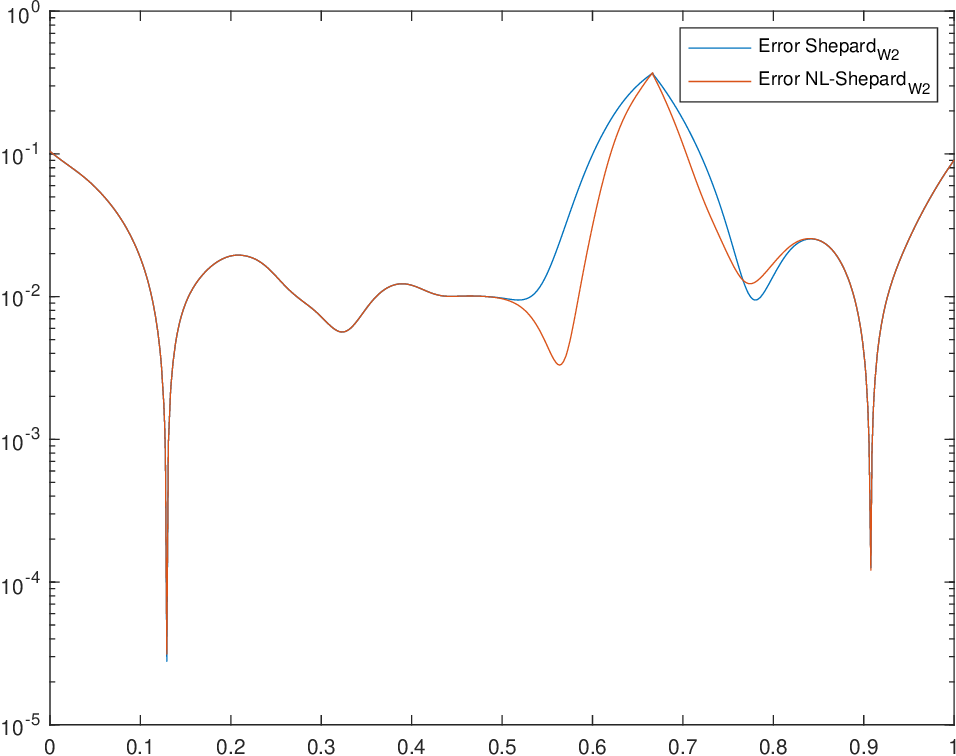}\\
		\includegraphics[width=5.2cm]{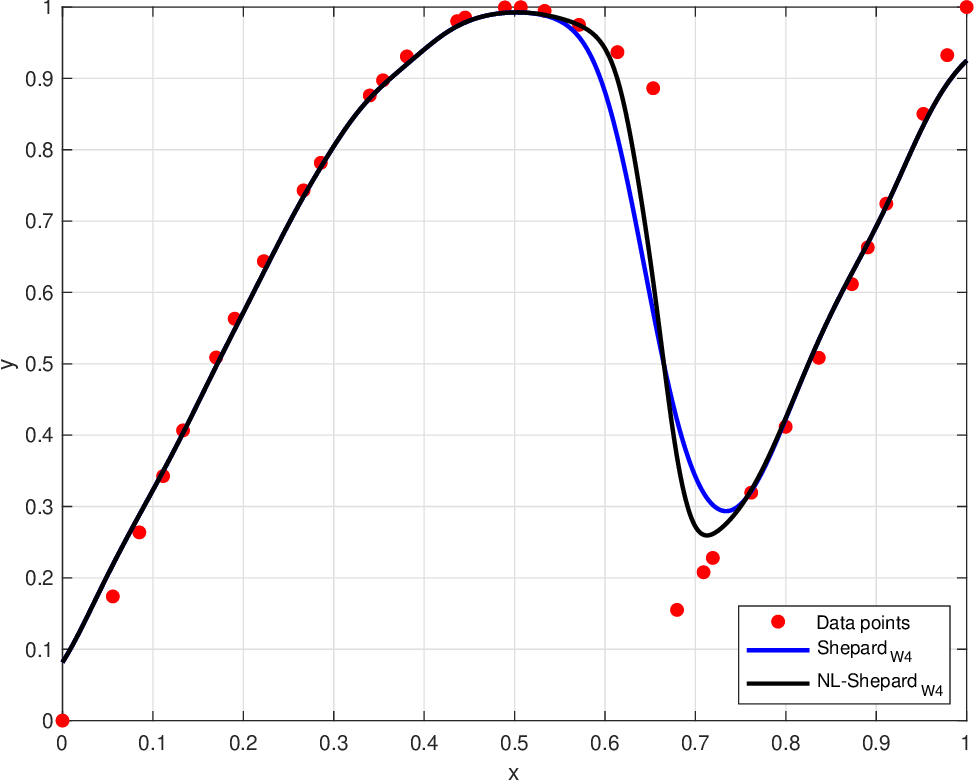} & 	\includegraphics[width=5.2cm]{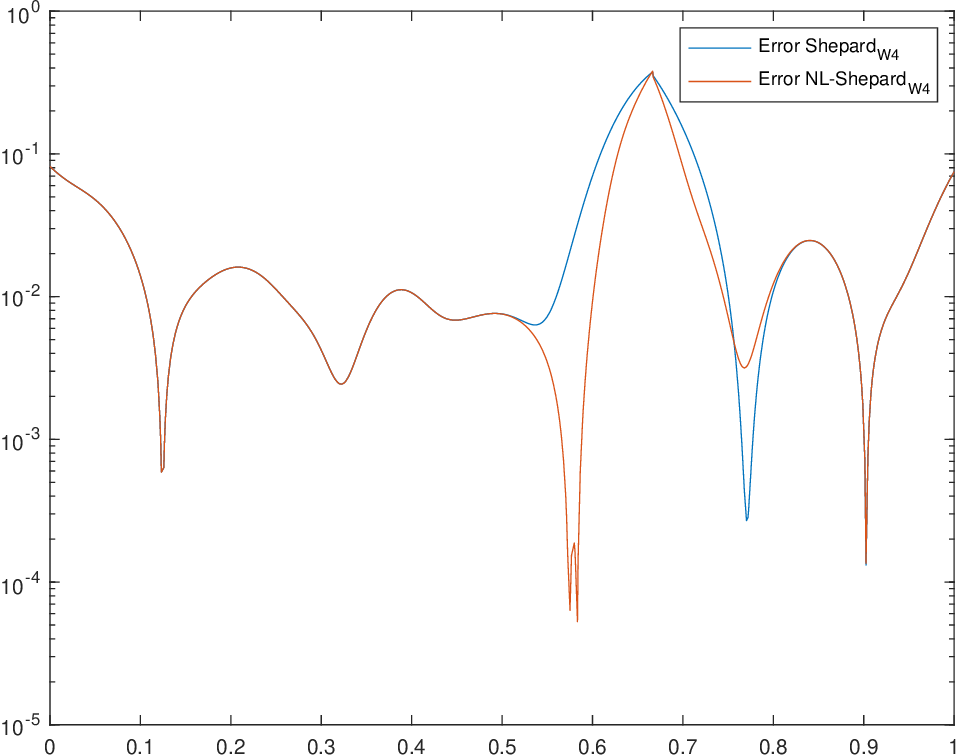}
		\end{tabular}
\end{center}
			\caption{Approximation to the function $g$ (red solid dots), Eq. \eqref{funciong}, using $n=32$ Halton points (red dots) and generating 20 Halton evaluation points between each of them. In each plot the classical (blue solid line) and data-dependent (black solid line) shepard algorithms have been used. The first column presents the approximation obtained by the algorithms. The second column presents the error obtained in a semilogarithmic scale. In this case, the blue line is for the classical Shepard and the red line for the data-dependent. From top to bottom these algorithms are: Shepard$_{\text{G}}$ and Shepard$_{\text{G}}$, Shepard$_{\text{W2}}$ and  NL-Shepard$_{\text{W2}}$,  Shepard$_{\text{W4}}$ and  NL-Shepard$_{\text{W4}}$,  Shepard$_{\text{M2}}$ and NL-Shepard$_{\text{M2}}$, Shepard$_{\text{M4}}$ and NL-Shepard$_{\text{M4}}$.}
		\label{exp2}
	\end{figure}

\subsection{Mitigation of the oscillations close to jump discontinuities in the function for bivariate data}

In this subsection, we examine how the linear and data-dependent schemes perform when applied to bivariate data containing piecewise smooth features. For this purpose, we consider the widely used Franke's function \cite{Franke}, which serves as a standard benchmark in many surface approximation studies.

\begin{equation}\label{frankesfunction}
\begin{split}
f(x,y) =\;& \frac{3}{4} e^{- \frac{1}{4}((9x - 2)^2 + (9y - 2)^2)} + \frac{3}{4} e^{- \frac{1}{49}(9x + 1)^2 - \frac{1}{10}(9y + 1)} \\
& + \frac{1}{2} e^{- \frac{1}{4}((9x - 7)^2 + (9y - 3)^2)} - \frac{1}{5} e^{-(9x - 4)^2 - (9y - 7)^2}.
\end{split}
\end{equation}
To create a jump discontinuity in the test function, we alter its definition in the following manner:
\begin{equation}\label{frankesdisc}
f_1(x,y) = \left\{
\begin{array}{ll}
2 + f(x,y), & x^2 + y^2 - 0.3^2 \geq 0, \\
-1 + f(x,y), & x^2 + y^2 - 0.3^2 < 0.
\end{array}
\right.
\end{equation}

As in the one-dimensional tests, we work with two types of sampling strategies: structured grids and Halton point sets. In both settings, the data sites lie within the square \([0,1]^2\), starting from an initial configuration of \(40 \times 40\) sample points.

For the experiments based on uniform grids (shown in Figures~\ref{exp1_2D}-\ref{exp5_2D}), the shape parameter \(\varepsilon\) is selected as in the previous subsections, following (\ref{vareps}), either for gridded data or Halton points.

Across all experiments (Figures~\ref{exp1_2D}--\ref{exp5_2D}), we approximate the piecewise smooth function \(f_1\) defined in Eq.~\eqref{frankesdisc}. The procedure is as follows. We begin with \(n = 40 \times 40\) initial data sites, sampled either on a uniform grid or using a Halton sequence. To refine the evaluation mesh, five additional interpolation points are inserted between each pair of neighboring initial sites in both coordinate directions. This refinement results in a final evaluation grid of \(235 \times 235\) nodes, which includes the original sample points.

Each figure is organized into three columns. The left column displays the approximation obtained using either the linear \(\textsc{RBF}_{\mathcal{H}}\) method or the data-dependent \(\textsc{NL-RBF}_{\mathcal{H}}\) variant, following the notation in Table~\ref{tabla1nucleos}. The central column shows the same approximation viewed from above (cenital perspective). The right column presents the pointwise absolute error computed over the entire evaluation grid using (\ref{error}).

Figure~\ref{exp1_2D} contains four rows, corresponding to the following cases:
\begin{itemize}
    \item Row 1: linear RBF\(_{\text{G}}\) approximation using gridded data.
    \item Row 2: data-dependent NL-RBF\(_{\text{G}}\) approximation using gridded data.
    \item Row 3: linear RBF\(_{\text{G}}\) approximation using Halton points.
    \item Row 4: data-dependent NL-RBF\(_{\text{G}}\) approximation using Halton points.
\end{itemize}

Figures~\ref{exp2_2D}--\ref{exp5_2D} follow the same pattern and correspond to different choices of radial basis kernels:
\begin{itemize}
    \item Figure~\ref{exp2_2D}: Mat\'ern \(\mathcal{C}^2\),
    \item Figure~\ref{exp3_2D}: Mat\'ern \(\mathcal{C}^4\),
    \item Figure~\ref{exp4_2D}: Wendland \(\mathcal{C}^2\),
    \item Figure~\ref{exp5_2D}: Wendland \(\mathcal{C}^4\).
\end{itemize}

For the experiments on uniform grids (first and second rows of Figures~\ref{exp1_2D}-\ref{exp5_2D}), the classical Shepard method exhibits pronounced smearing along the discontinuity curve, which is clearly reflected in the error plots and extend noticeably away from the jump discontinuity. The data-dependent variant reduces these artifacts substantially, although a slight blurring near the discontinuity persists. A similar behaviour is observed when Halton sampling is used (third and fourth rows of Figures~\ref{exp1_2D}-\ref{exp5_2D}).

	\begin{figure}[htbp!]
\begin{center}
		\begin{tabular}{ccc}
	\hspace{-1cm}\includegraphics[width=6cm]{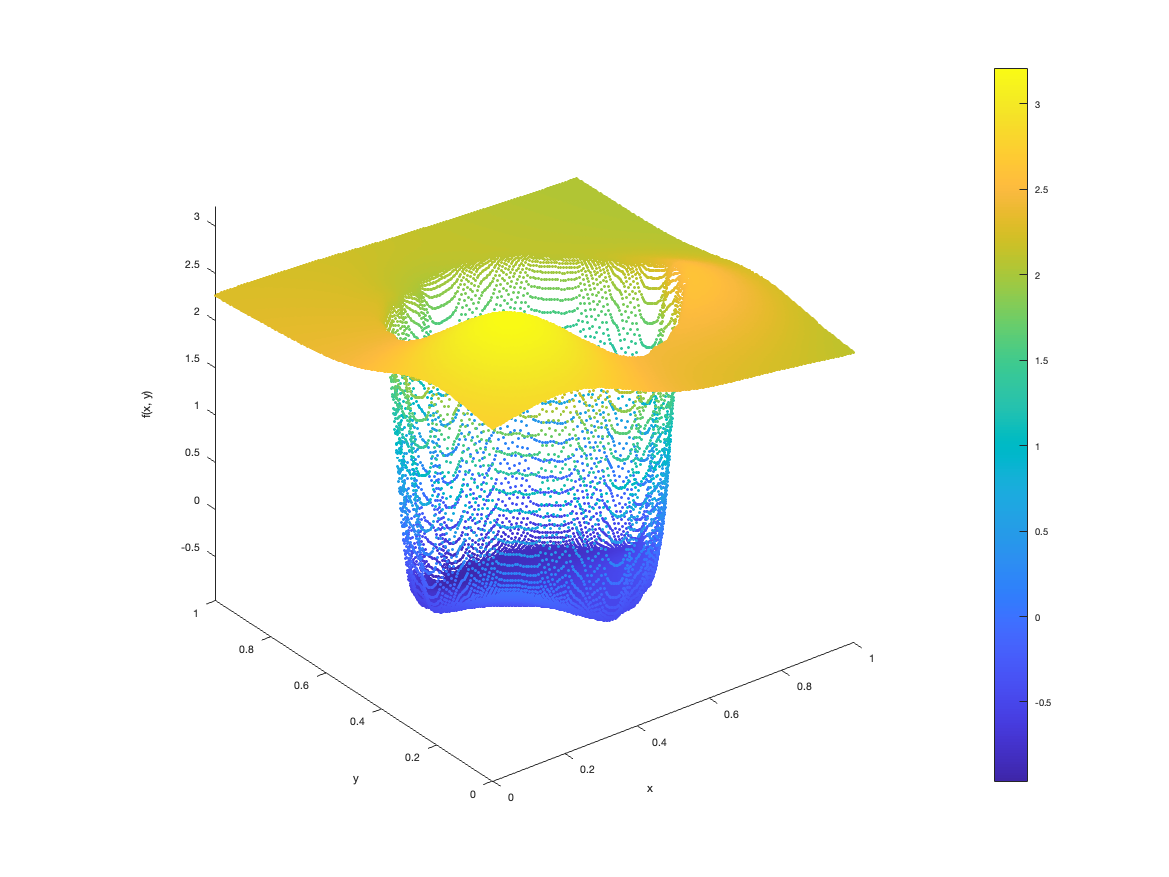} & 	\hspace{-0.9cm}\includegraphics[width=6cm]{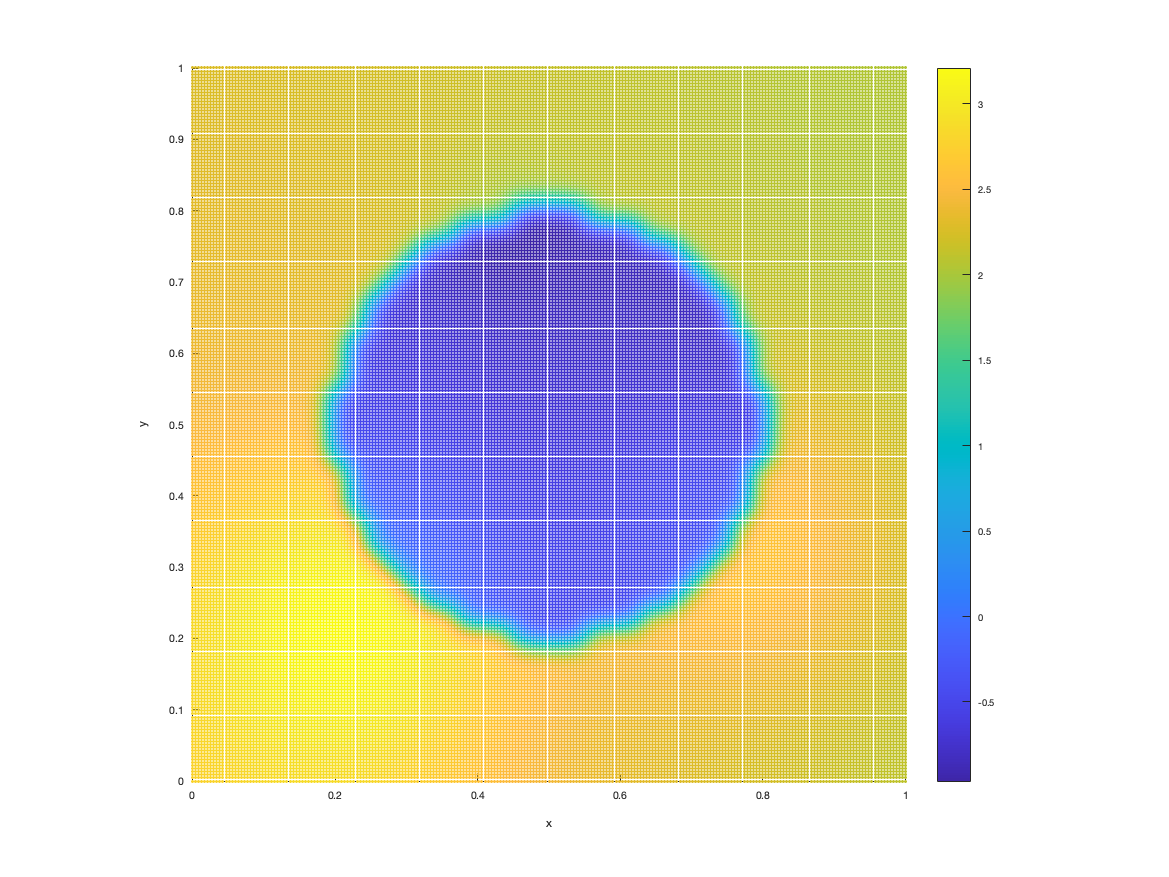} & \hspace{-0.9cm}\includegraphics[width=6cm]{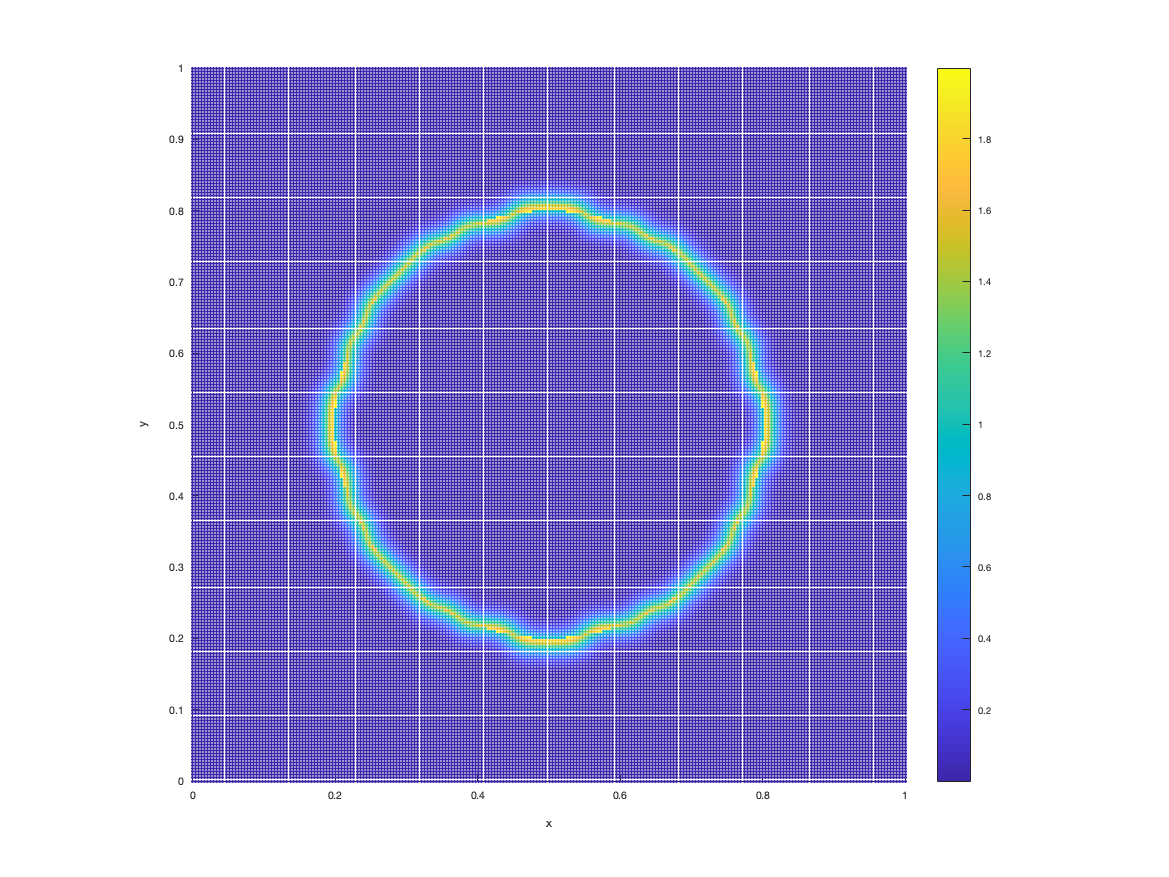}\\
\hspace{-1cm}\includegraphics[width=6cm]{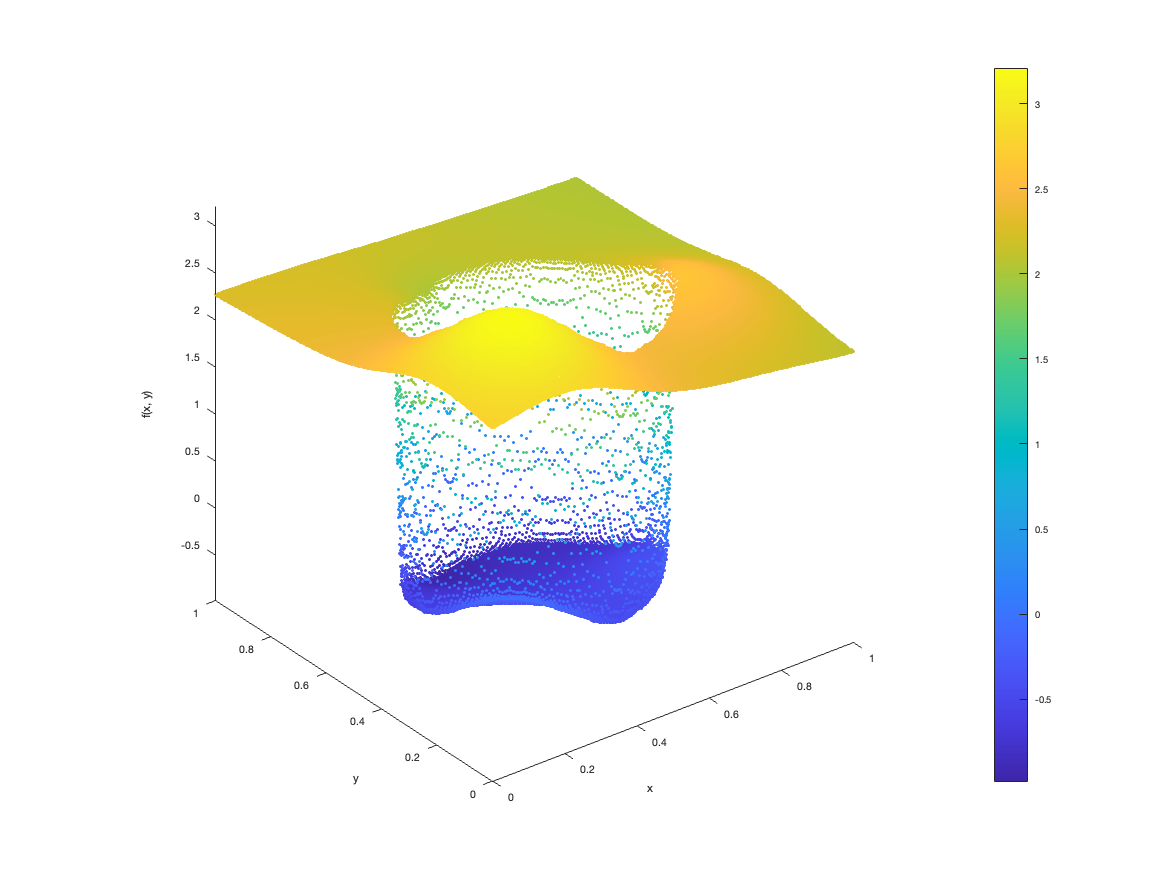} & 	\hspace{-0.9cm}\includegraphics[width=6cm]{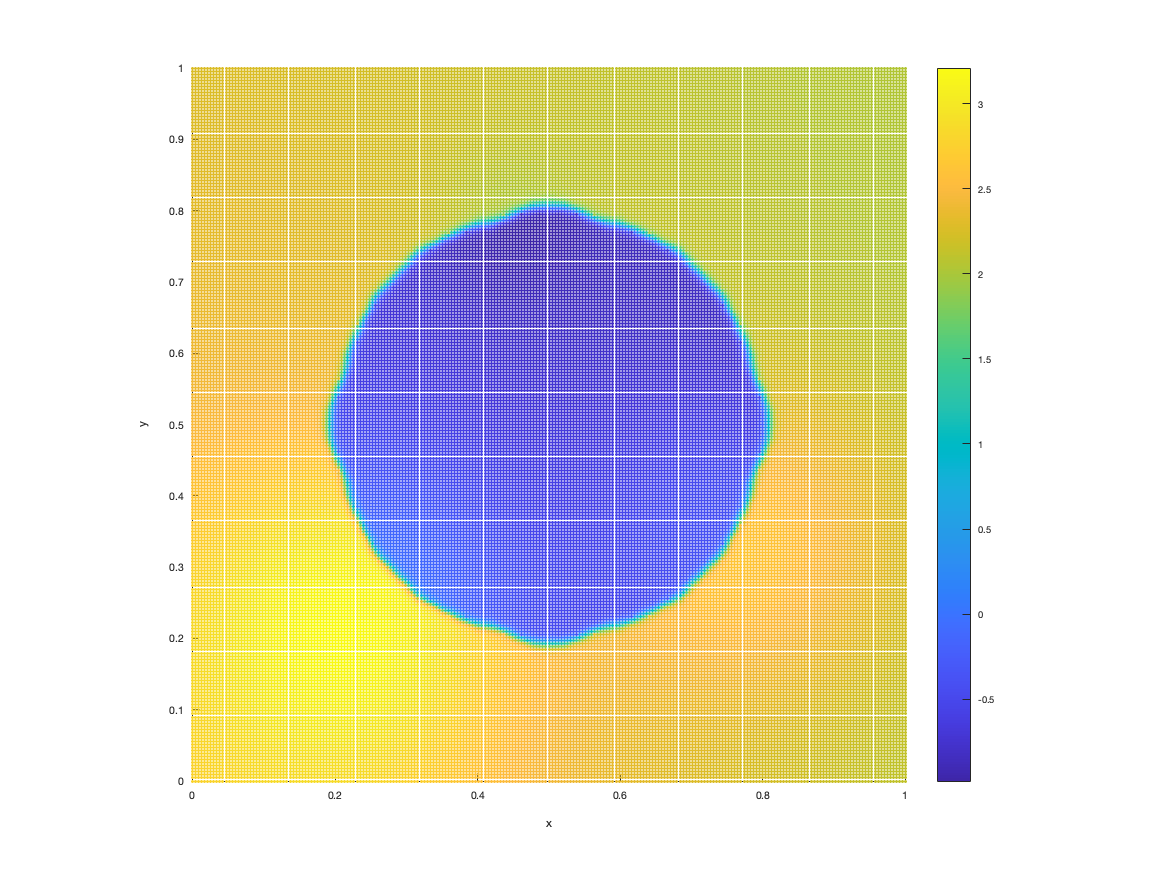} & \hspace{-0.9cm}\includegraphics[width=6cm]{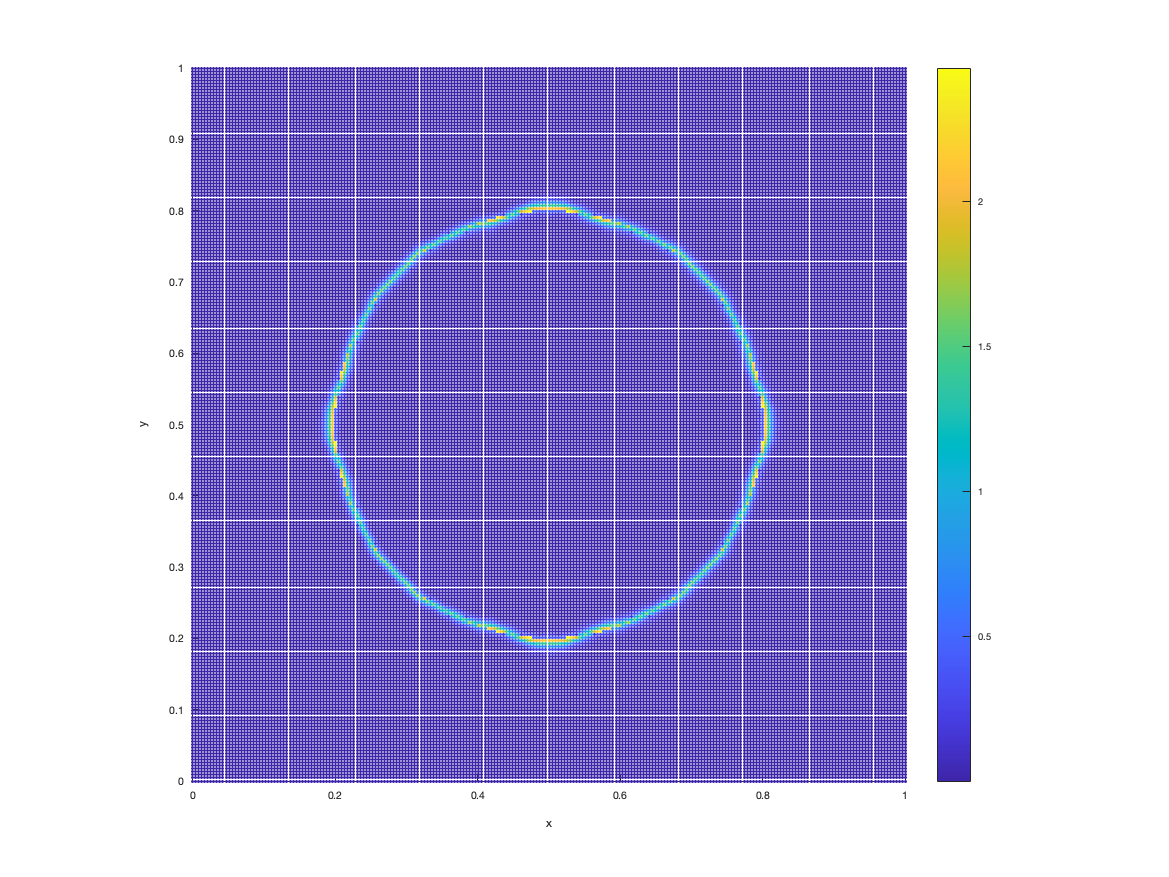}\\
	\hspace{-1cm}\includegraphics[width=6cm]{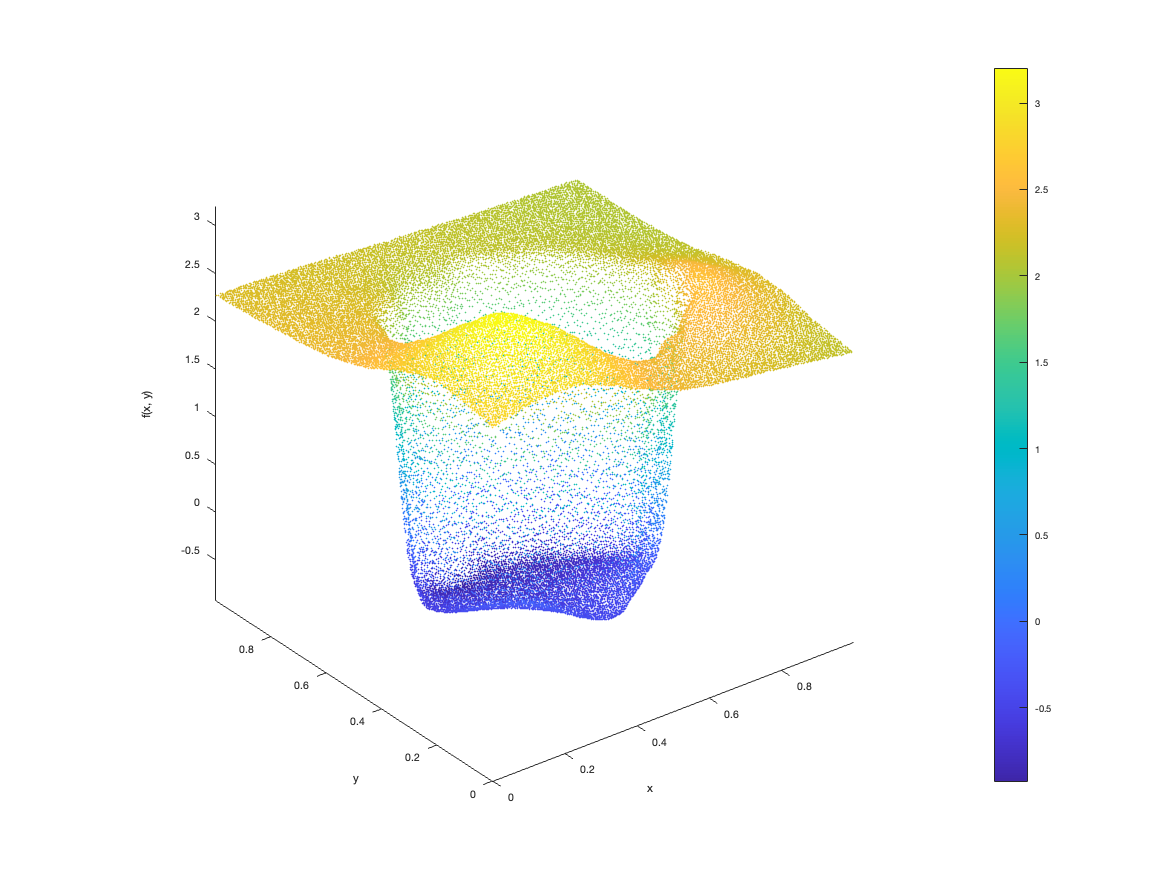} & 	\hspace{-0.9cm}\includegraphics[width=6cm]{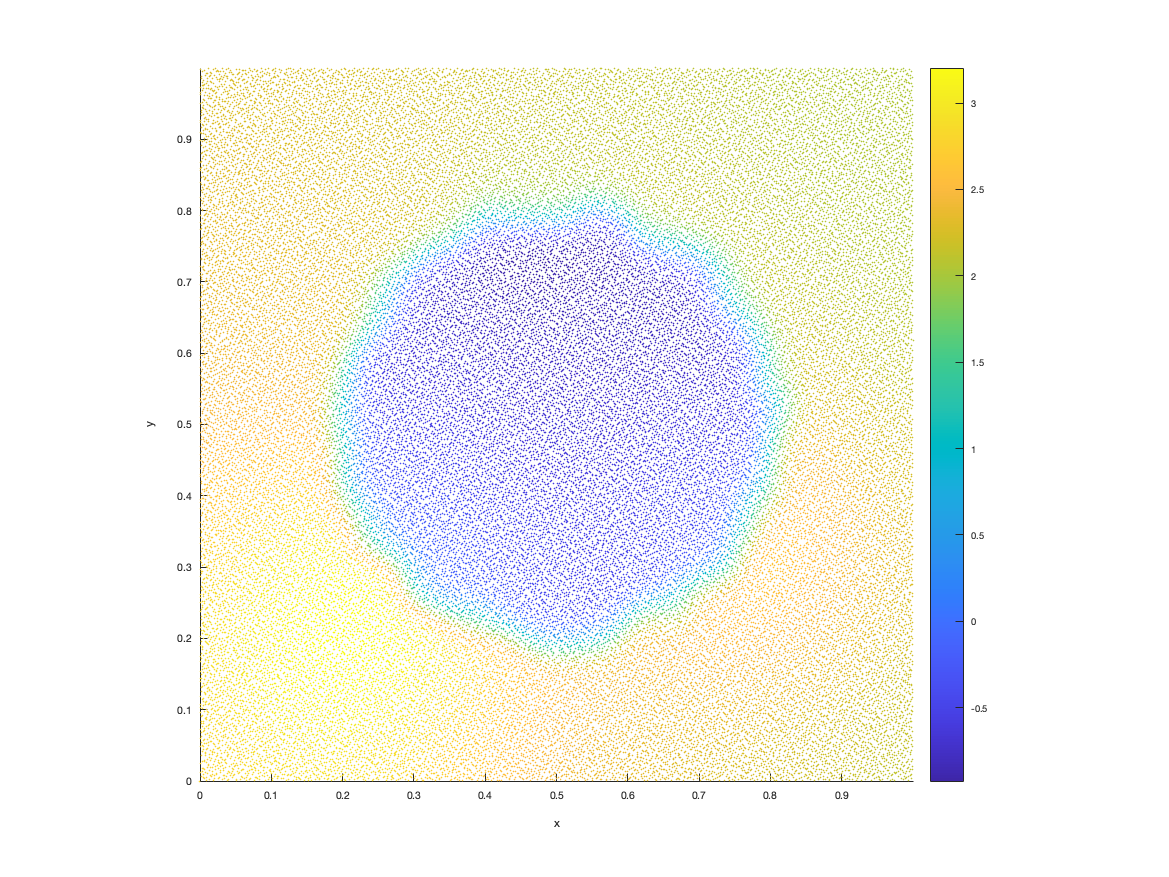} & \hspace{-0.9cm}\includegraphics[width=6cm]{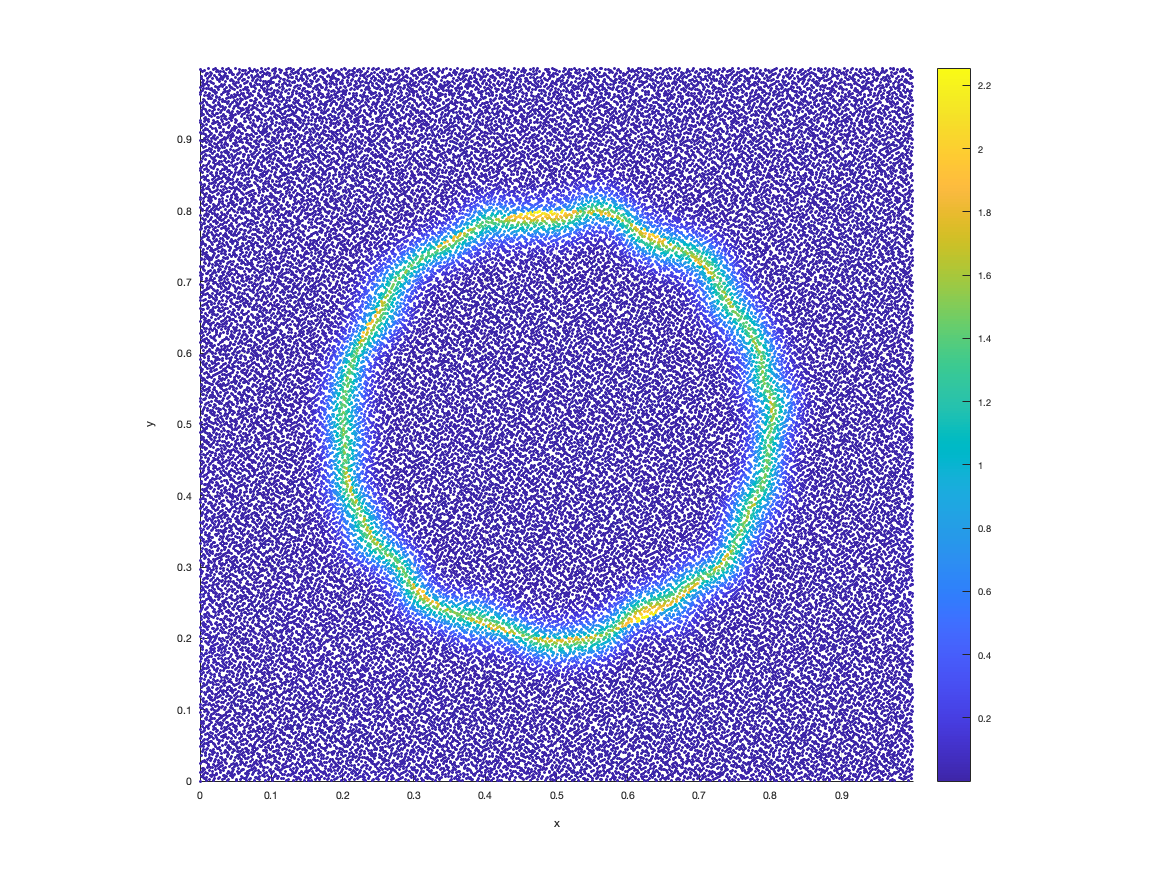}\\
\hspace{-1cm}\includegraphics[width=6cm]{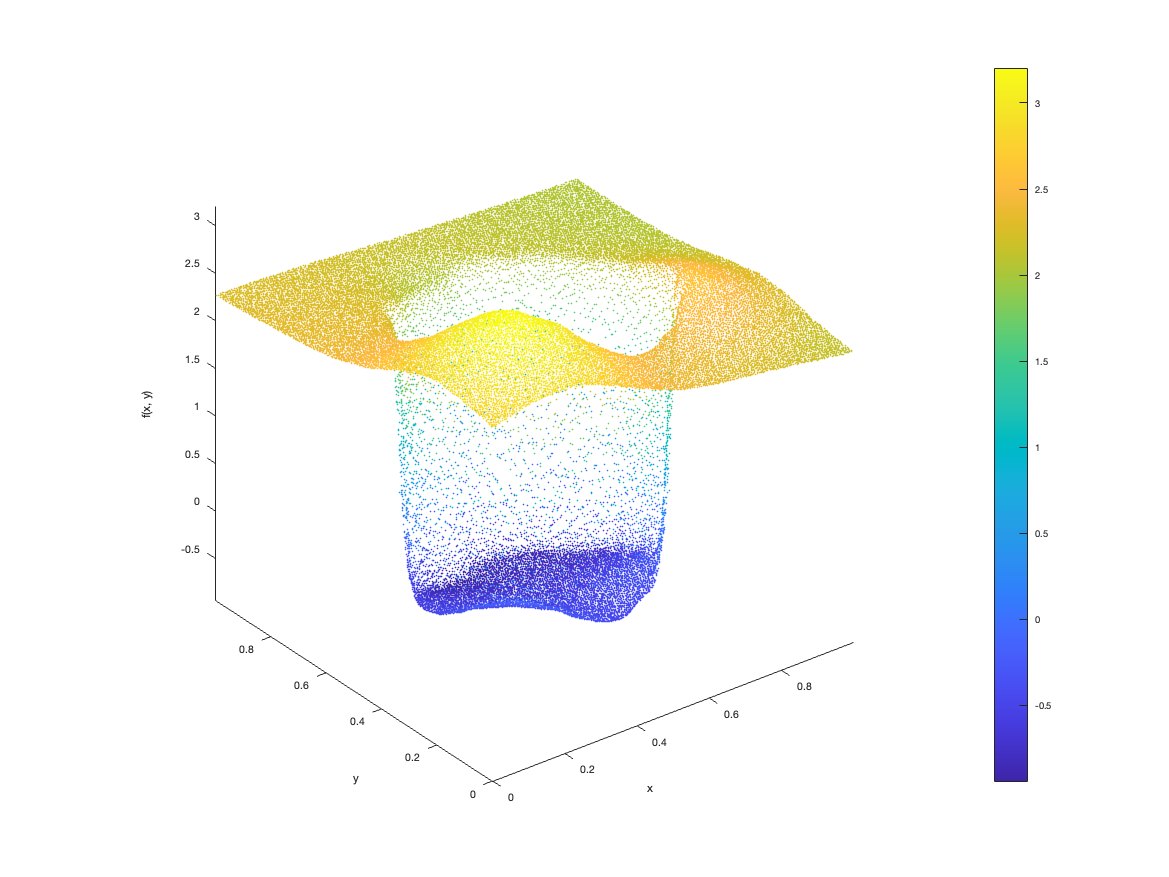} & 	\hspace{-0.9cm}\includegraphics[width=6cm]{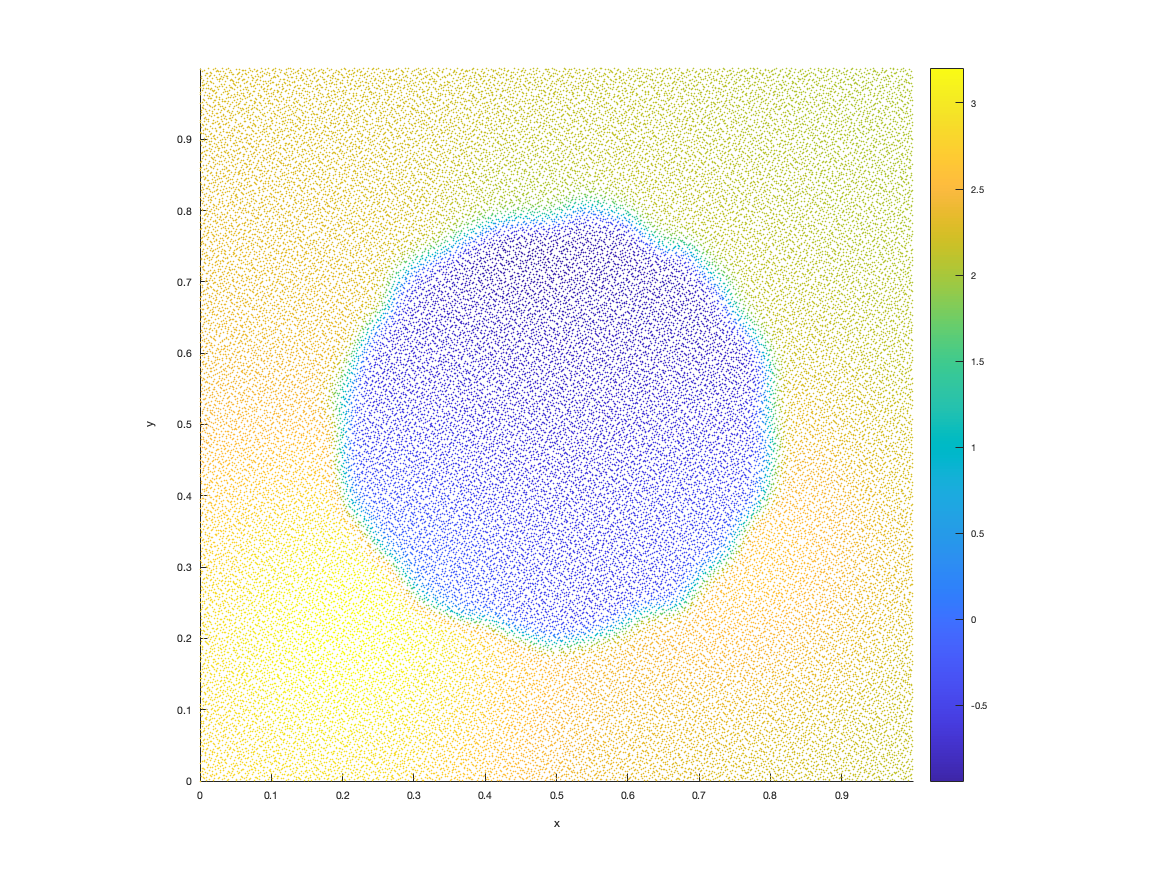} & \hspace{-0.9cm}\includegraphics[width=6cm]{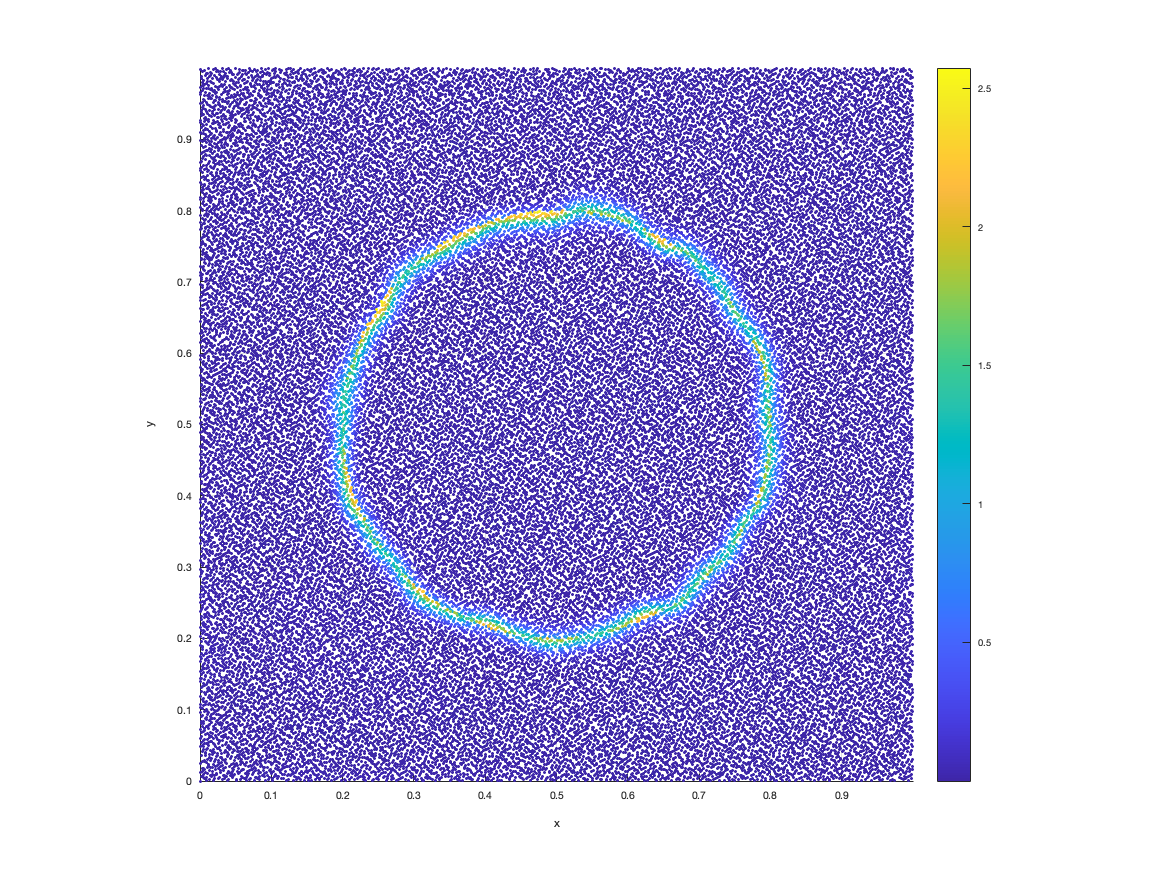}\\
		\end{tabular}
\end{center}
			\caption{Approximation of the piecewise smooth function \(f_1\) (defined in Eq.~\eqref{frankesdisc}). 
Each figure is arranged in three columns: the left column shows the surface reconstructed using either the linear RBF\(_{\text{G}}\) method or its data-dependent counterpart NL-RBF\(_{\text{G}}\). The central column displays the same approximation from a top (cenital) view. 
The right column depicts the pointwise absolute error over the evaluation grid.
The figure contains four rows: 
Row~1 shows the linear RBF\(_{\text{G}}\) approximation using gridded points, 
Row~2 shows the data-dependent NL-RBF\(_{\text{G}}\) approximation on the same grid, 
Row~3 presents the linear RBF\(_{\text{G}}\) approximation computed from Halton points, 
and Row~4 displays the corresponding data-dependent NL-RBF\(_{\text{G}}\) result.}
		\label{exp1_2D}
	\end{figure}

	\begin{figure}[htbp!]
\begin{center}
		\begin{tabular}{ccc}
	\hspace{-1cm}\includegraphics[width=6cm]{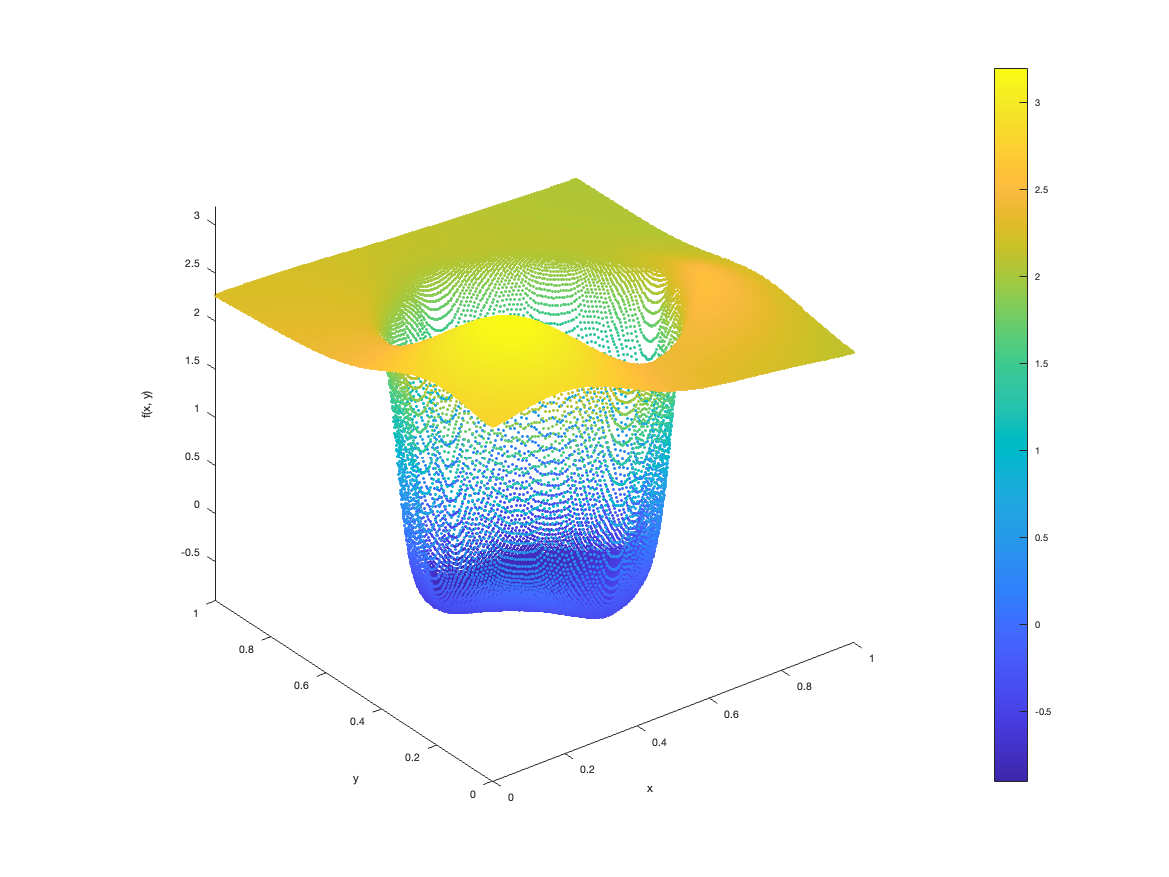} & 	\hspace{-0.9cm}\includegraphics[width=6cm]{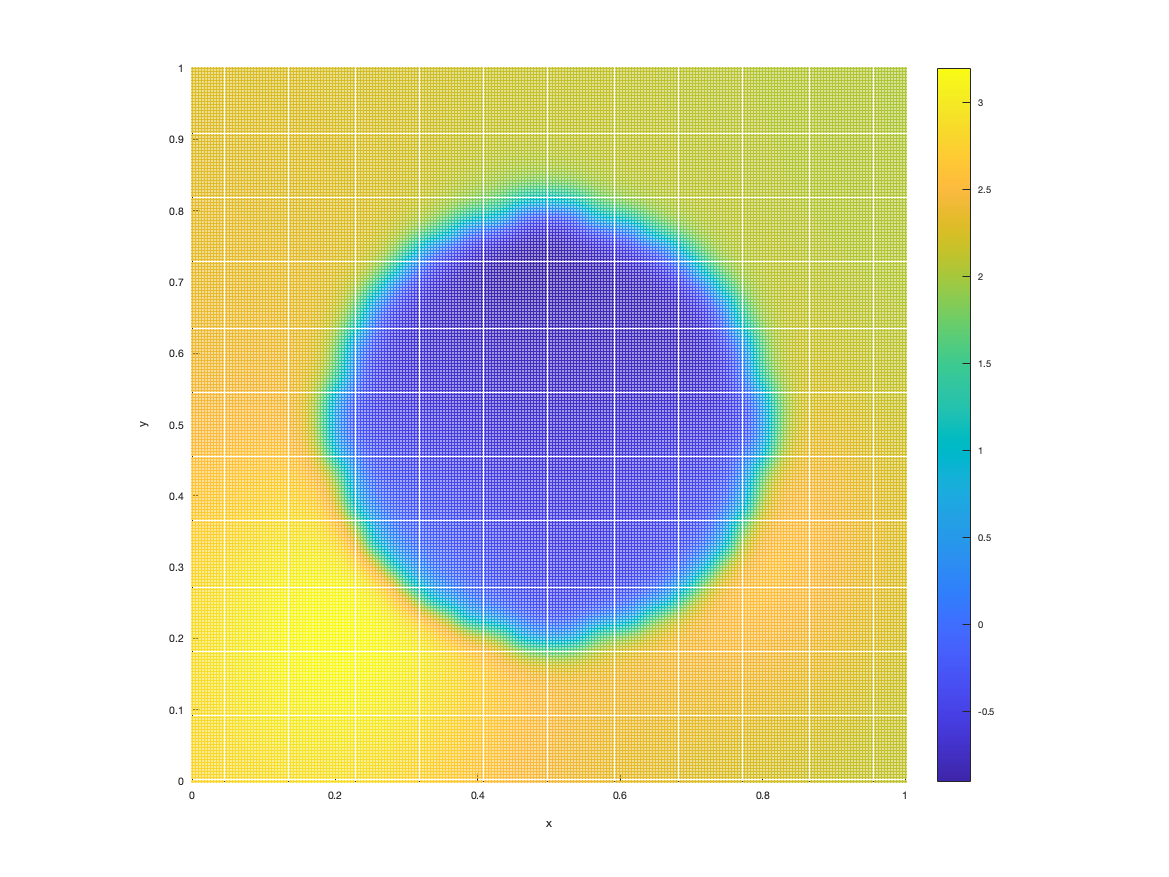} & \hspace{-0.9cm}\includegraphics[width=6cm]{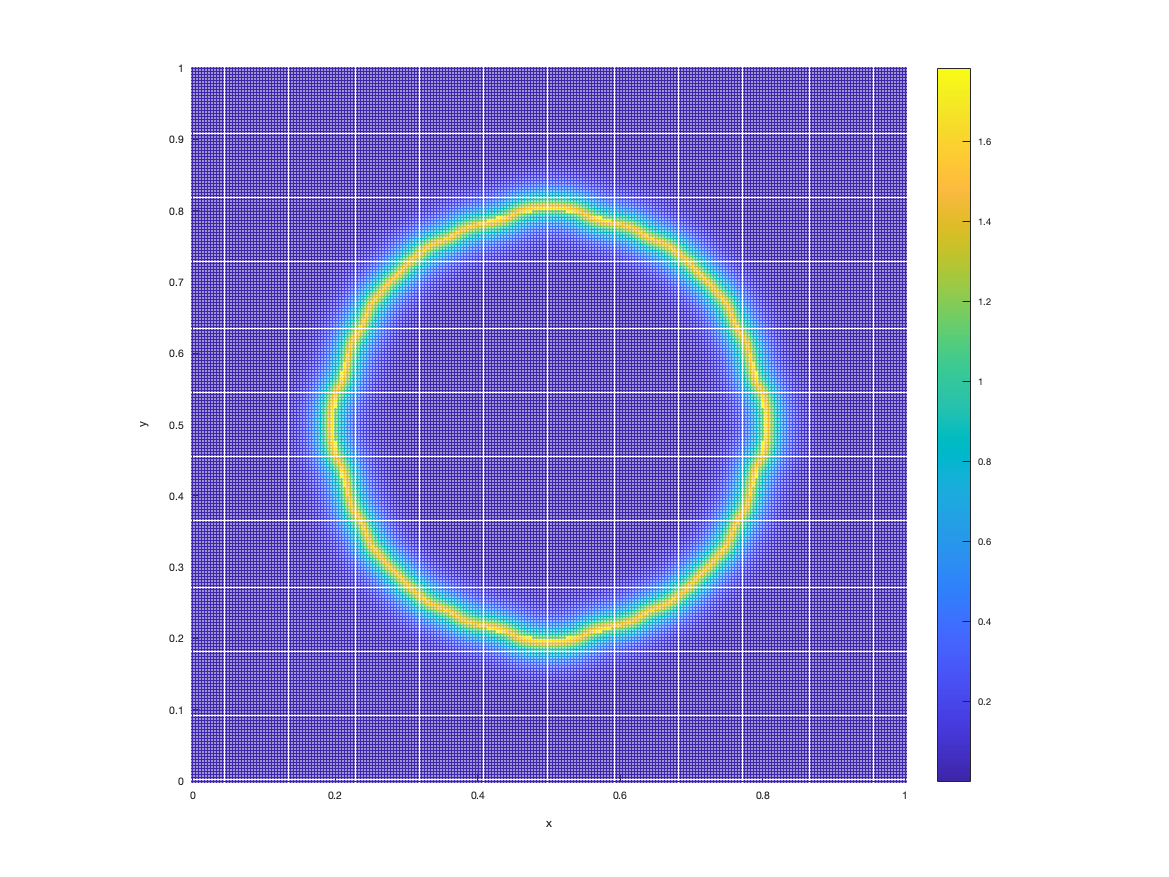}\\
\hspace{-1cm}\includegraphics[width=6cm]{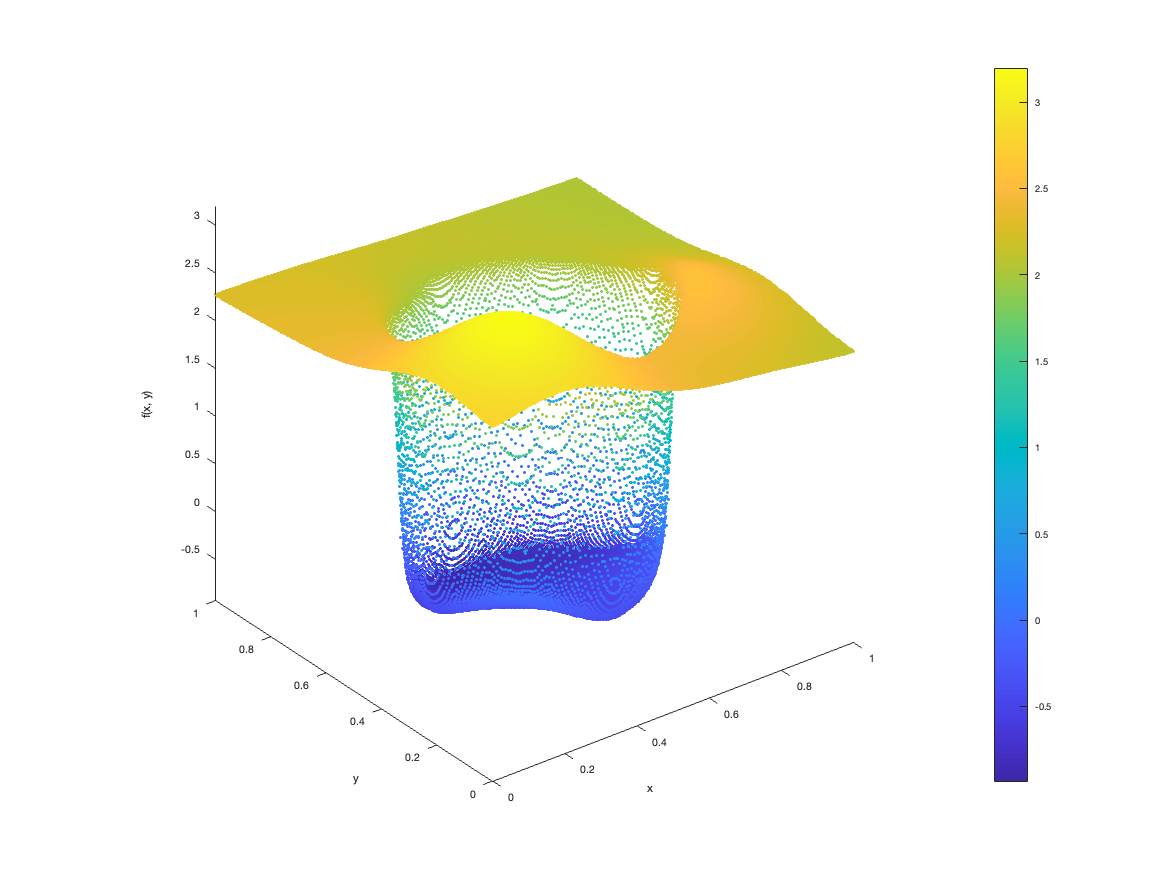} & 	\hspace{-0.9cm}\includegraphics[width=6cm]{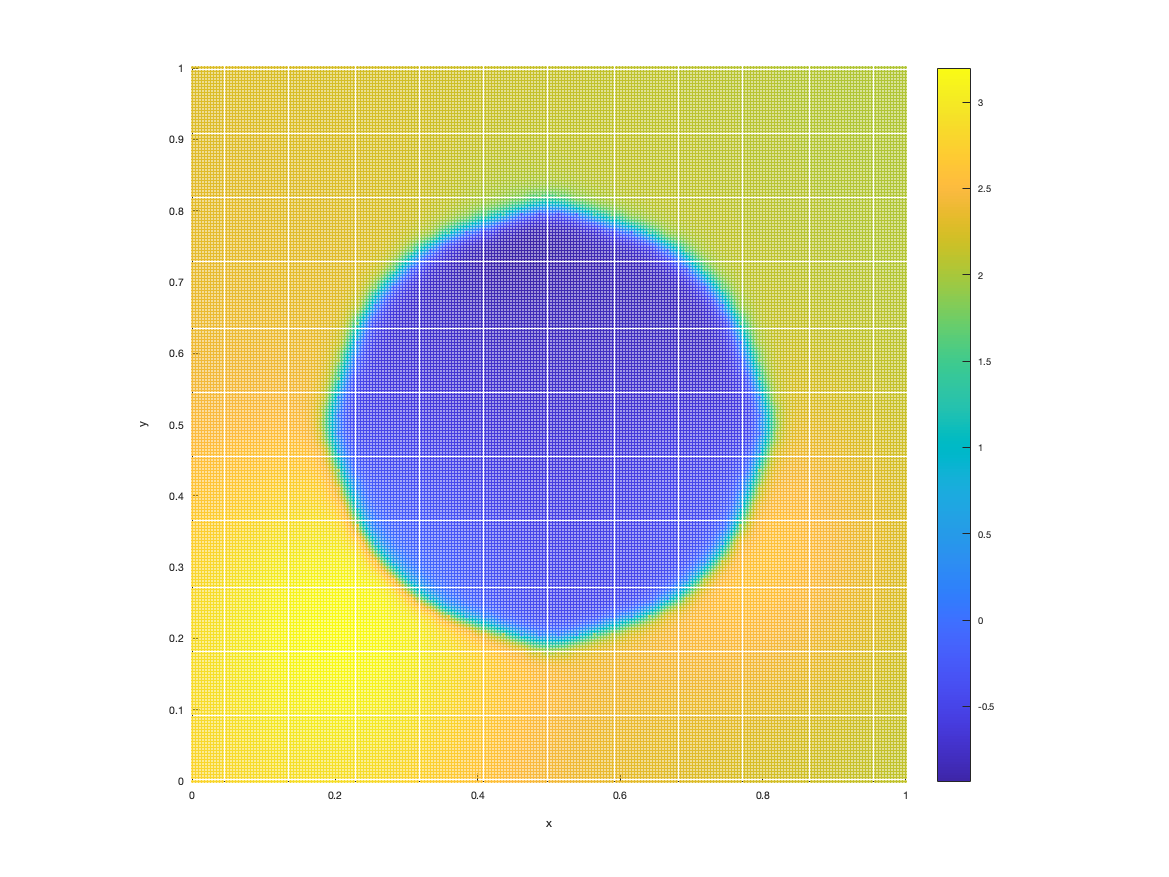} & \hspace{-0.9cm}\includegraphics[width=6cm]{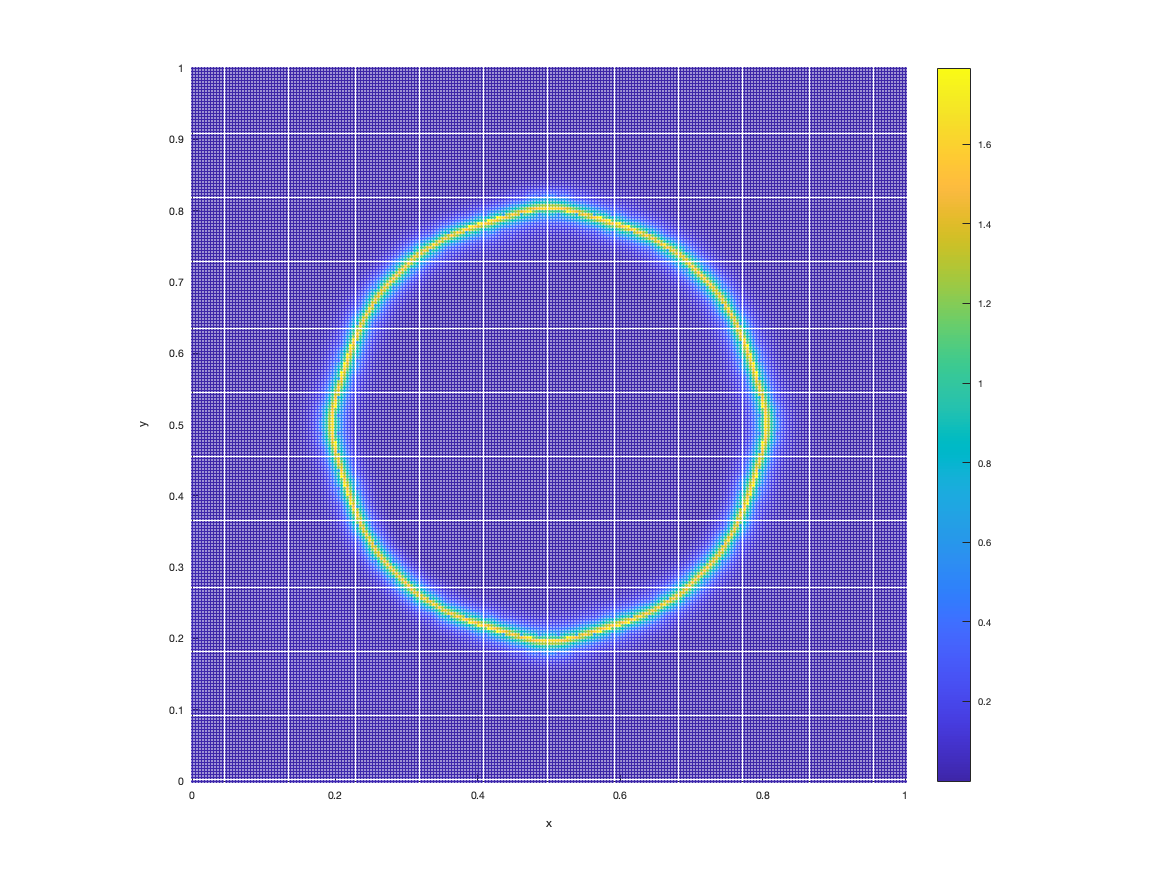}\\
	\hspace{-1cm}\includegraphics[width=6cm]{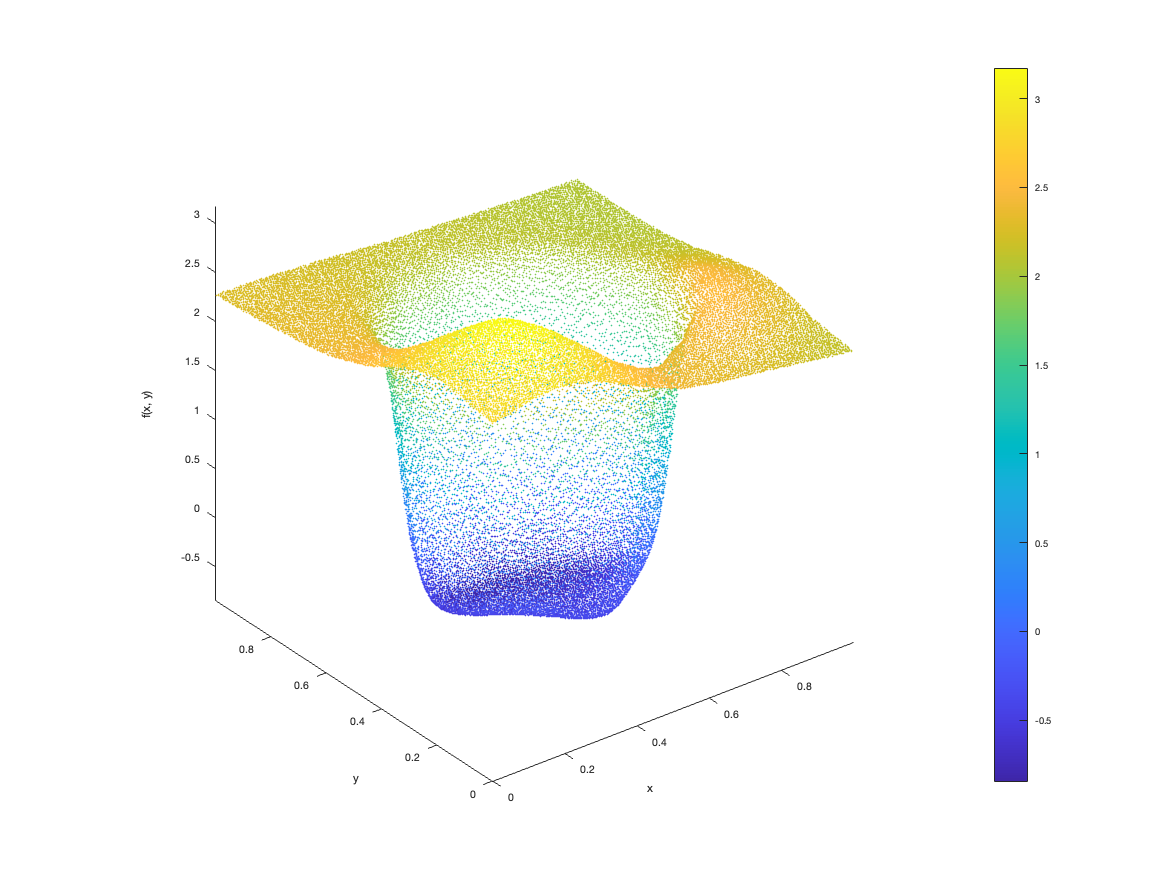} & 	\hspace{-0.9cm}\includegraphics[width=6cm]{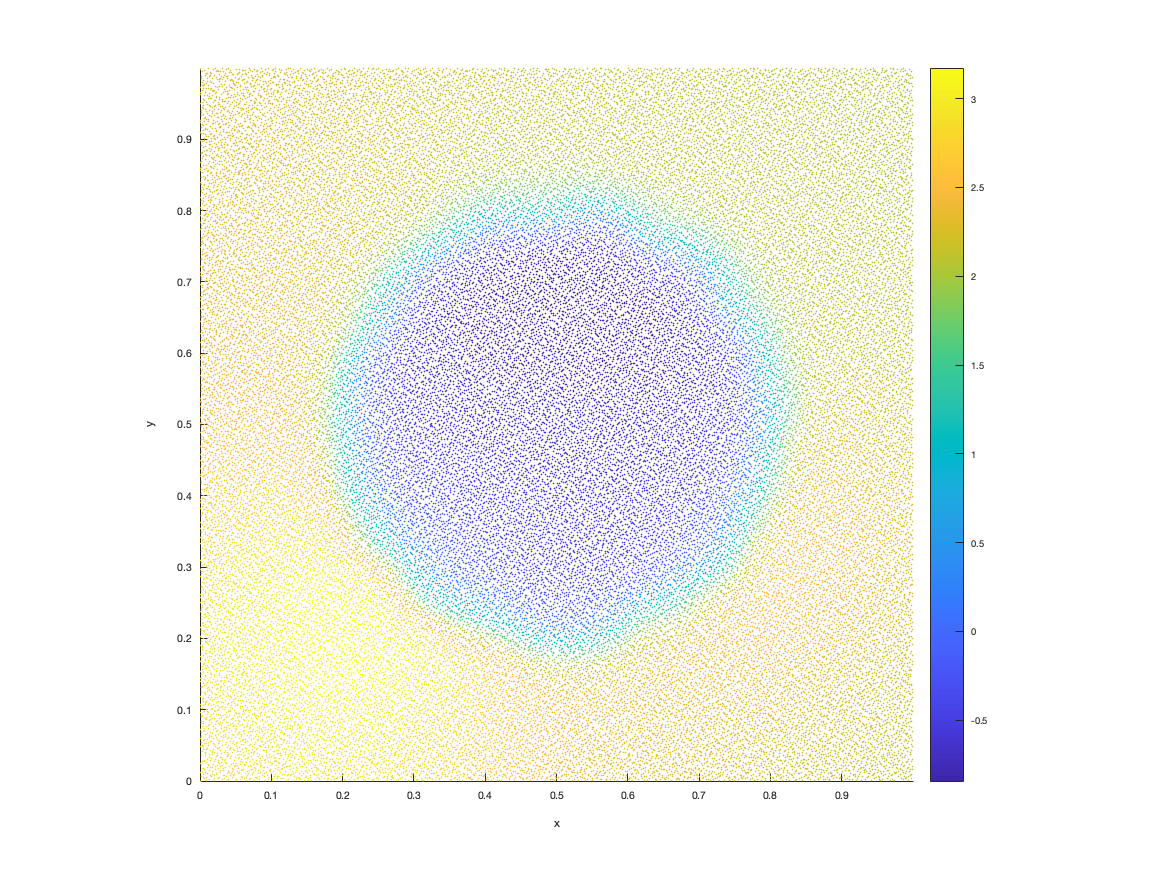} & \hspace{-0.9cm}\includegraphics[width=6cm]{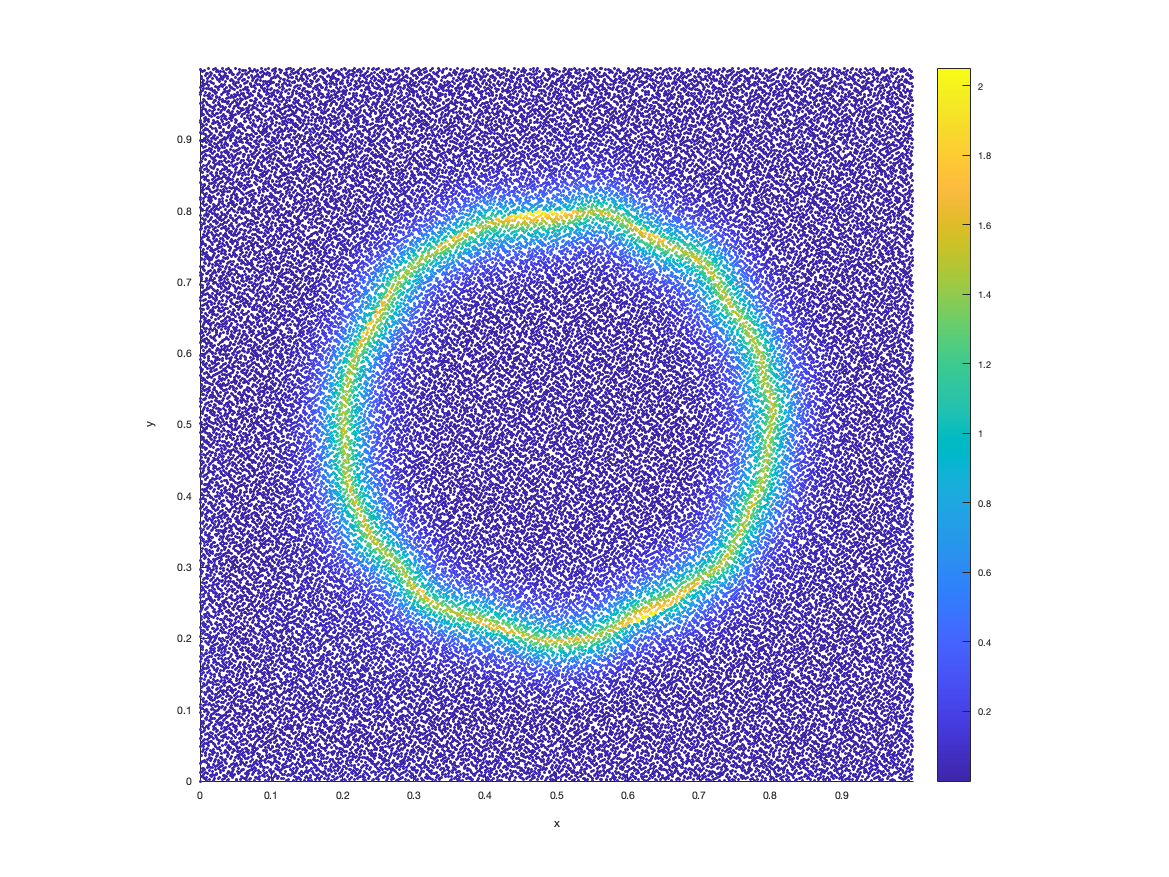}\\
\hspace{-1cm}\includegraphics[width=6cm]{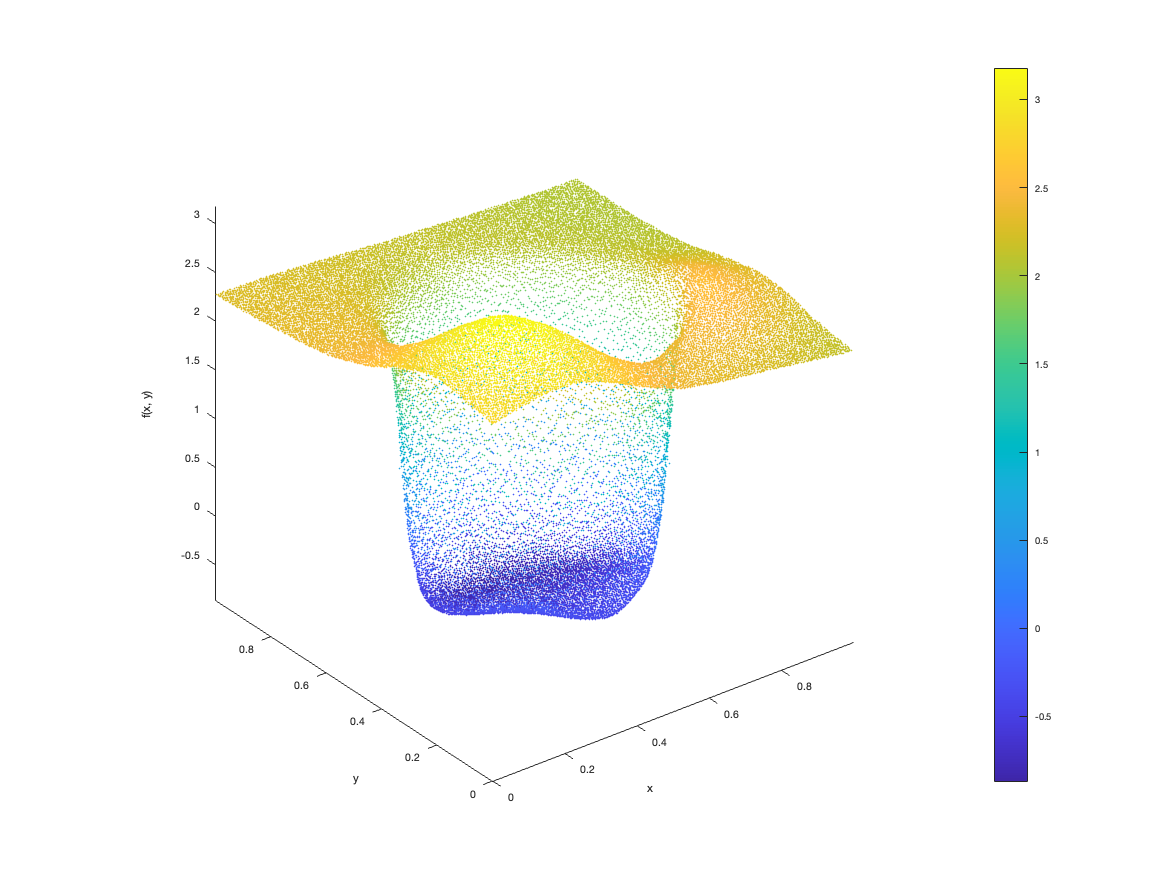} & 	\hspace{-0.9cm}\includegraphics[width=6cm]{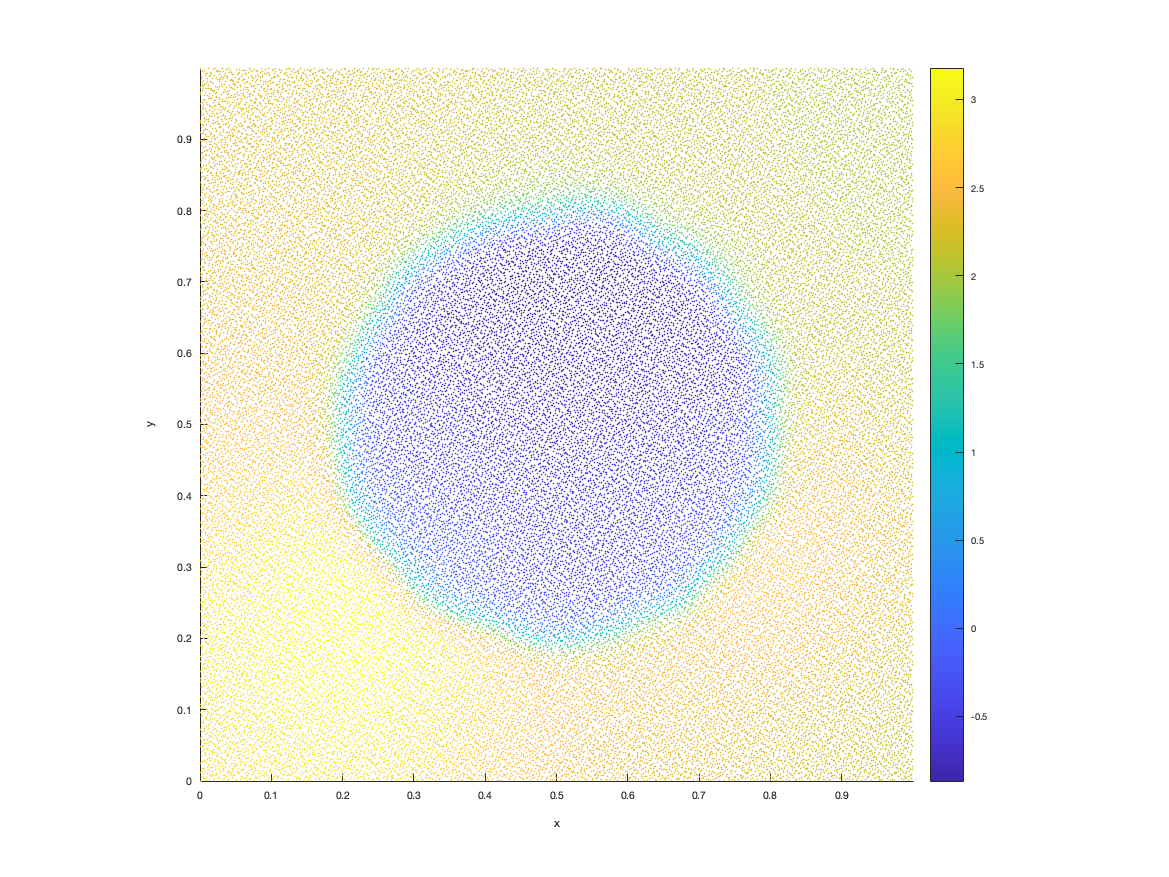} & \hspace{-0.9cm}\includegraphics[width=6cm]{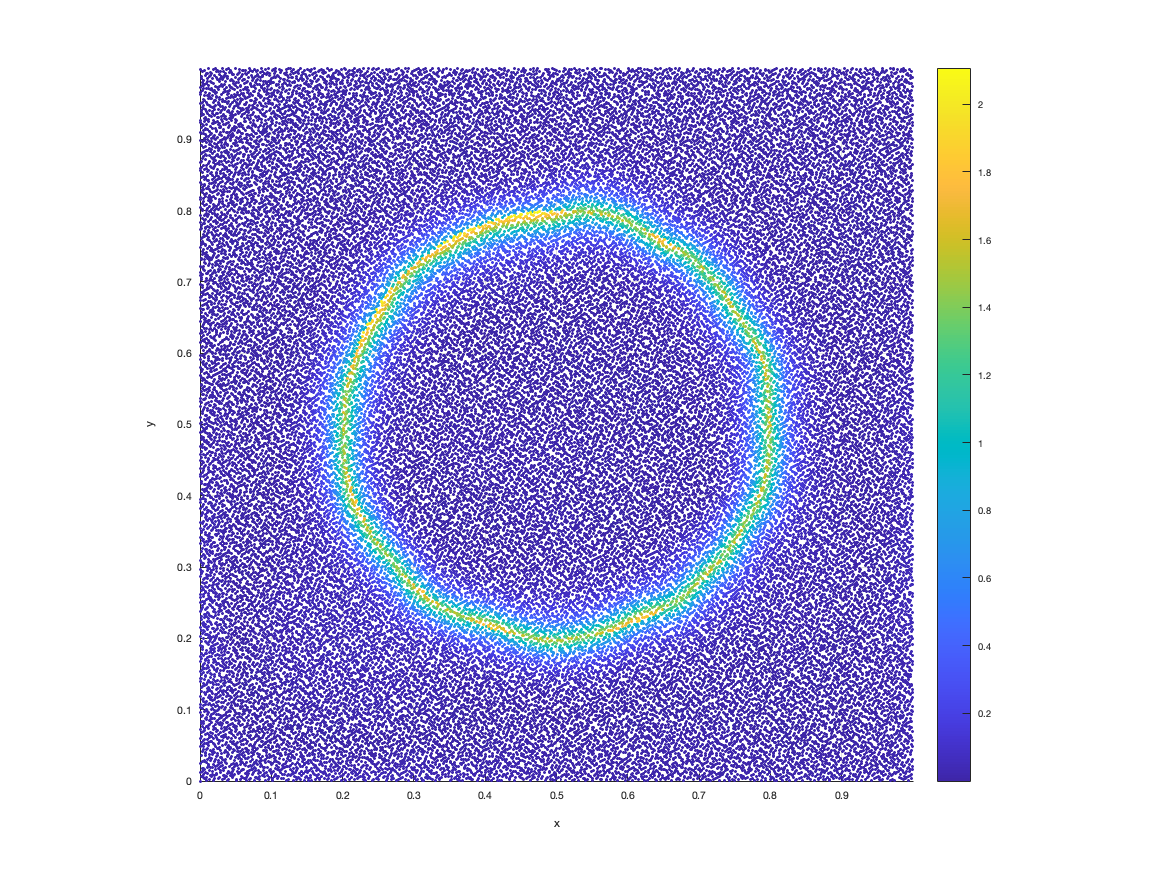}\\
		\end{tabular}
\end{center}
			\caption{Approximation of the piecewise smooth function \(f_1\) (defined in Eq.~\eqref{frankesdisc}). 
Each figure is arranged in three columns: the left column shows the surface reconstructed using either the linear RBF\(_{\text{M2}}\) method or its data-dependent counterpart NL-RBF\(_{\text{M2}}\). The central column displays the same approximation from a top (cenital) view. 
The right column depicts the pointwise absolute error over the evaluation grid.
The figure contains four rows: 
Row~1 shows the linear RBF\(_{\text{M2}}\) approximation using gridded points, 
Row~2 shows the data-dependent NL-RBF\(_{\text{M2}}\) approximation on the same grid, 
Row~3 presents the linear RBF\(_{\text{M2}}\) approximation computed from Halton points, 
and Row~4 displays the corresponding data-dependent NL-RBF\(_{\text{M2}}\) result.}
		\label{exp2_2D}
	\end{figure}

	\begin{figure}[htbp!]
\begin{center}
		\begin{tabular}{ccc}
	\hspace{-1cm}\includegraphics[width=6cm]{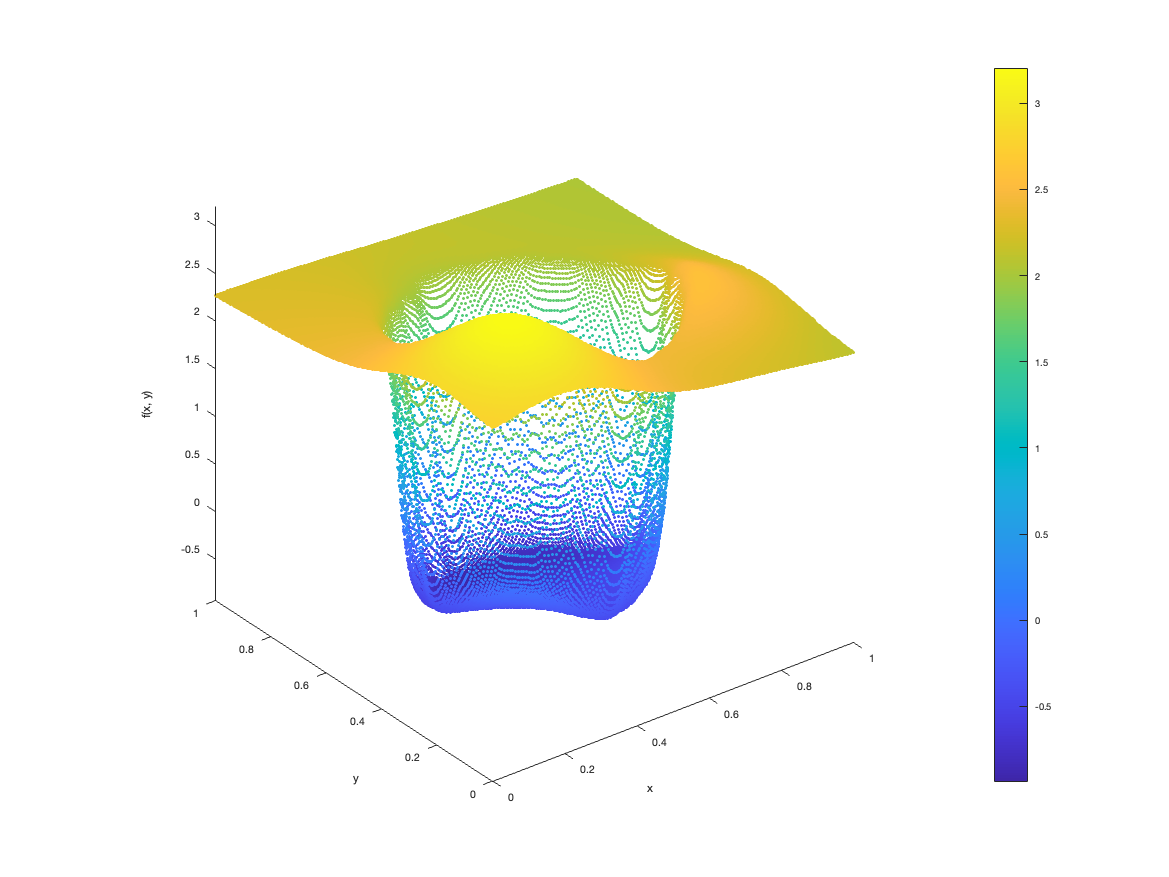} & 	\hspace{-0.9cm}\includegraphics[width=6cm]{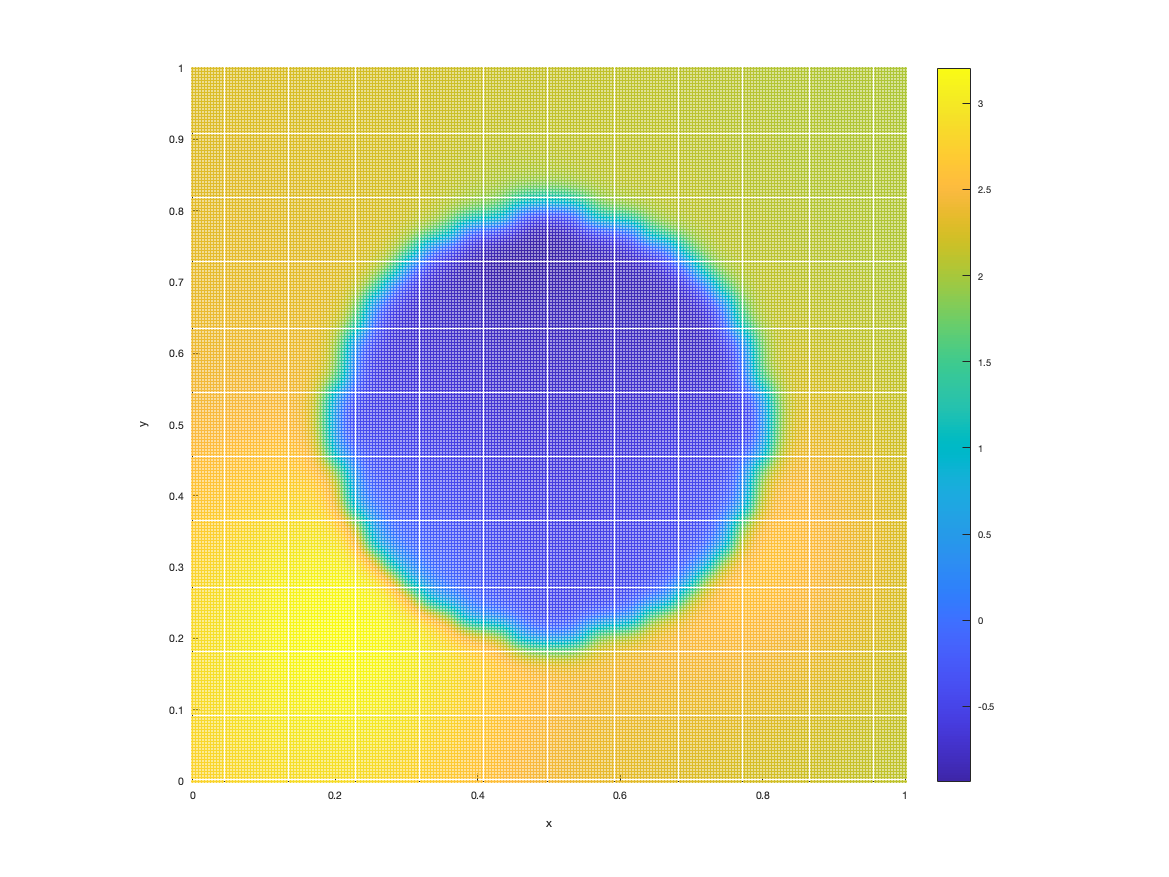} & \hspace{-0.9cm}\includegraphics[width=6cm]{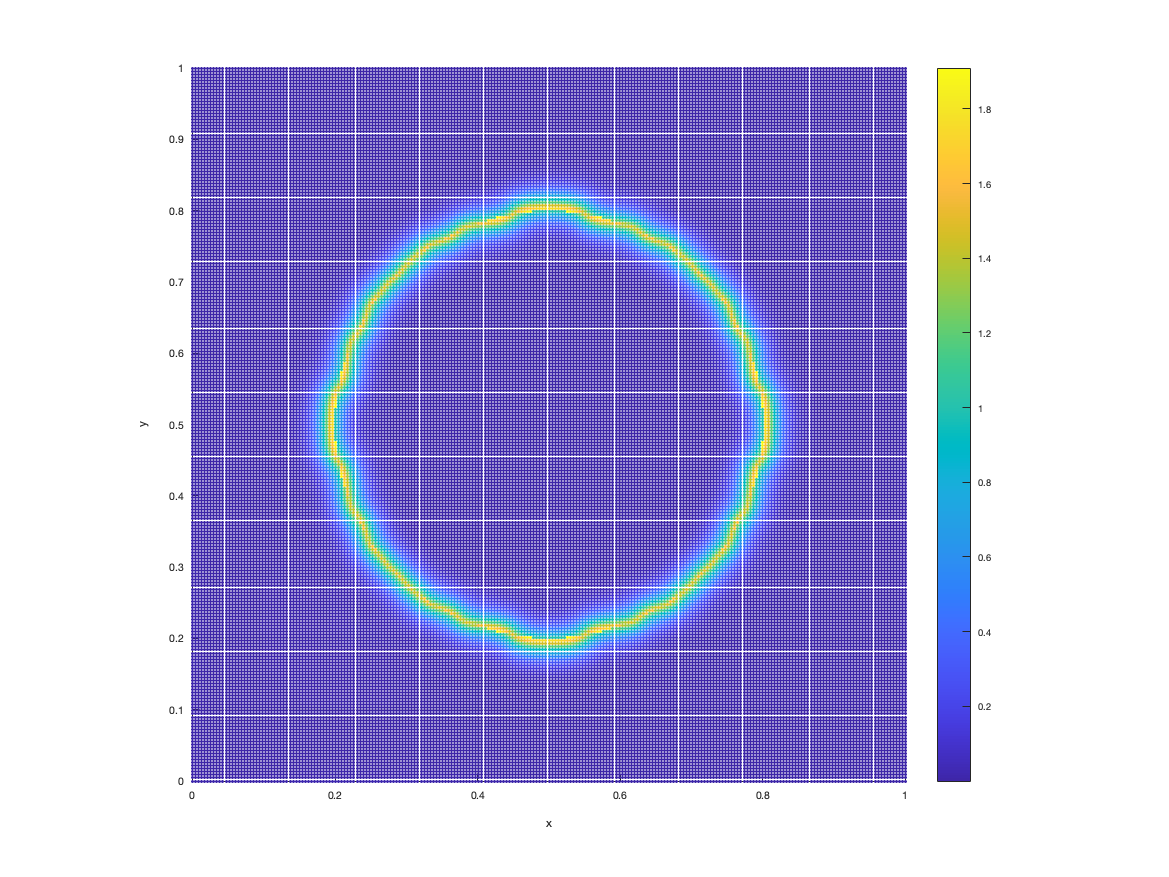}\\
\hspace{-1cm}\includegraphics[width=6cm]{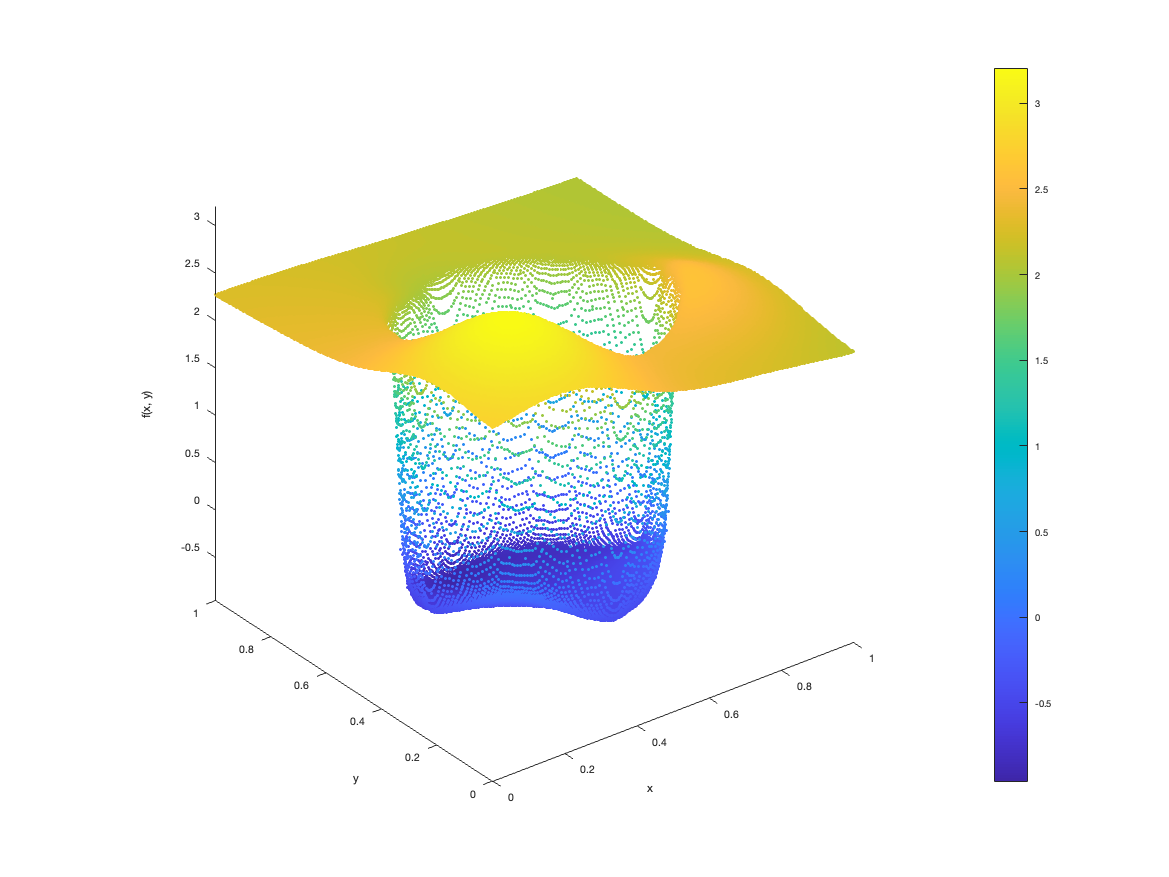} & 	\hspace{-0.9cm}\includegraphics[width=6cm]{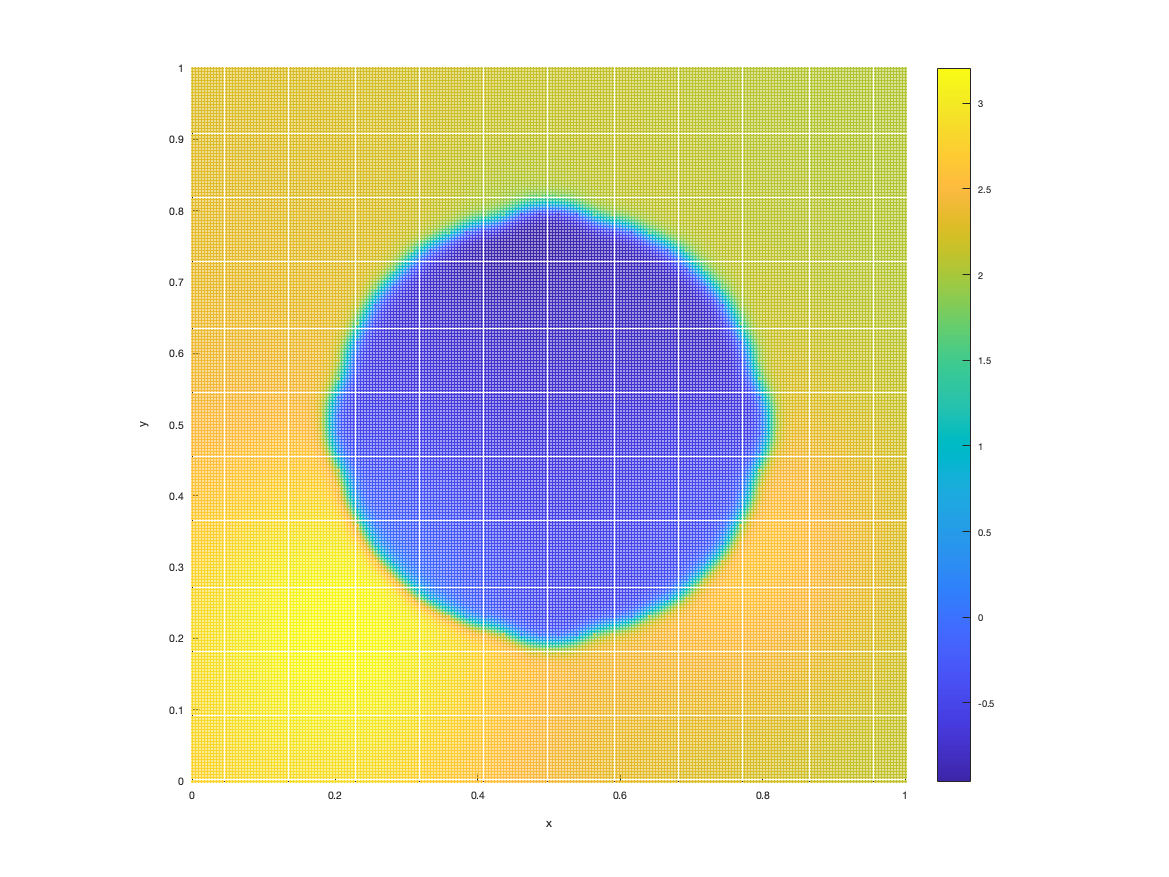} & \hspace{-0.9cm}\includegraphics[width=6cm]{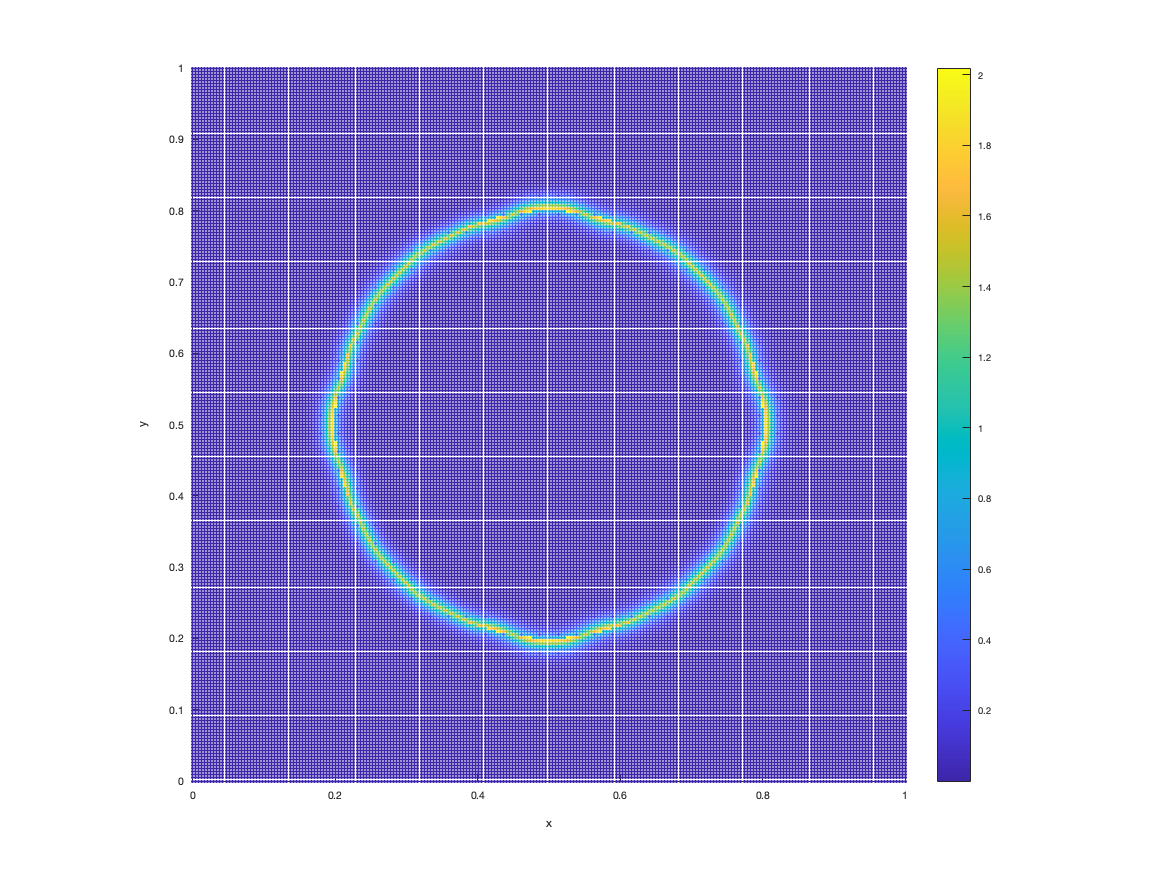}\\
	\hspace{-1cm}\includegraphics[width=6cm]{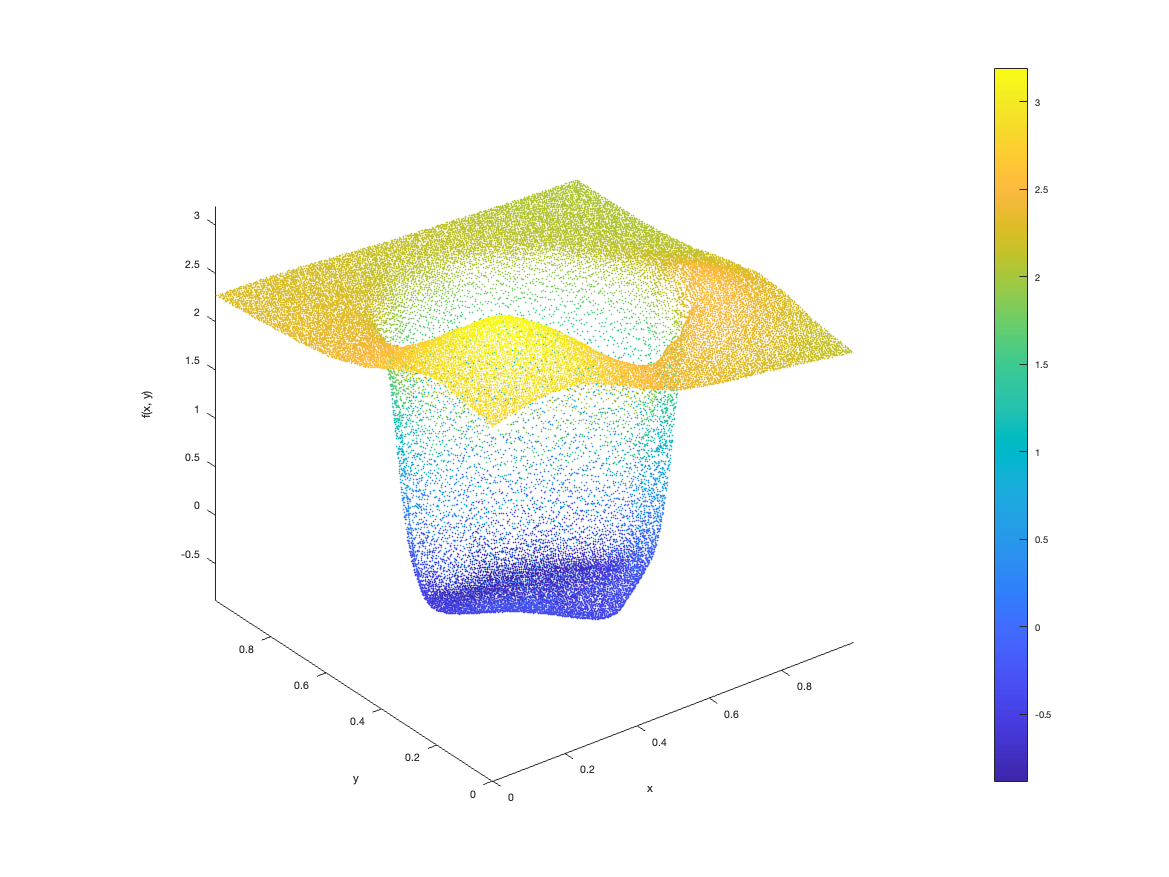} & 	\hspace{-0.9cm}\includegraphics[width=6cm]{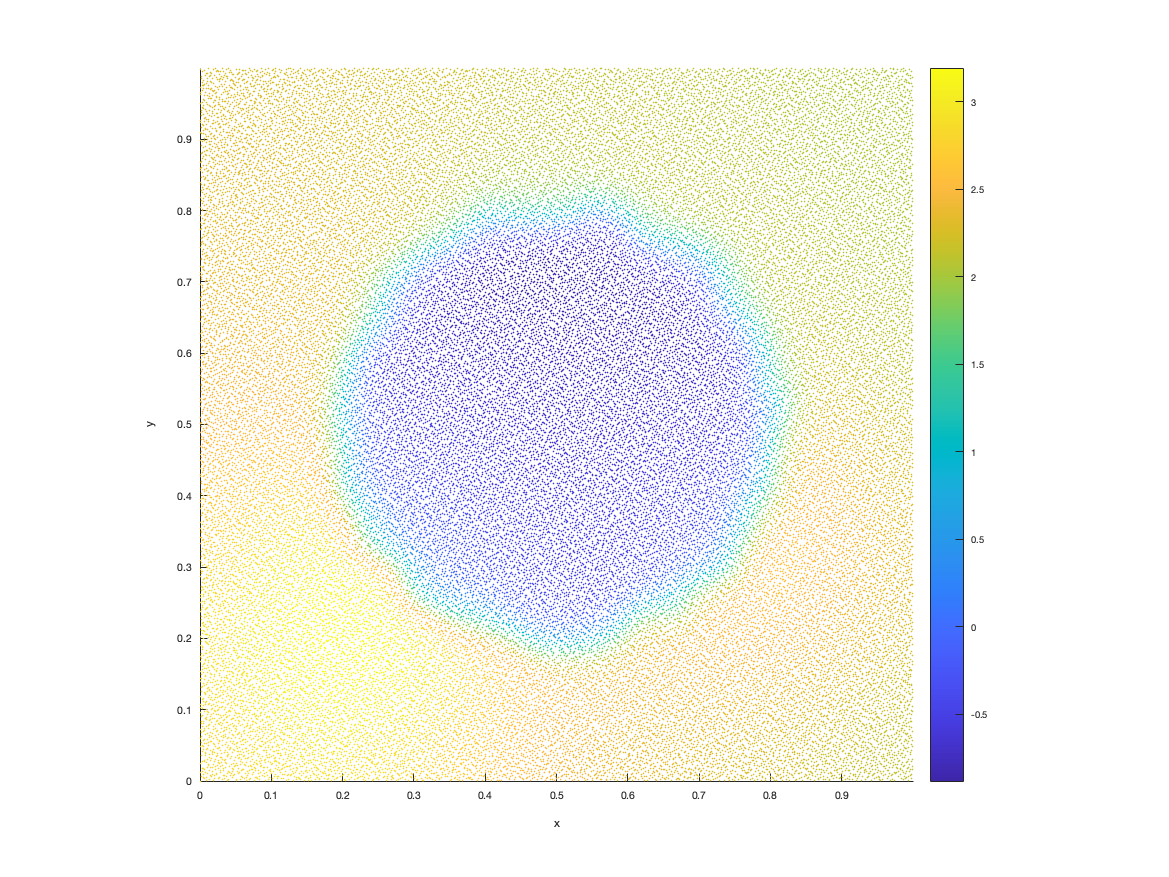} & \hspace{-0.9cm}\includegraphics[width=6cm]{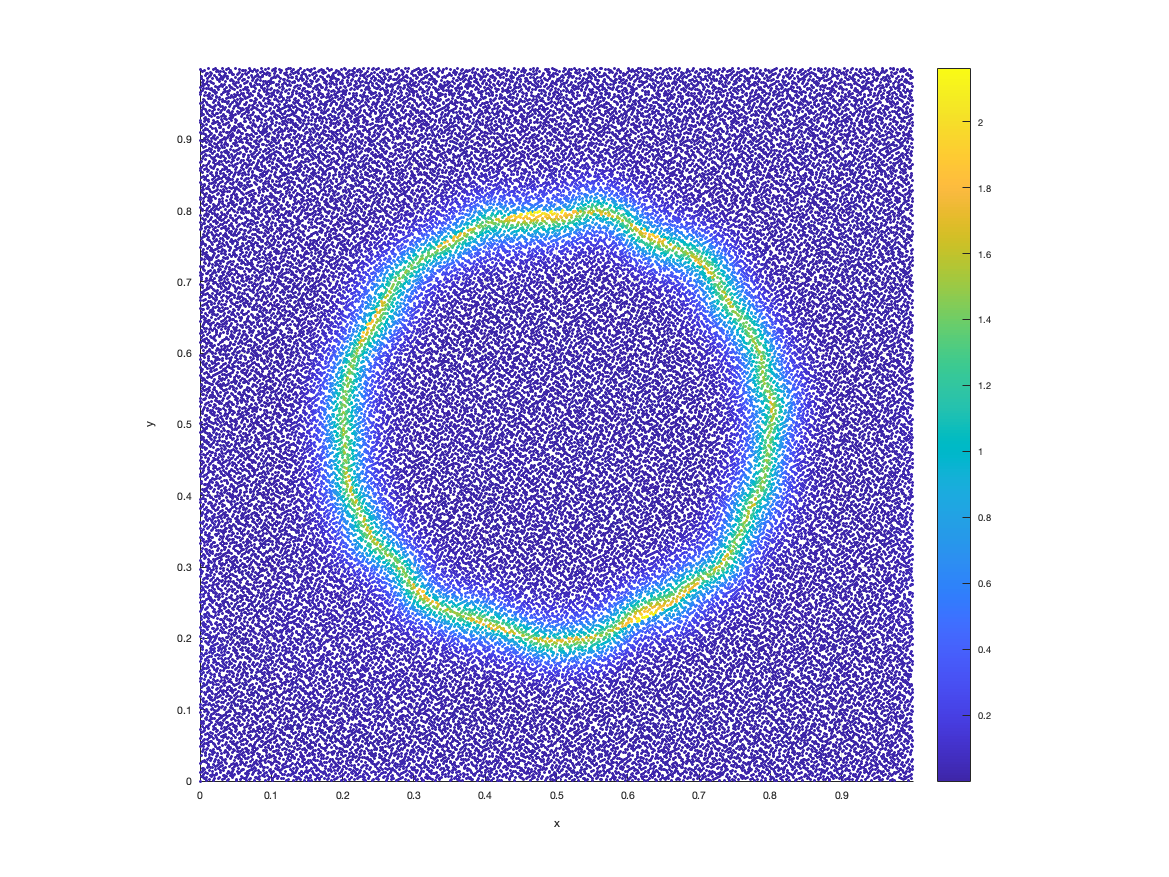}\\
\hspace{-1cm}\includegraphics[width=6cm]{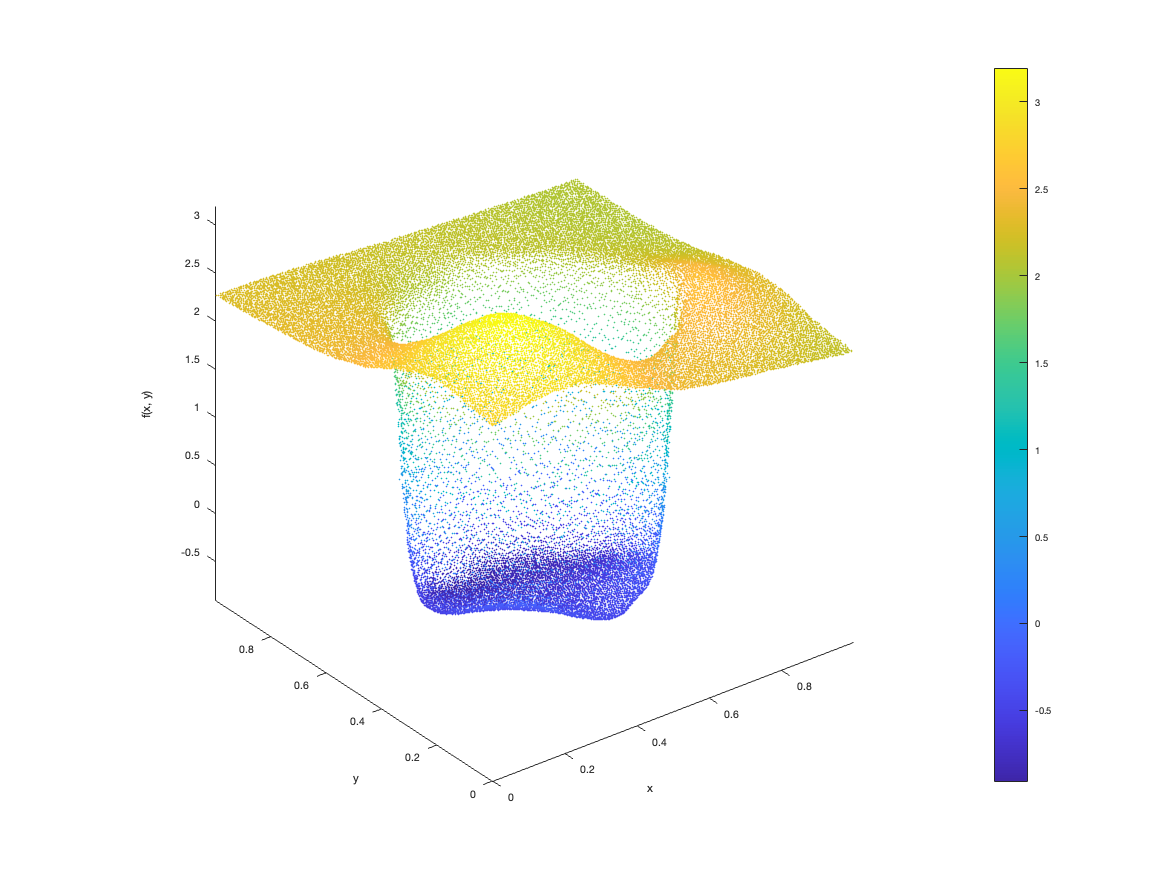} & 	\hspace{-0.9cm}\includegraphics[width=6cm]{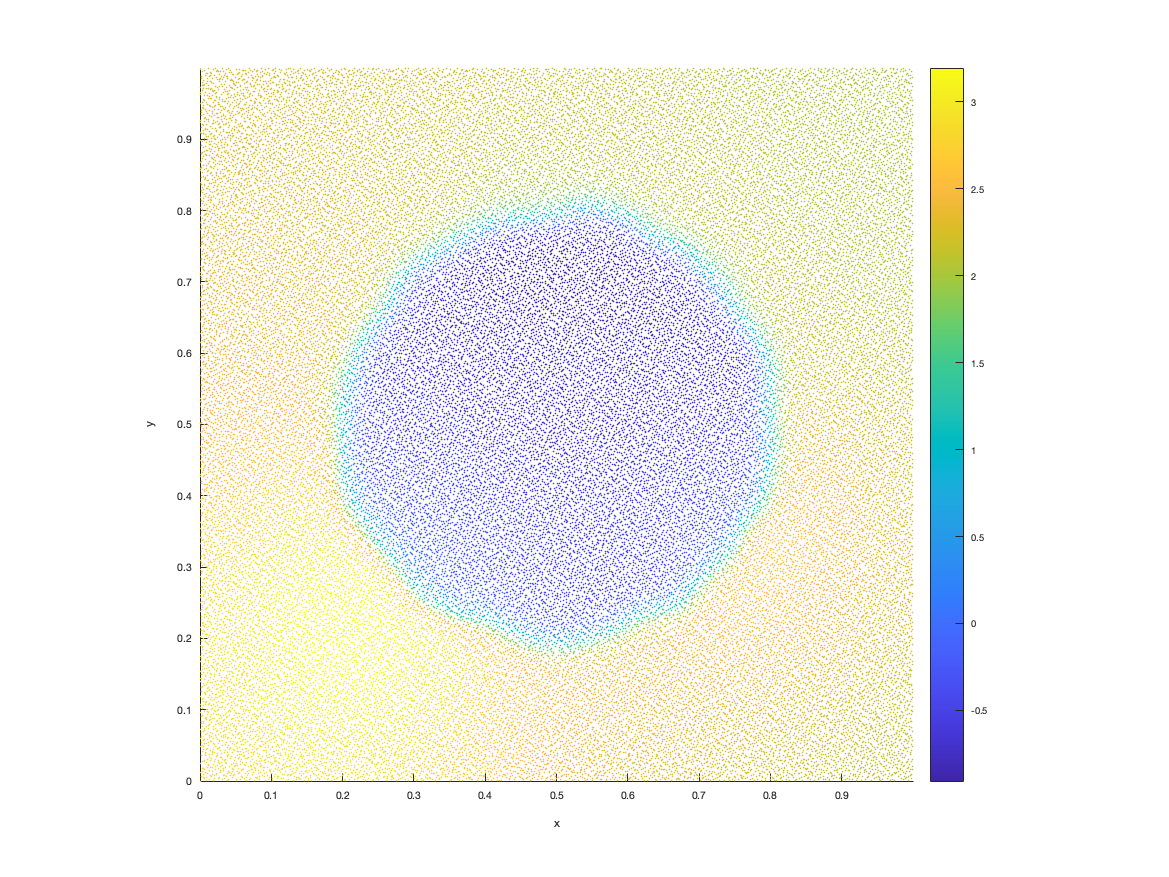} & \hspace{-0.9cm}\includegraphics[width=6cm]{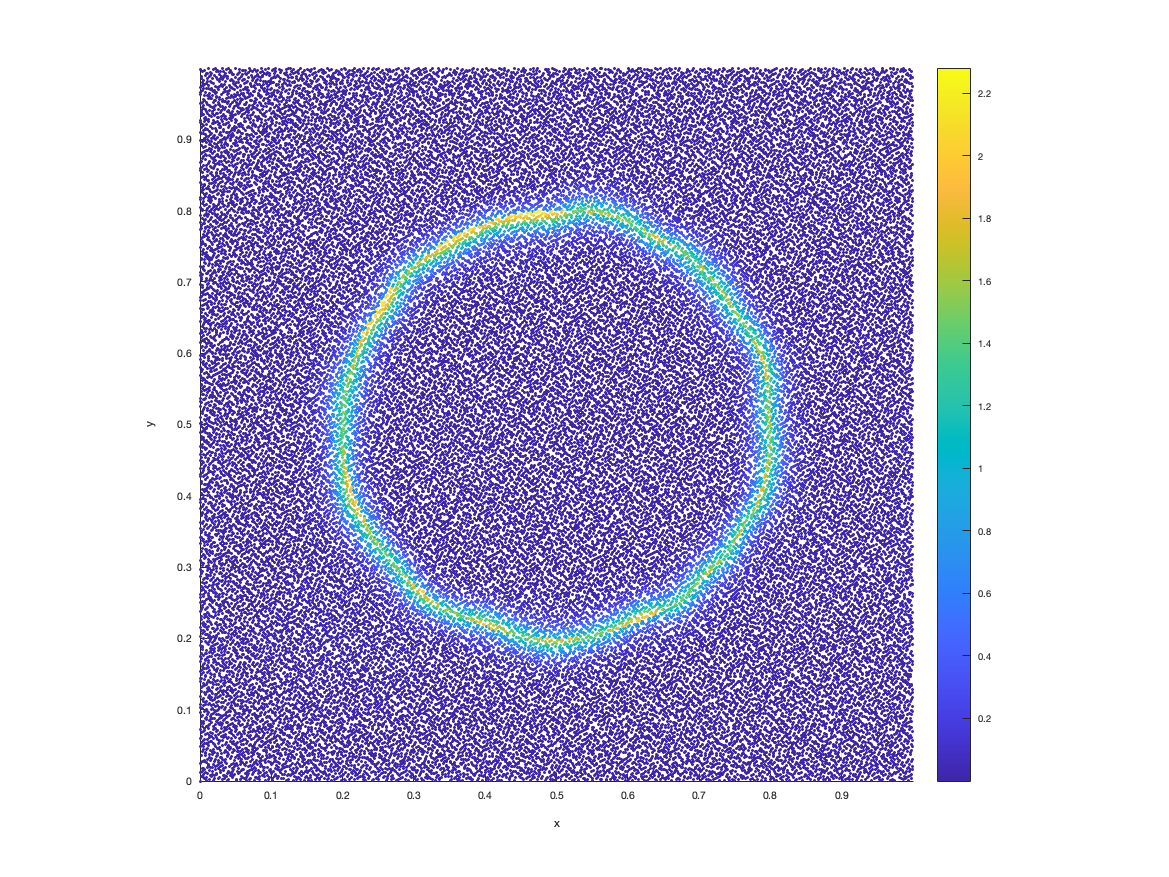}\\
		\end{tabular}
\end{center}
			\caption{Approximation of the piecewise smooth function \(f_1\) (defined in Eq.~\eqref{frankesdisc}). 
Each figure is arranged in three columns: the left column shows the surface reconstructed using either the linear RBF\(_{\text\\{M4}}\) method or its data-dependent counterpart NL-RBF\(_{\text\\{M4}}\). The central column displays the same approximation from a top (cenital) view. 
The right column depicts the pointwise absolute error over the evaluation grid.
The figure contains four rows: 
Row~1 shows the linear RBF\(_{\text\\{M4}}\) approximation using gridded points, 
Row~2 shows the data-dependent NL-RBF\(_{\text\\{M4}}\) approximation on the same grid, 
Row~3 presents the linear RBF\(_{\text\\{M4}}\) approximation computed from Halton points, 
and Row~4 displays the corresponding data-dependent NL-RBF\(_{\text\\{M4}}\) result.}

		\label{exp3_2D}
	\end{figure}

	\begin{figure}[htbp!]
\begin{center}
		\begin{tabular}{ccc}
	\hspace{-1cm}\includegraphics[width=6cm]{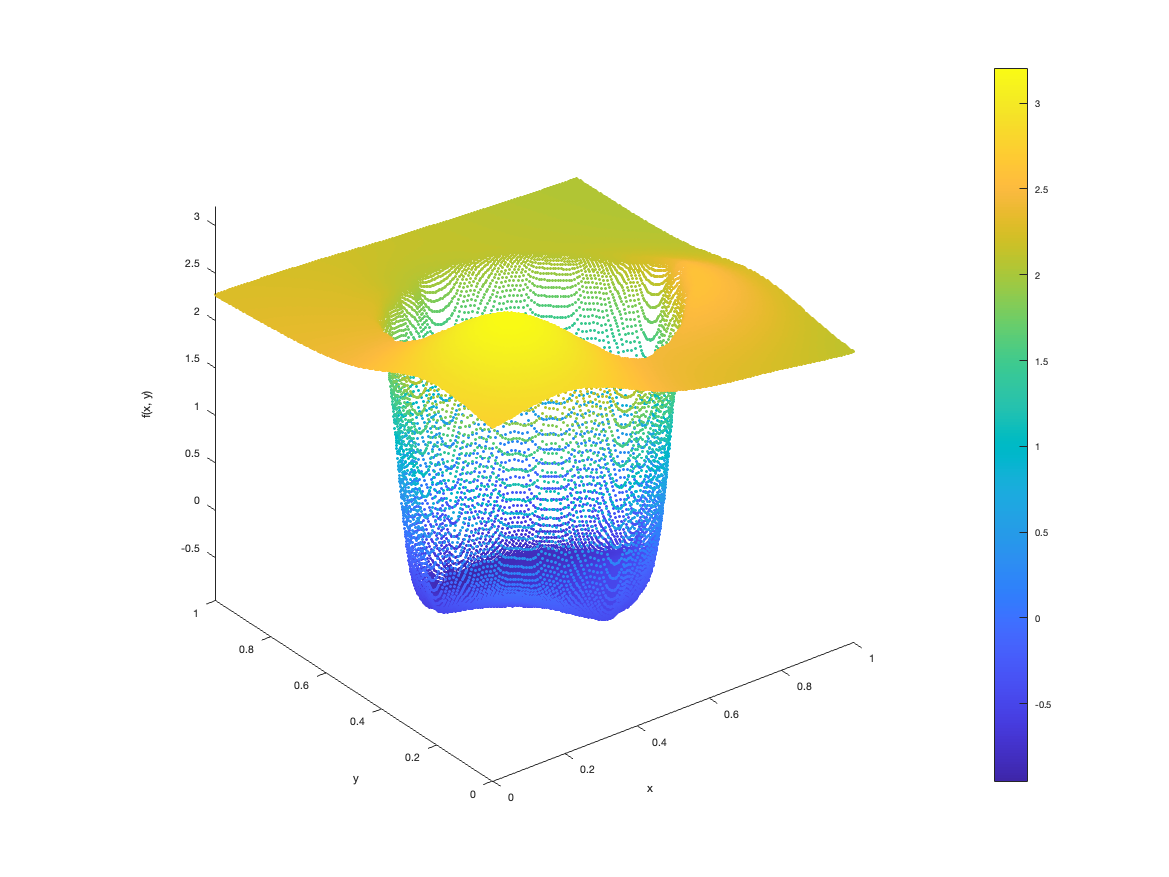} & 	\hspace{-0.9cm}\includegraphics[width=6cm]{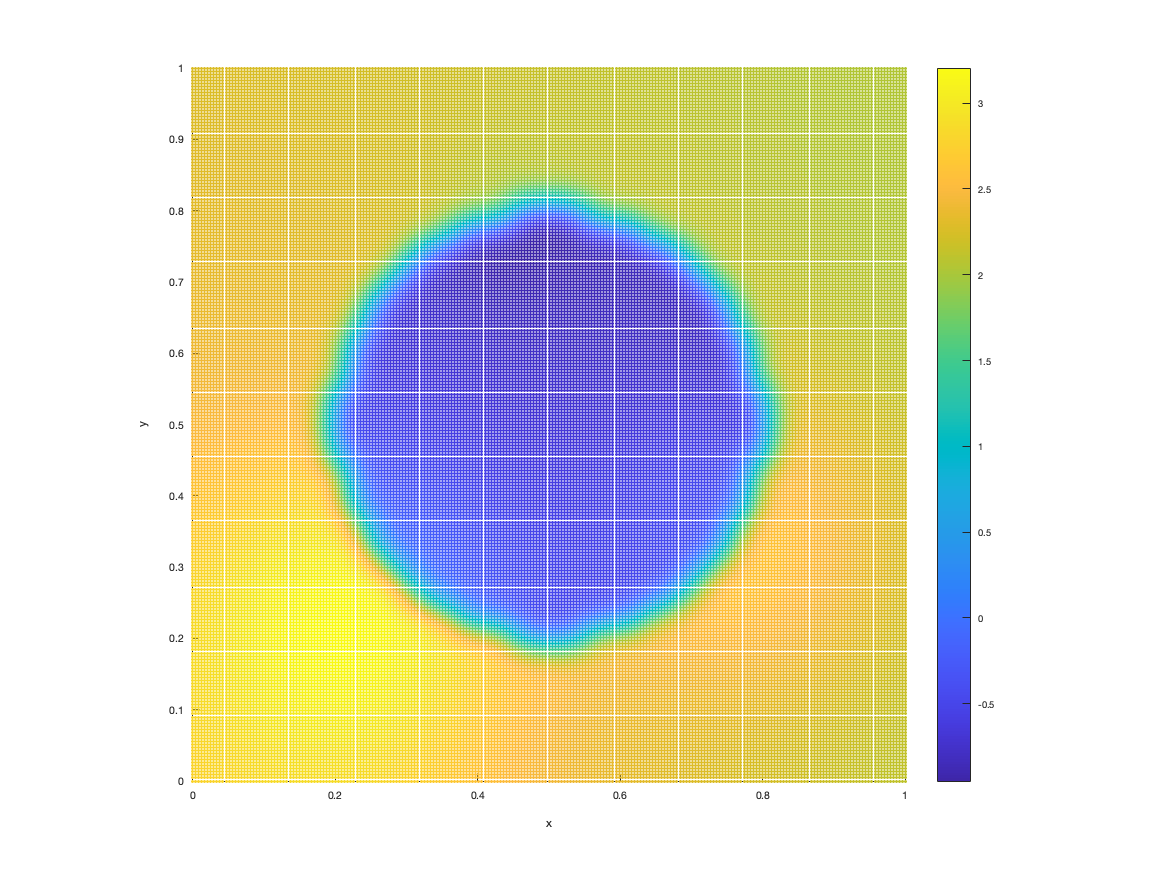} & \hspace{-0.9cm}\includegraphics[width=6cm]{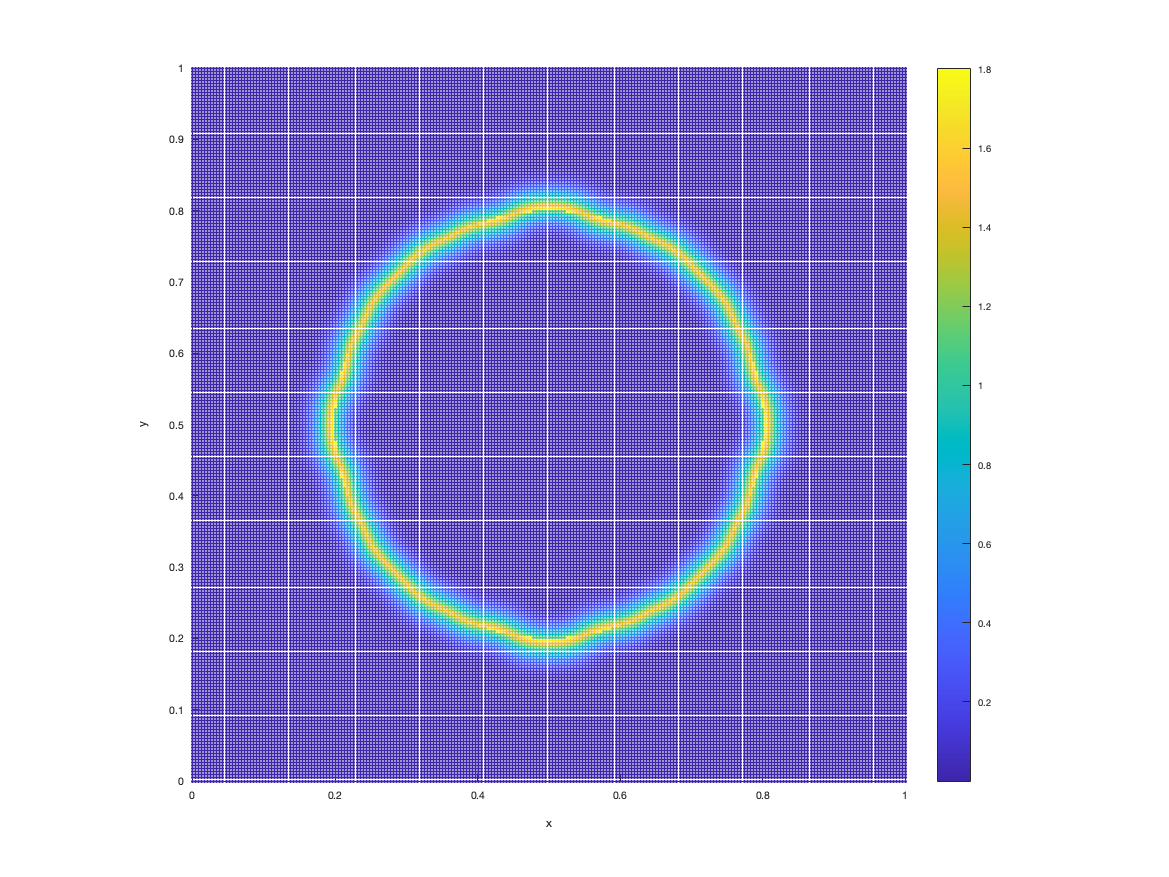}\\
\hspace{-1cm}\includegraphics[width=6cm]{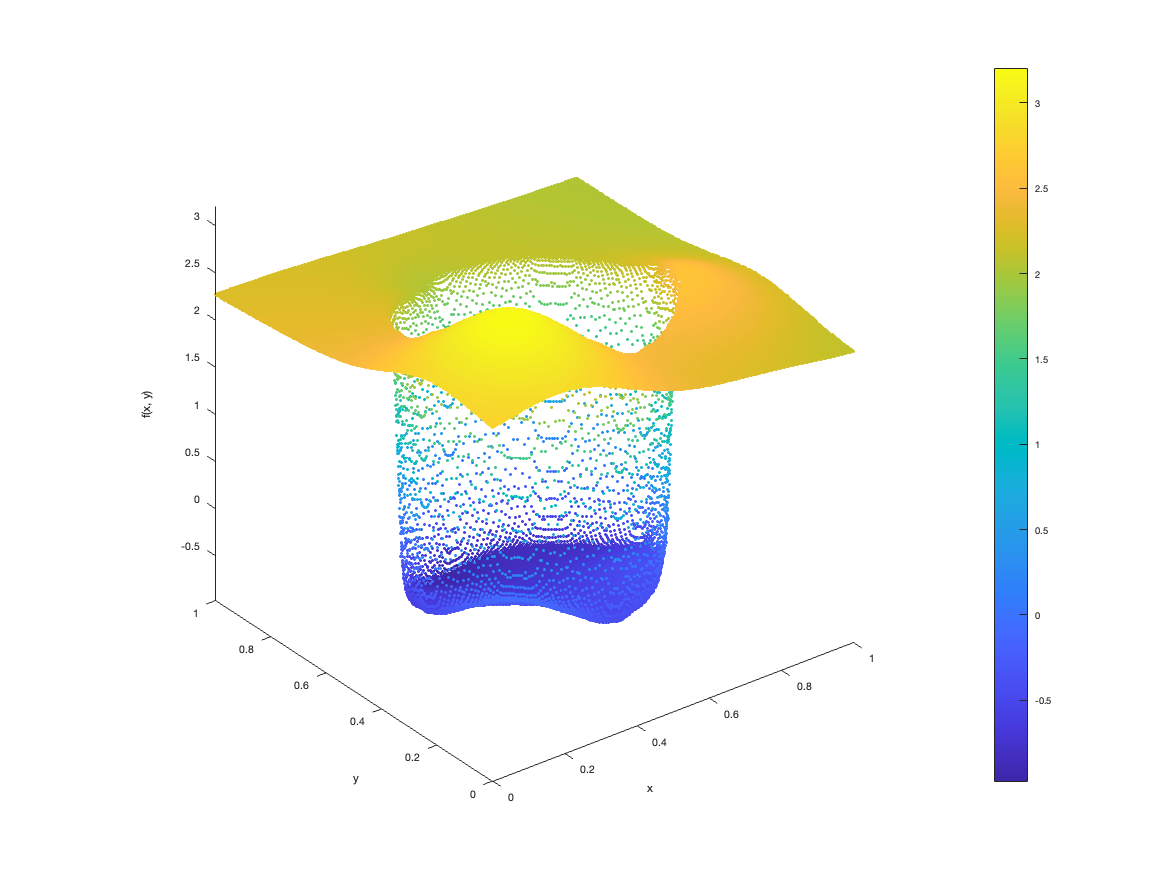} & 	\hspace{-0.9cm}\includegraphics[width=6cm]{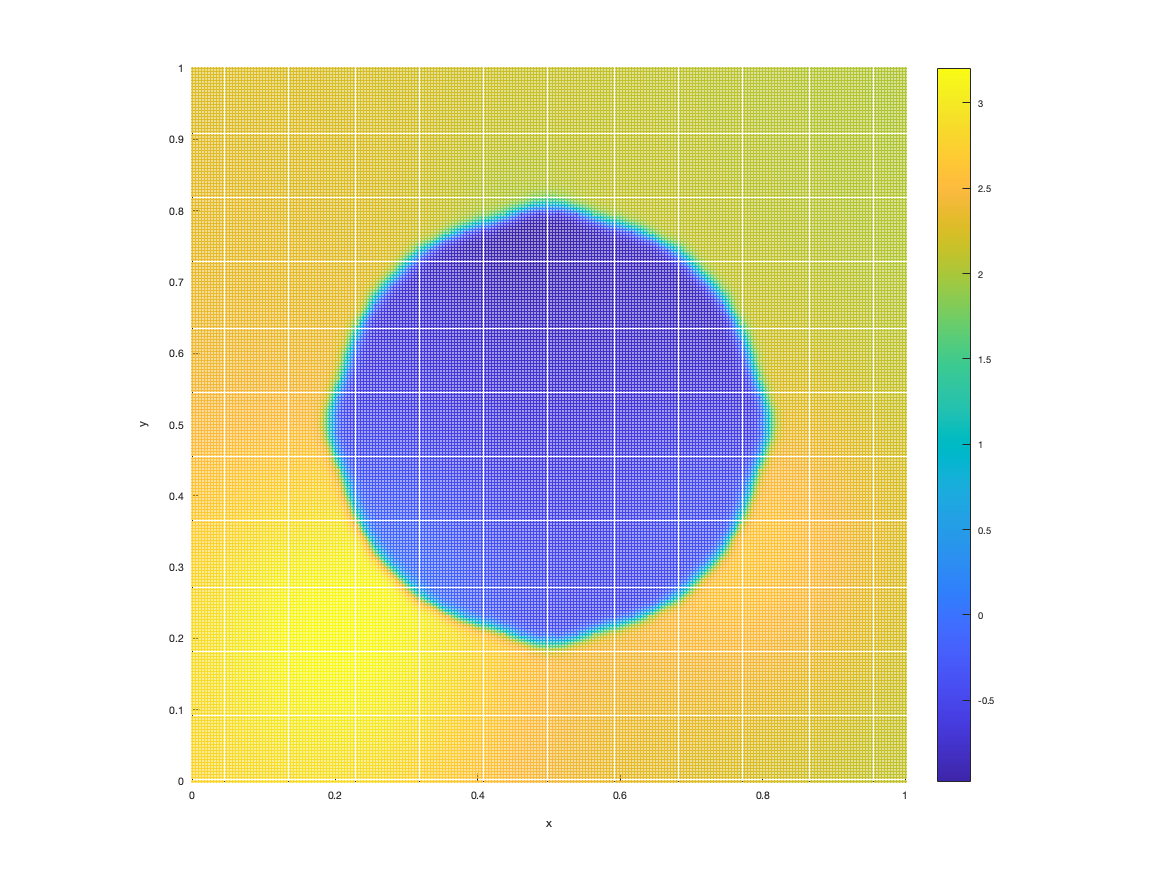} & \hspace{-0.9cm}\includegraphics[width=6cm]{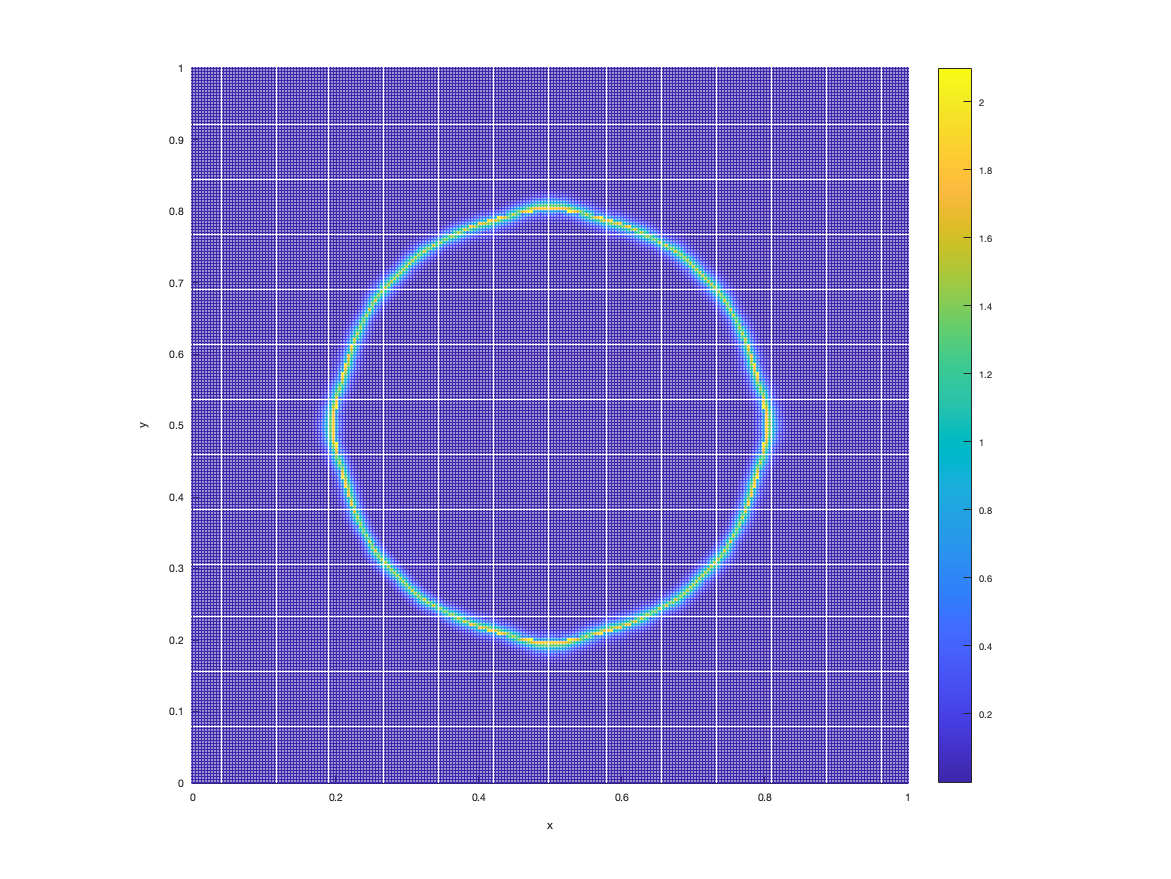}\\
	\hspace{-1cm}\includegraphics[width=6cm]{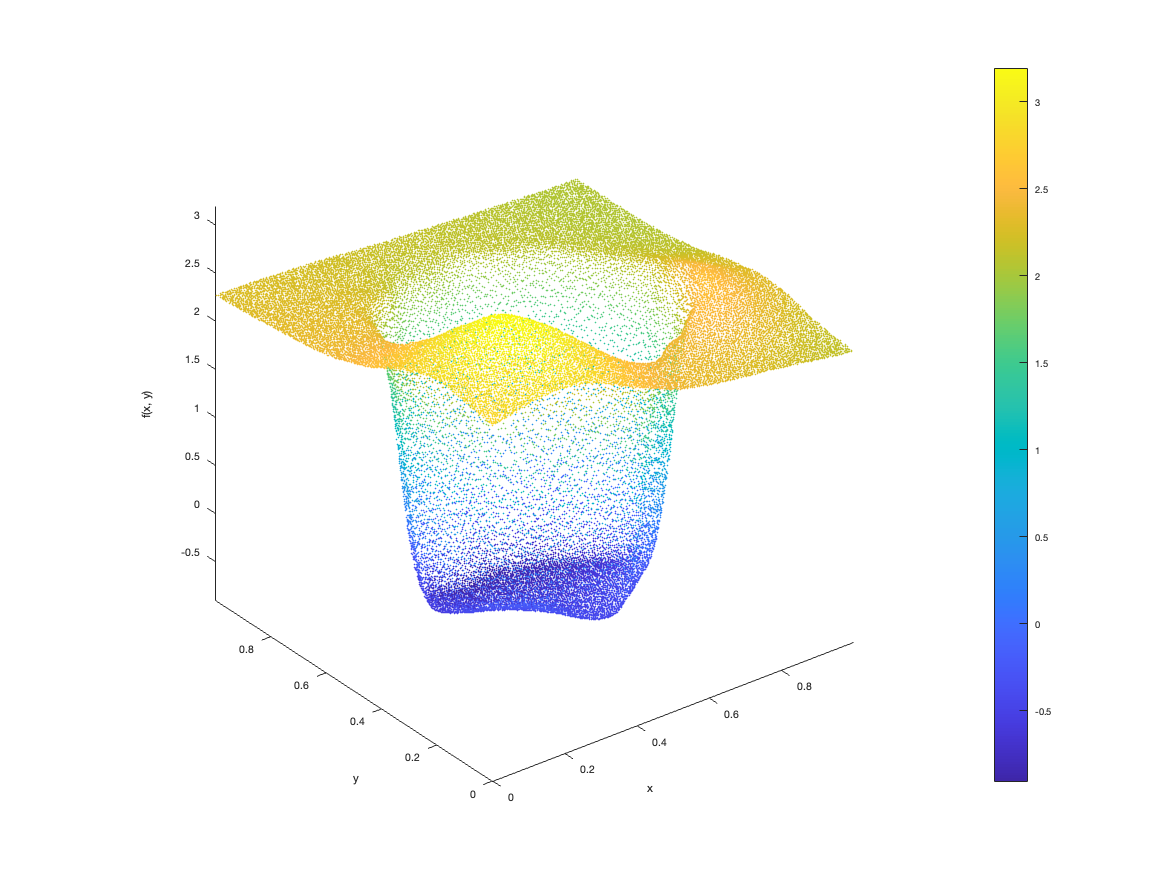} & 	\hspace{-0.9cm}\includegraphics[width=6cm]{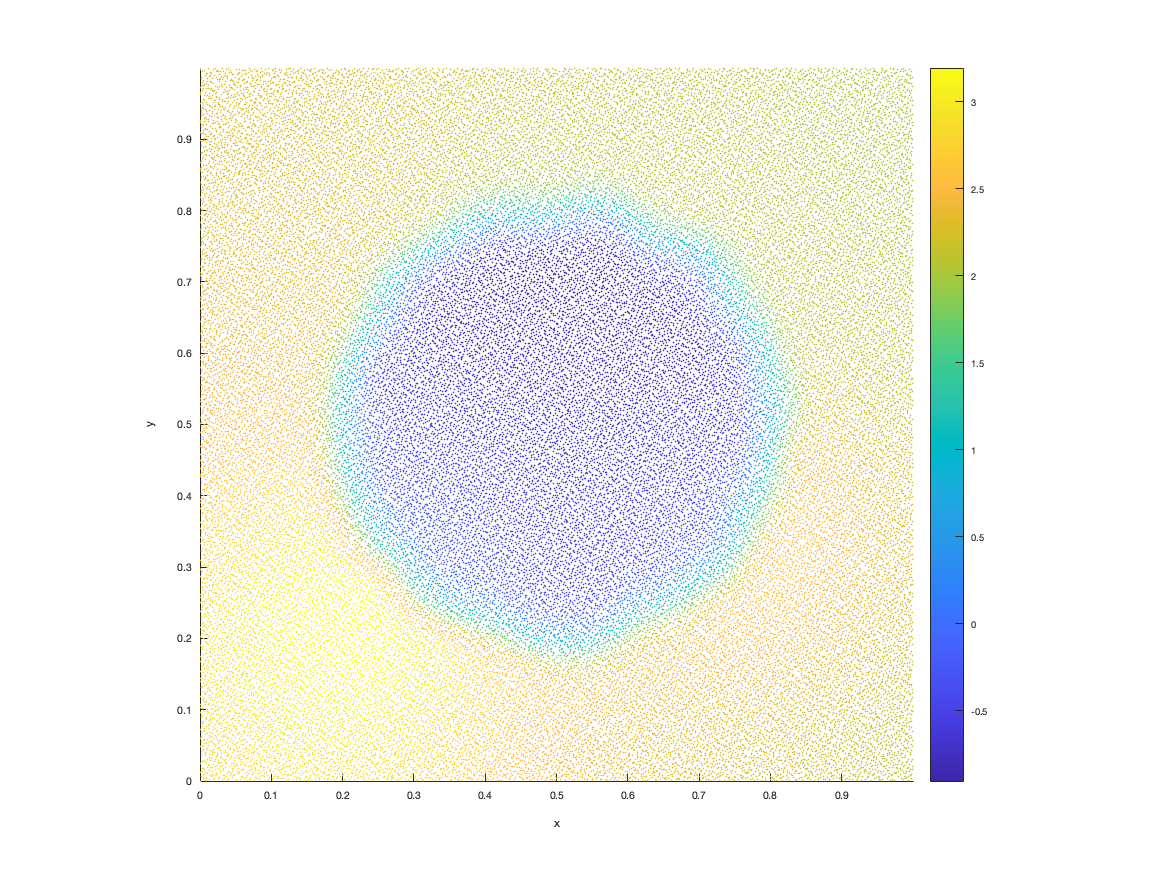} & \hspace{-0.9cm}\includegraphics[width=6cm]{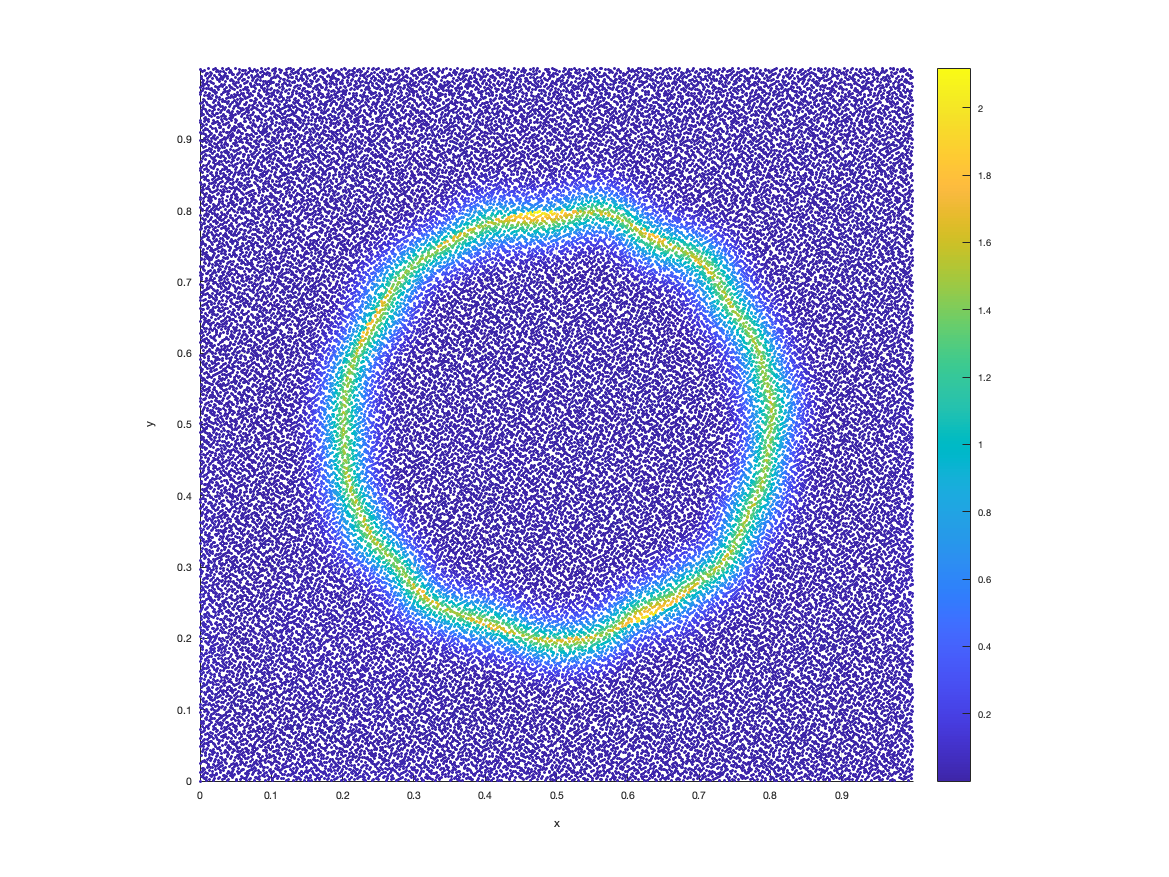}\\
\hspace{-1cm}\includegraphics[width=6cm]{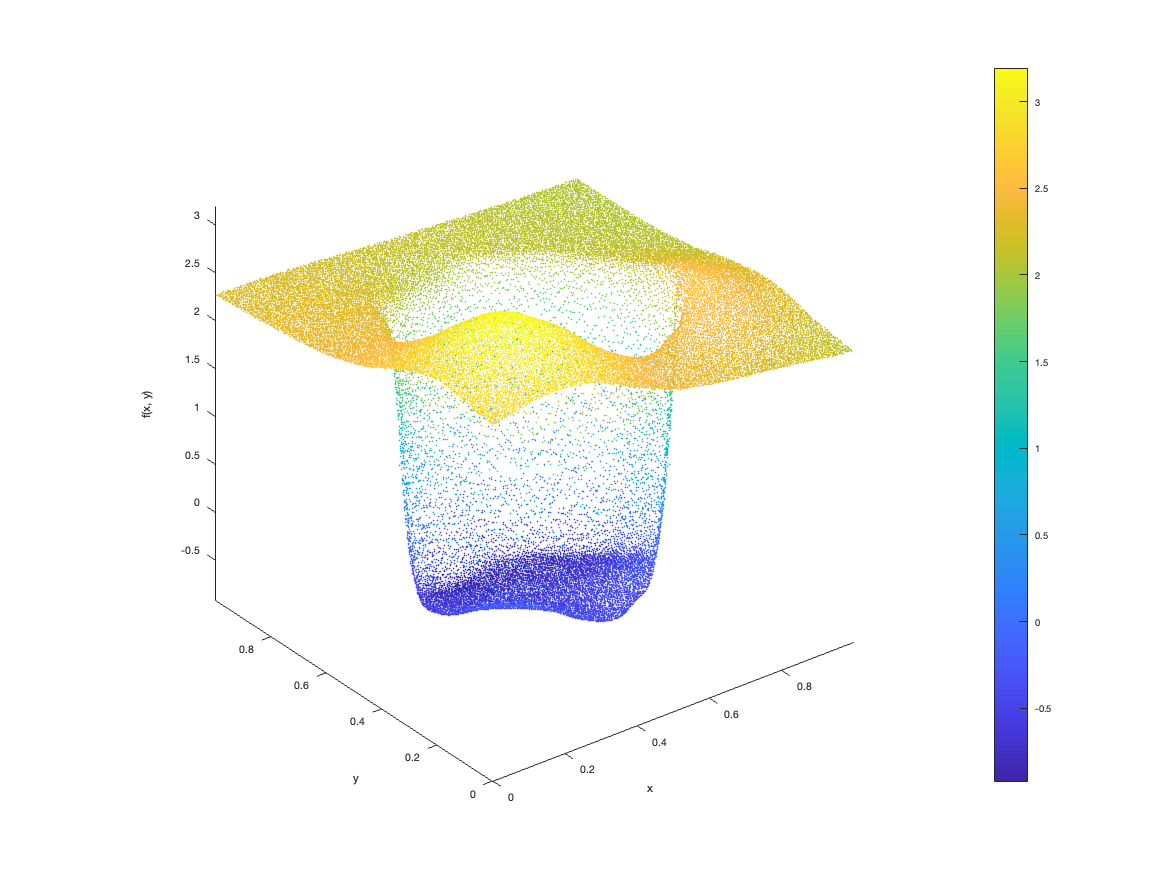} & 	\hspace{-0.9cm}\includegraphics[width=6cm]{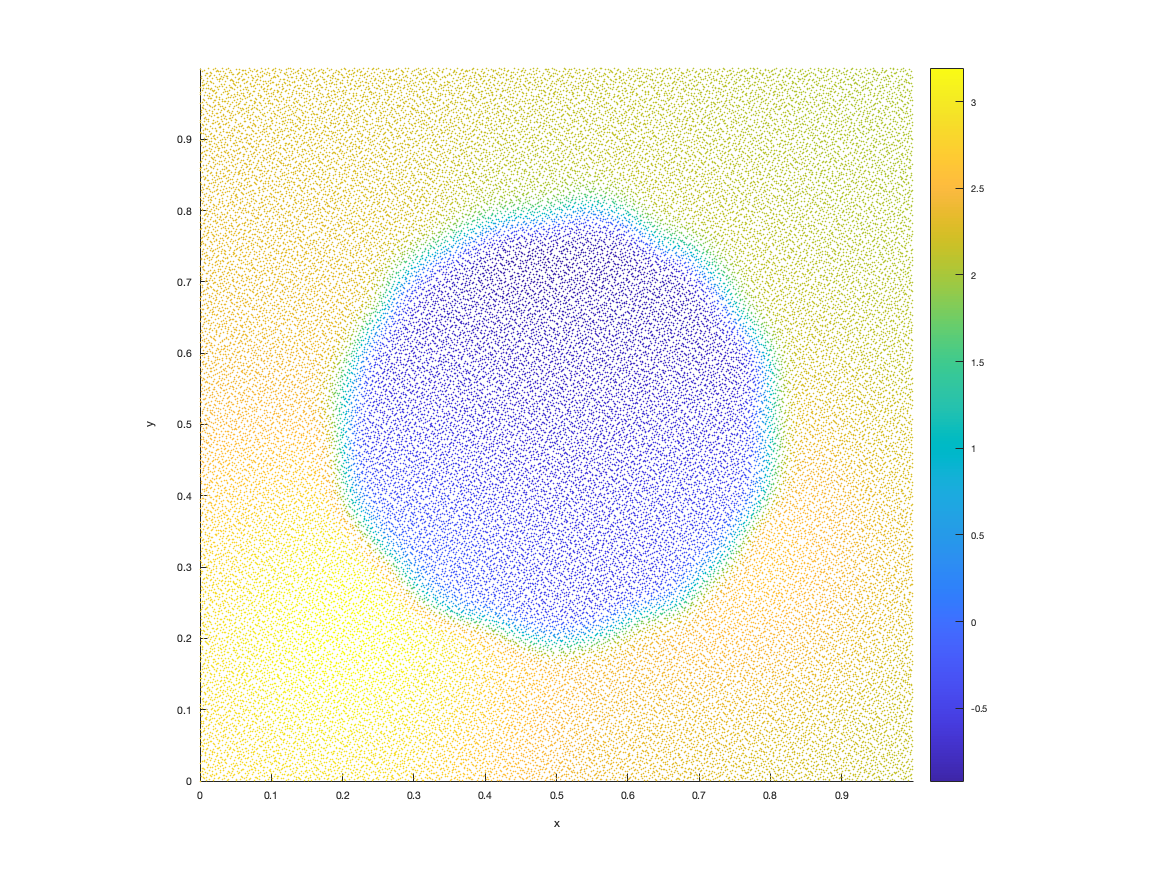} & \hspace{-0.9cm}\includegraphics[width=6cm]{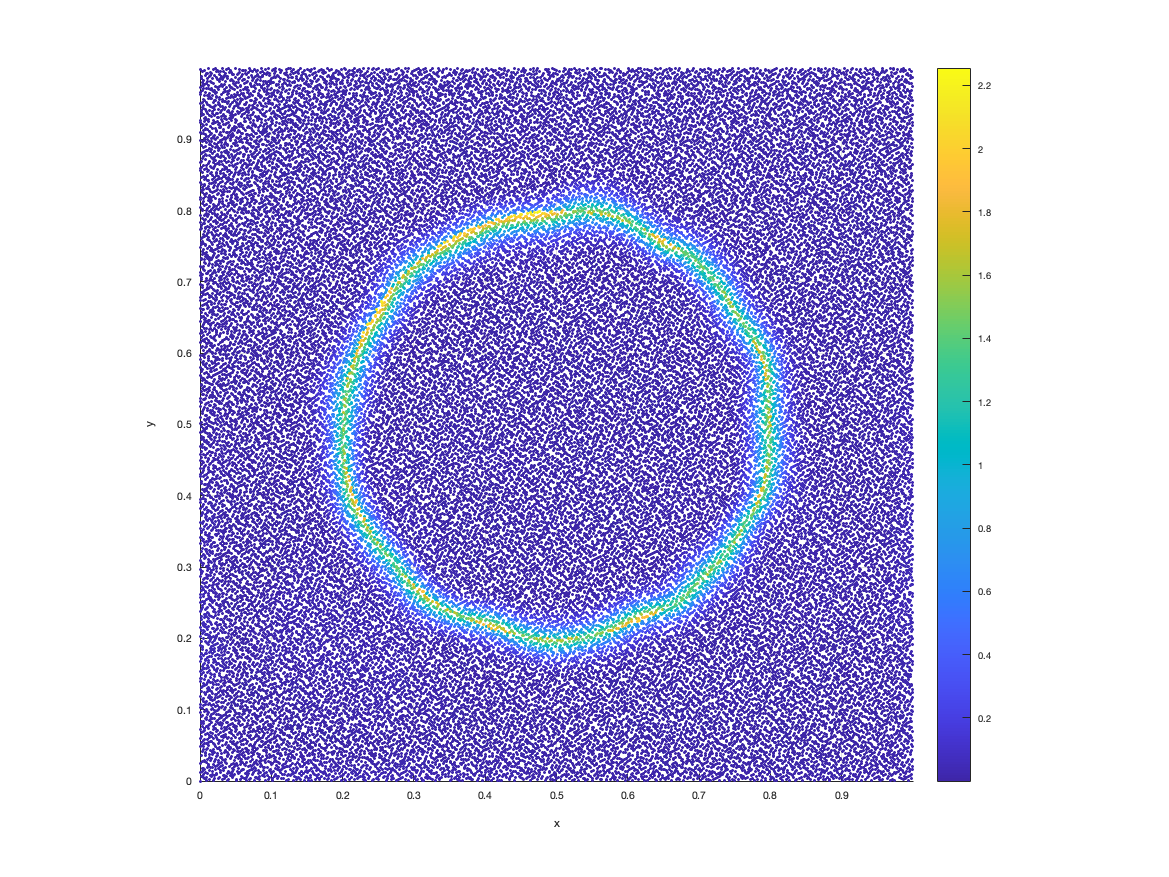}\\
		\end{tabular}
\end{center}
			\caption{Approximation of the piecewise smooth function \(f_1\) (defined in Eq.~\eqref{frankesdisc}). 
Each figure is arranged in three columns: the left column shows the surface reconstructed using either the linear RBF\(_{\text{{W2}}}\) method or its data-dependent counterpart NL-RBF\(_{\text{{W2}}}\). The central column displays the same approximation from a top (cenital) view. 
The right column depicts the pointwise absolute error over the evaluation grid.
The figure contains four rows: 
Row~1 shows the linear RBF\(_{\text{{W2}}}\) approximation using gridded points, 
Row~2 shows the data-dependent NL-RBF\(_{\text{{W2}}}\) approximation on the same grid, 
Row~3 presents the linear RBF\(_{\text{{W2}}}\) approximation computed from Halton points, 
and Row~4 displays the corresponding data-dependent NL-RBF\(_{\text{{W2}}}\) result.}
		\label{exp4_2D}
	\end{figure}

	\begin{figure}[htbp!]
\begin{center}
		\begin{tabular}{ccc}
	\hspace{-1cm}\includegraphics[width=6cm]{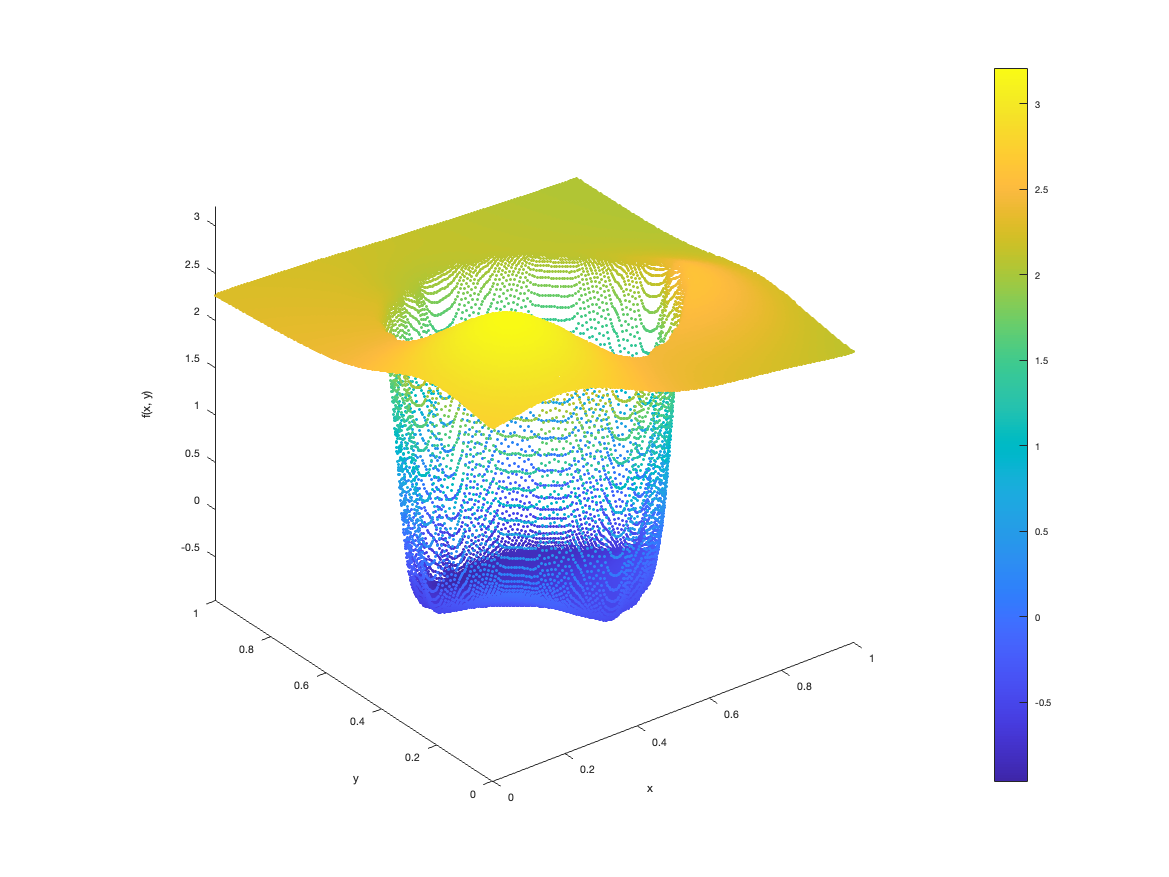} & 	\hspace{-0.9cm}\includegraphics[width=6cm]{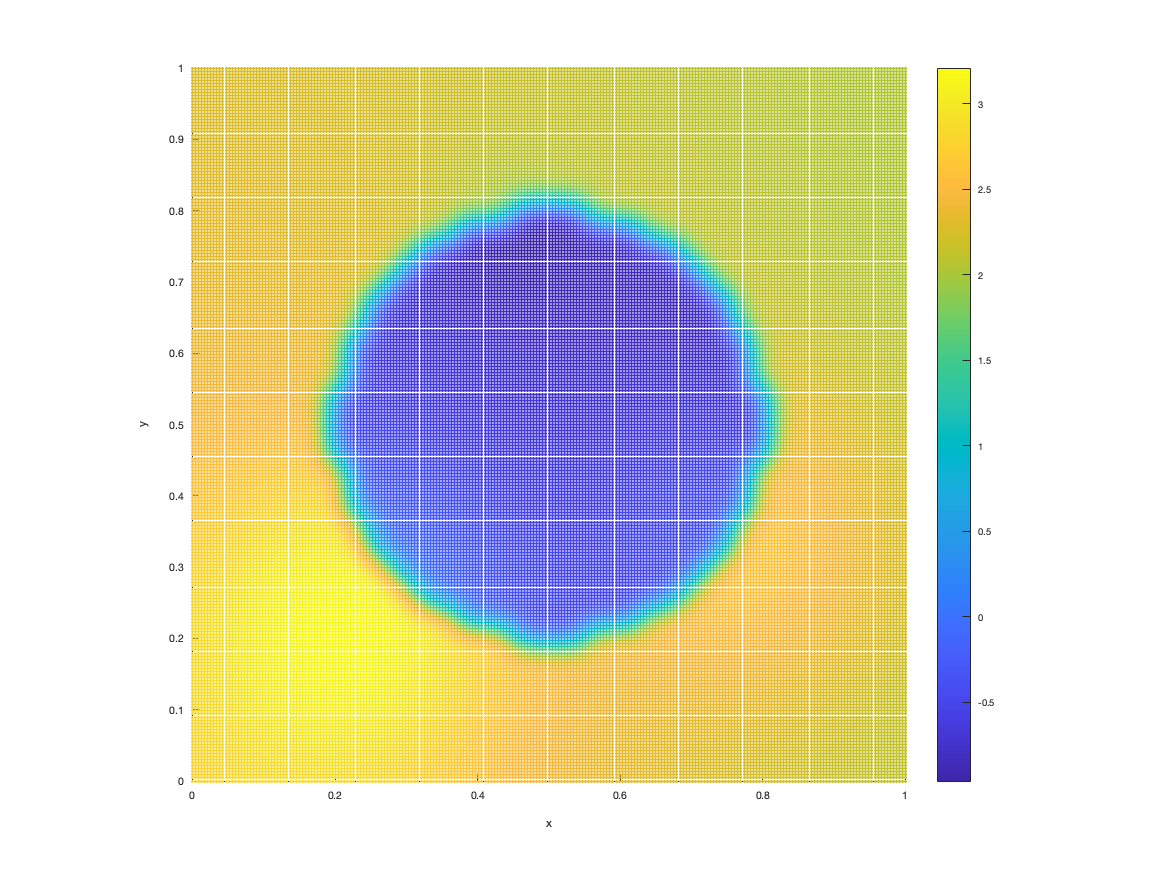} & \hspace{-0.9cm}\includegraphics[width=6cm]{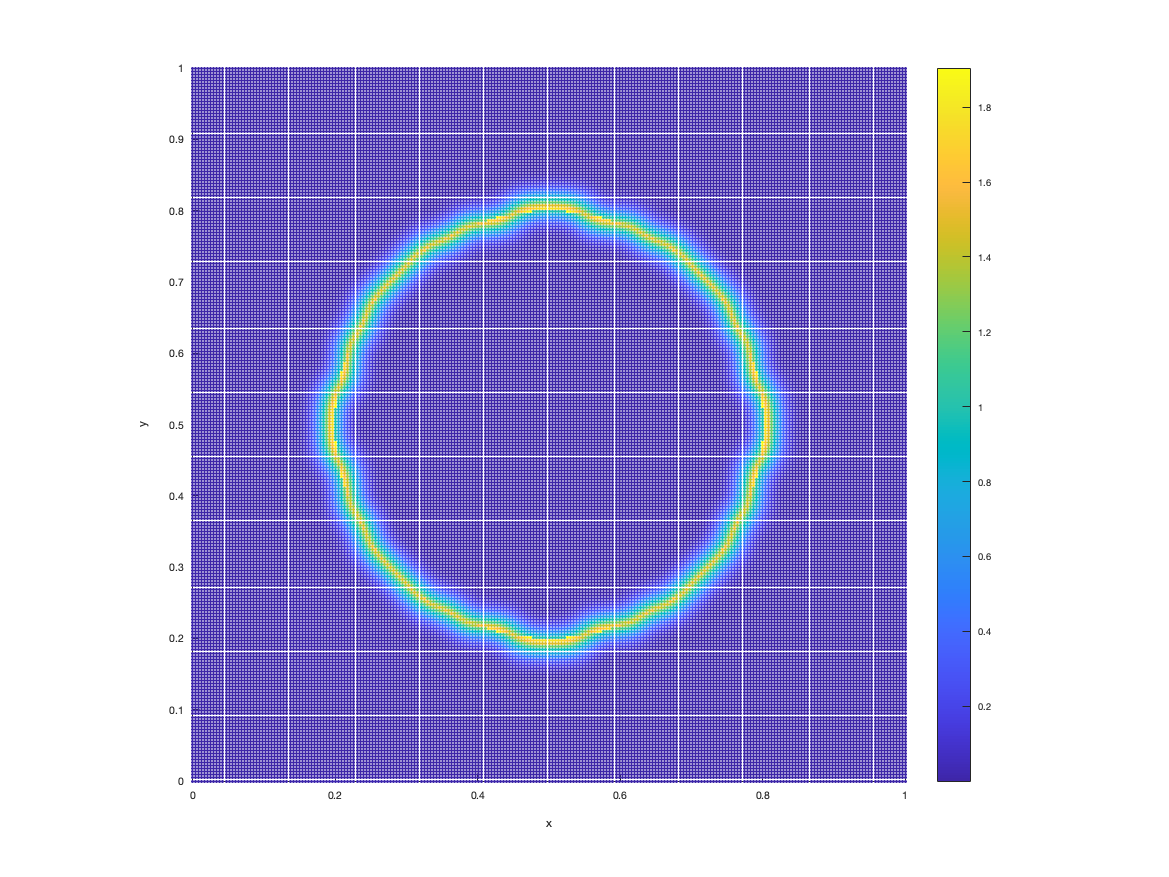}\\
\hspace{-1cm}\includegraphics[width=6cm]{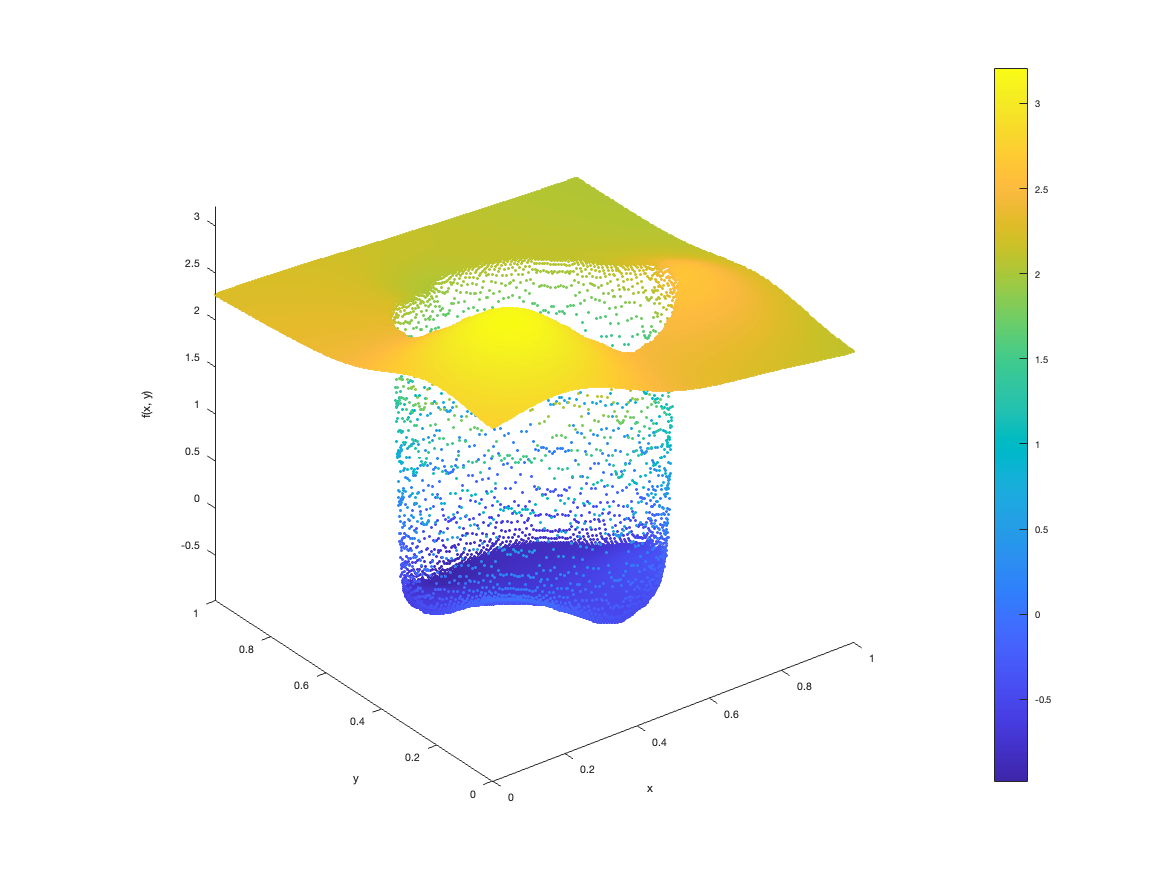} & 	\hspace{-0.9cm}\includegraphics[width=6cm]{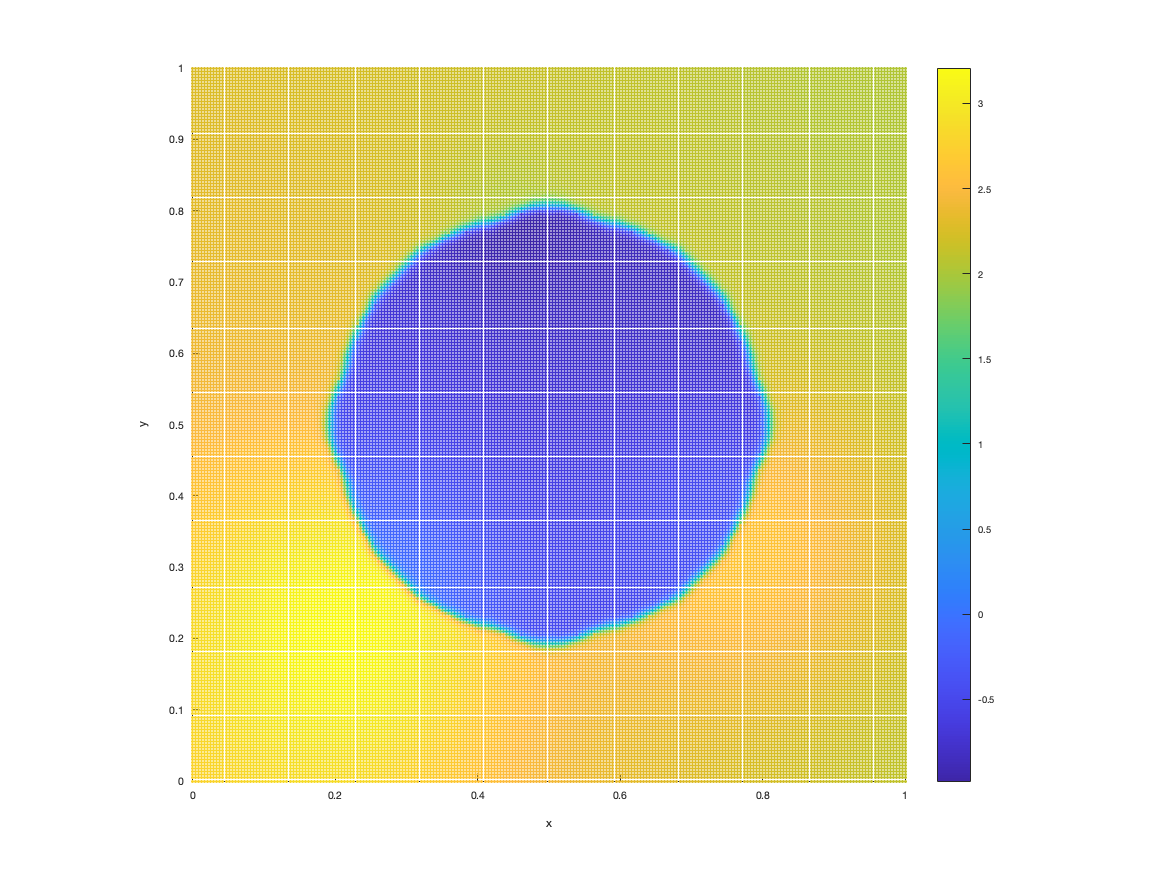} & \hspace{-0.9cm}\includegraphics[width=6cm]{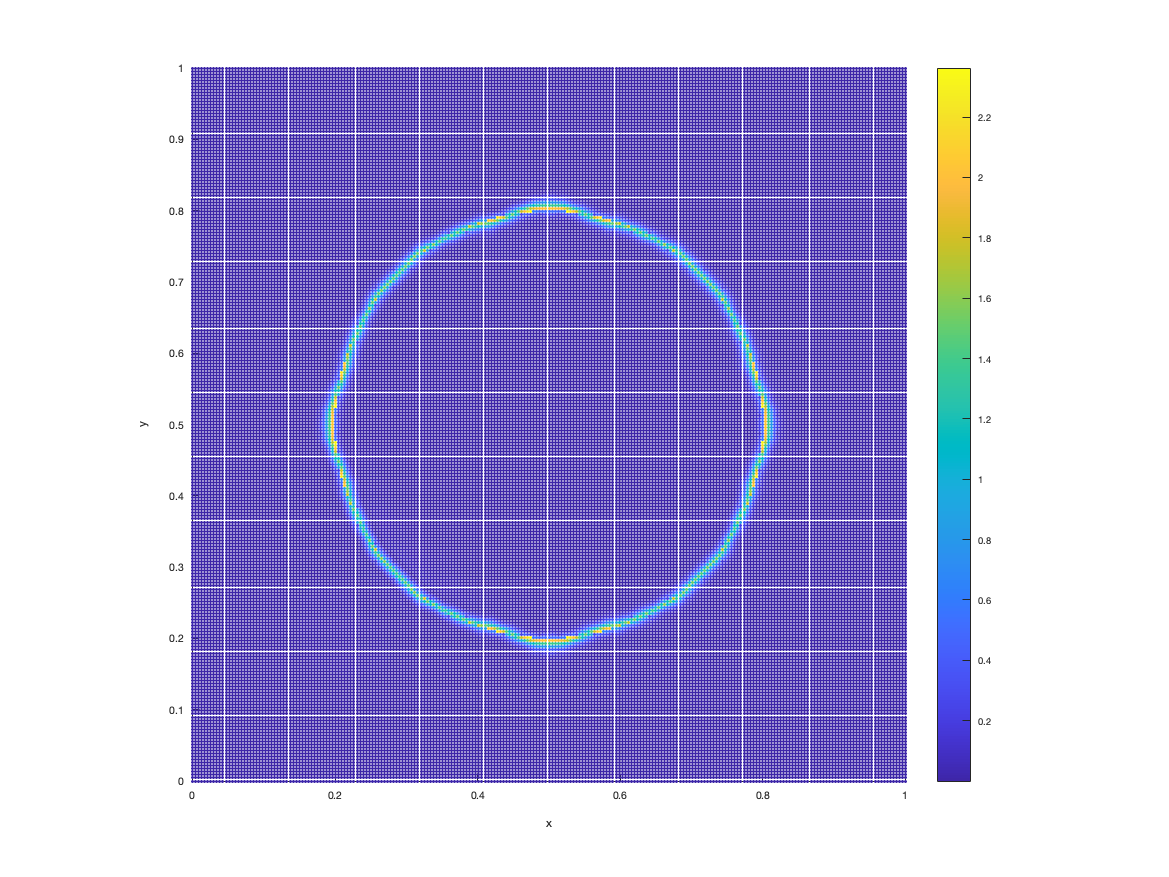}\\
	\hspace{-1cm}\includegraphics[width=6cm]{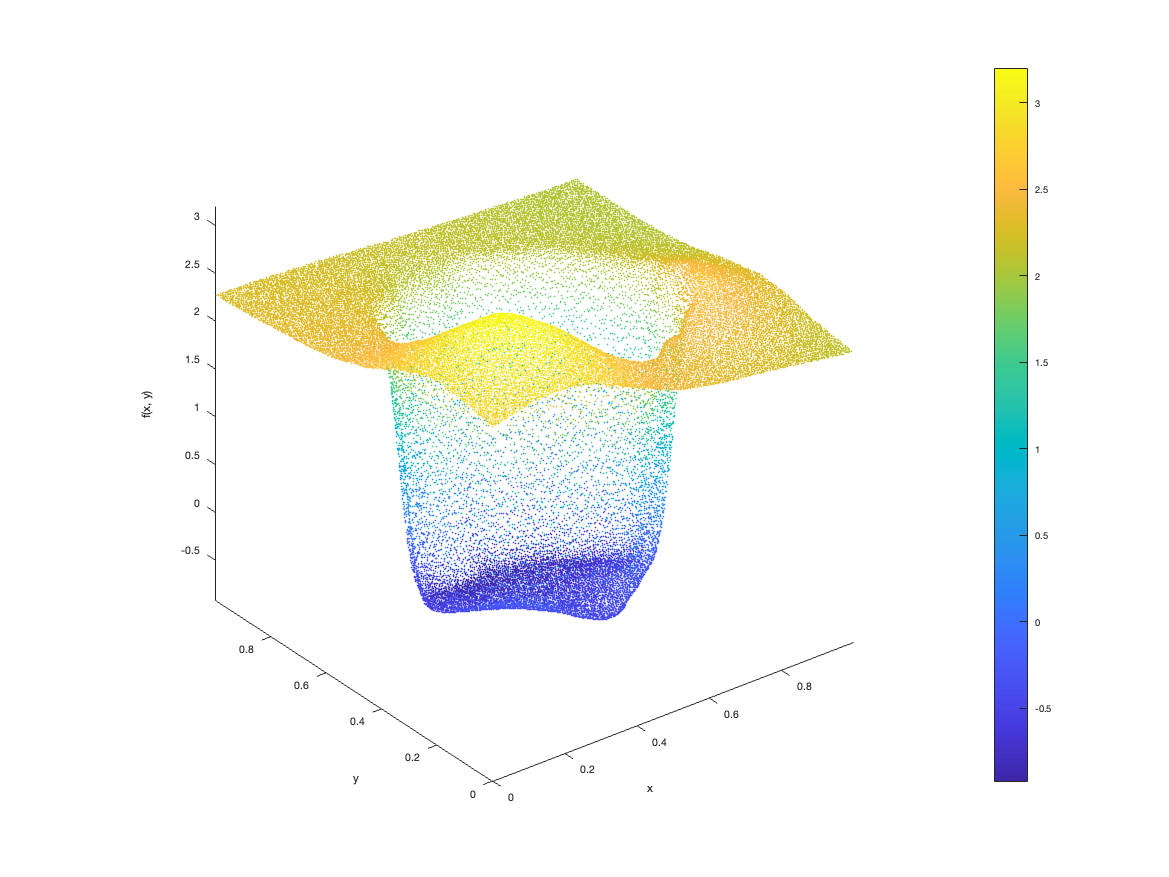} & 	\hspace{-0.9cm}\includegraphics[width=6cm]{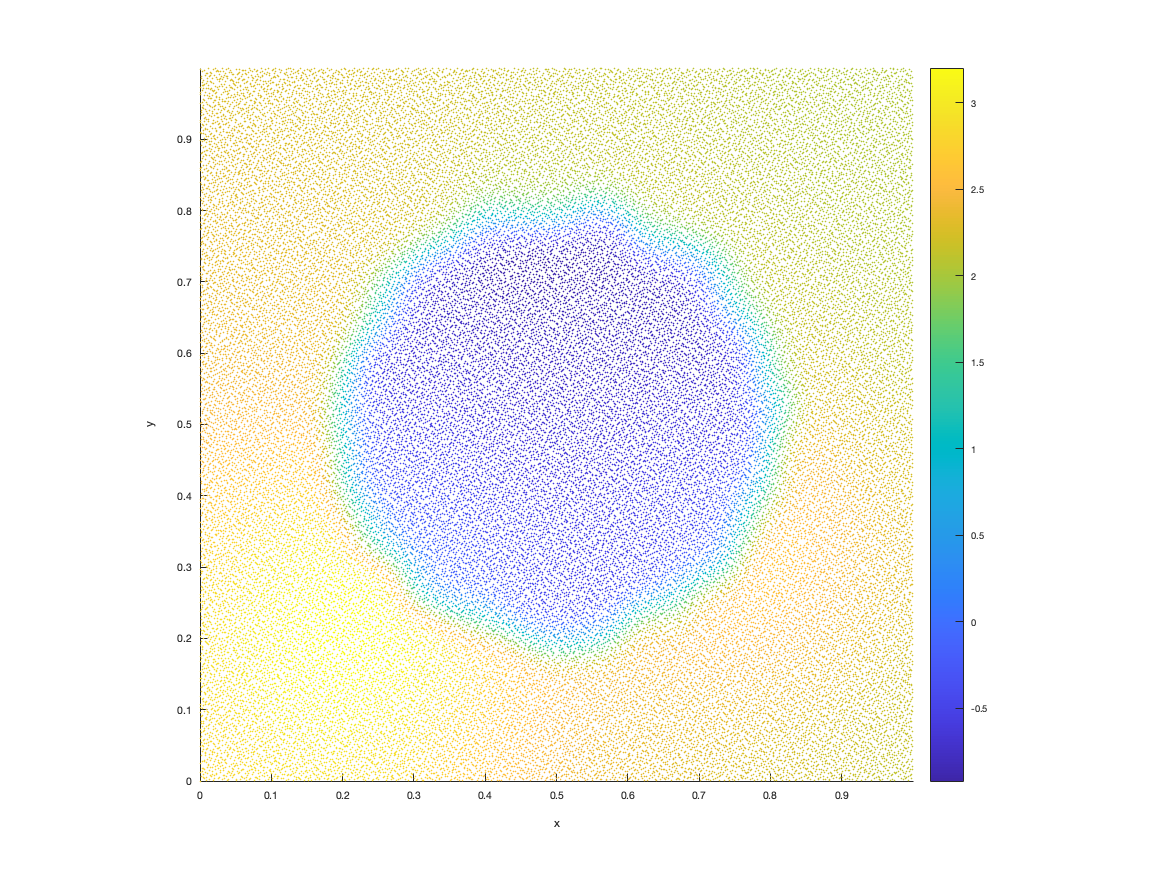} & \hspace{-0.9cm}\includegraphics[width=6cm]{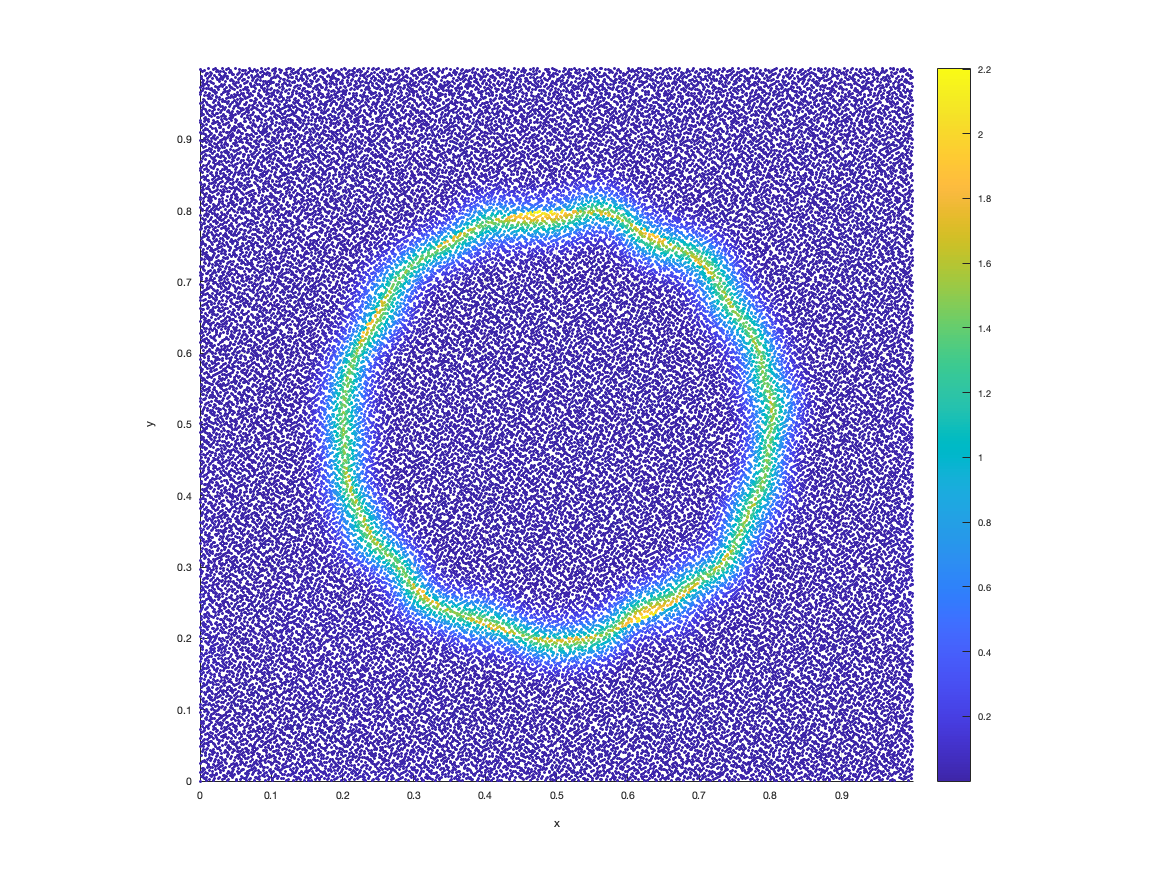}\\
\hspace{-1cm}\includegraphics[width=6cm]{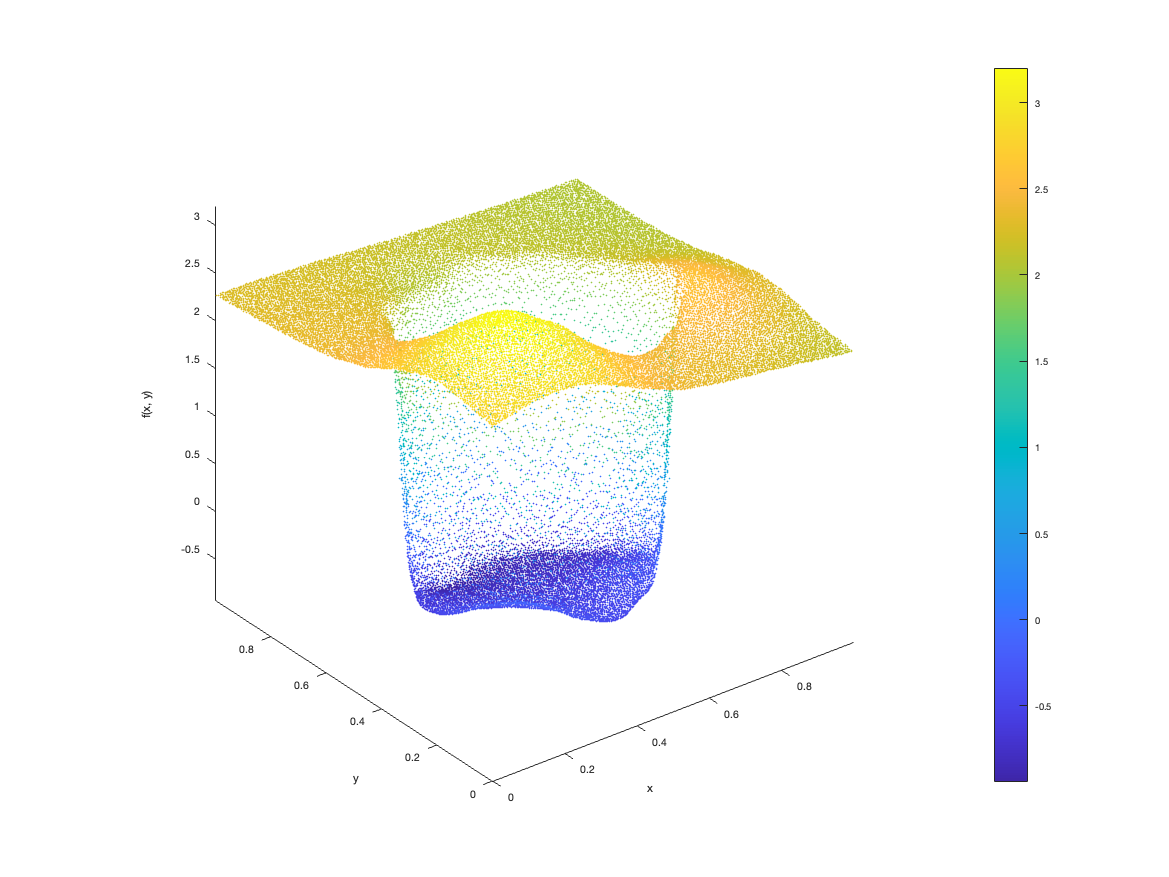} & 	\hspace{-0.9cm}\includegraphics[width=6cm]{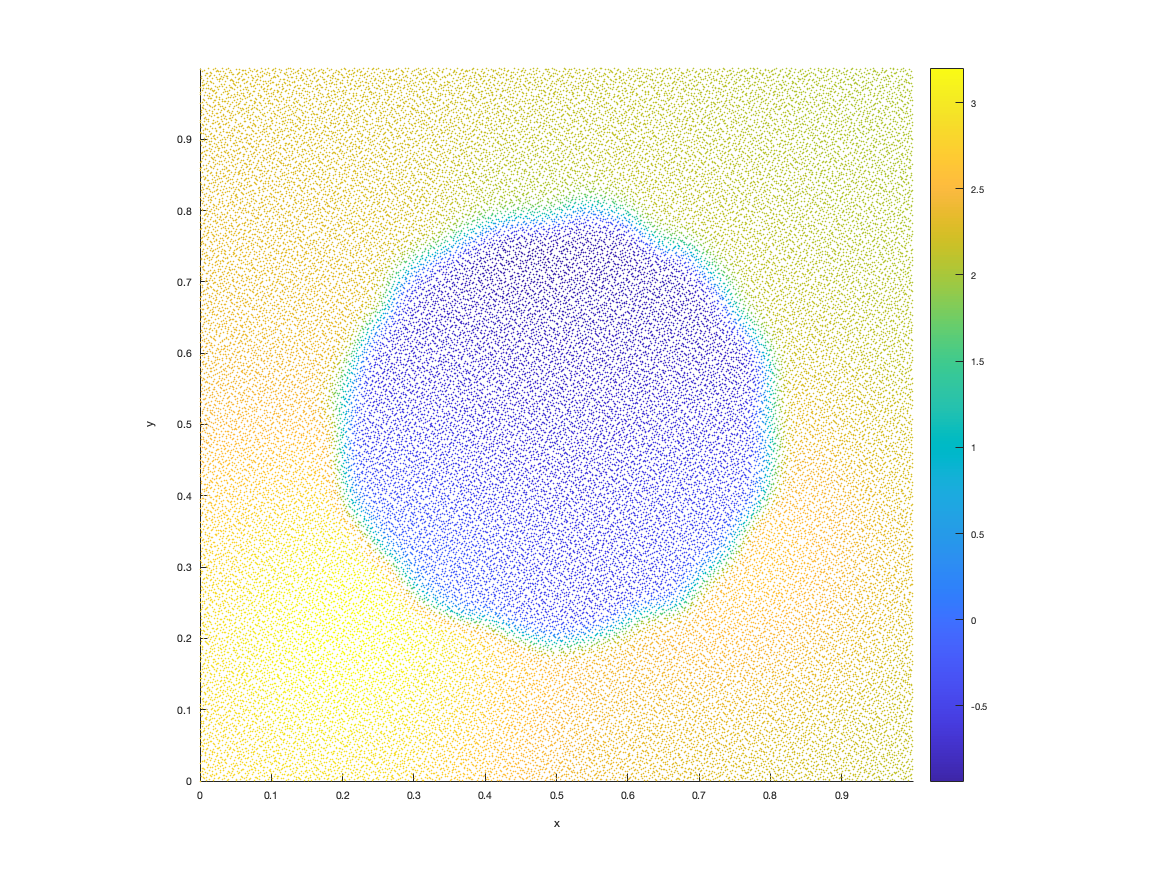} & \hspace{-0.9cm}\includegraphics[width=6cm]{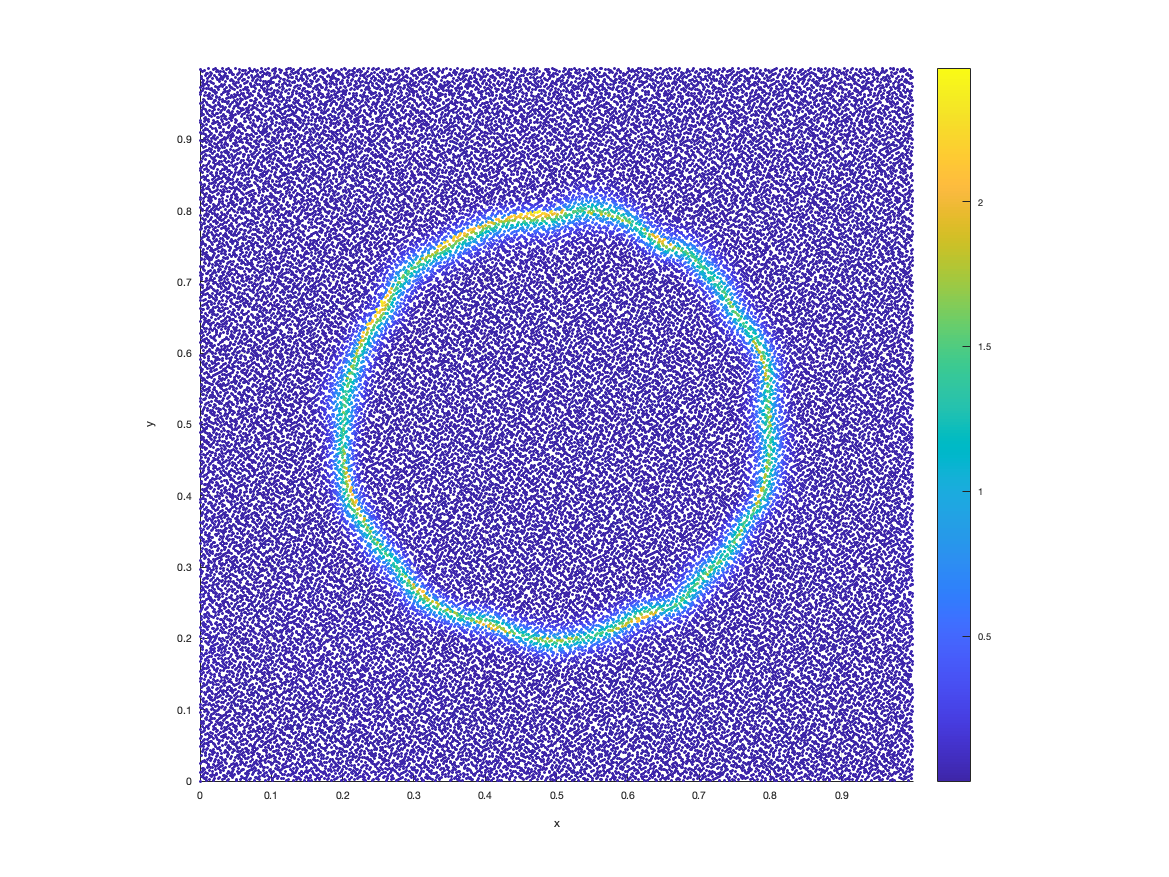}\\
		\end{tabular}
\end{center}
			\caption{Approximation of the piecewise smooth function \(f_1\) (defined in Eq.~\eqref{frankesdisc}). 
Each figure is arranged in three columns: the left column shows the surface reconstructed using either the linear RBF\(_{\text{W4}}\) method or its data-dependent counterpart NL-RBF\(_{\text{W4}}\). The central column displays the same approximation from a top (cenital) view. 
The right column depicts the pointwise absolute error over the evaluation grid.
The figure contains four rows: 
Row~1 shows the linear RBF\(_{\text{W4}}\) approximation using gridded points, 
Row~2 shows the data-dependent NL-RBF\(_{\text{W4}}\) approximation on the same grid, 
Row~3 presents the linear RBF\(_{\text{W4}}\) approximation computed from Halton points, 
and Row~4 displays the corresponding data-dependent NL-RBF\(_{\text{W4}}\) result.}
		\label{exp5_2D}
	\end{figure}

\section{Conclusions}
In this work, we have presented a data-dependent extension of the classical Shepard interpolation framework, designed specifically to reduce the smearing artifacts that typically appear near discontinuities in one and two dimensions. The key idea is to modify the Shepard weights in an adaptive manner so that, close to a discontinuity, the influence of nearby nodes becomes increasingly localized. As the weighting mechanism sharpens in these regions, the corresponding basis functions behave more like delta functions.

A crucial element of the method is the discontinuity detection strategy, which relies on smoothness indicators constructed for both structured and unstructured datasets. For regularly gridded data, we employ squared undivided second differences, while for scattered samples we use least-squares approximations of the Laplacian, all of it squared, multiplied by the square of the average local spacing. These indicators guide the adaptive adjustment of the interpolation weights.

We have also given a proof to show that the modified Shepard method is capable of reducing the smearing belt around the discontinuity.

The numerical tests carried out in both one and two dimensions illustrate the benefits of the proposed strategy. For piecewise smooth functions, the data-dependent Shepard method substantially diminishes the smearing produced by the traditional version along discontinuity curves. While the classical scheme tends to generate strong smoothing patterns that spread away from the jump, the data-dependent approach successfully reduces these effects, although a minor amount of smoothing near the discontinuity persists. 

These observations hold consistently for both uniform grids and Halton-distributed nodes, demonstrating the robustness of the methodology across different sampling settings. Overall, the data-dependent Shepard interpolation approach introduced here offers an effective and reliable tool for reconstructing piece-wise smooth data with discontinuities, and it provides a promising basis for applications in which stable and accurate recovery of this kind of data is essential.

\appendix
\section{Exponential Upper Bounds for the Kernels in Table~\ref{tabla1nucleos}}\label{apendice}

In this annex we prove that all radial basis functions (RBFs) listed in 
Table~\ref{tabla1nucleos} satisfy an exponential upper bound of the form
\begin{equation}
\label{eq:expbound}
\phi(\varepsilon r)\;\le\;C_1\,\exp(-C_2\,\varepsilon r),
\qquad r>0,\;\varepsilon>0,
\end{equation}
for suitable positive constants $C_1$ and $C_2$ depending only on the chosen
kernel.  This bound is fundamental in establishing the decay of the ``wrong
side'' weights near a discontinuity and, consequently, the shrinking of the
smearing belt as the shape parameter~$\varepsilon$ increases.

\subsection*{A.1 Gaussian kernel}
\vspace{0.5cm}

Consider the Gaussian kernel,
\[
\phi(r,\varepsilon)=\exp\!\bigl(-(\varepsilon r)^2\bigr).
\]
For any $x>0$ it holds that $x^2\ge x$ when $x\ge1$, while on $0<x<1$ the 
function $e^{-x^2}$ is still dominated by $e^{-x}$ because the exponent 
$-x^2$ is larger than $-x$.  Hence for all $x>0$,
\[
e^{-x^2} \le e^{-x}.
\]
Setting $x=\varepsilon r$ yields
\[
\phi(\varepsilon r)=e^{-(\varepsilon r)^2}
\;\le\;
e^{-\varepsilon r}.
\]
Thus \eqref{eq:expbound} holds with $C_1=1$ and $C_2=1$.

\subsection*{A.2 Inverse multiquadric (IMQ)}
\vspace{0.5cm}

Let now consider,
\[
\phi(r,\varepsilon)=\bigl(1+(\varepsilon r)^2\bigr)^{-1/2}.
\]
For all $x>0$ the inequality
\[
1+x \;\ge\; e^{x/2}
\quad \Longrightarrow \quad
(1+x)^{-1/2}\le e^{-x/4}
\]
holds.  With $x=(\varepsilon r)^2$ we obtain
\[
\phi(\varepsilon r)
\le \exp\!\left(-\frac{(\varepsilon r)^2}{4}\right)
\le \exp\!\left(-\frac{\varepsilon r}{4}\right),
\]
where the last inequality follows from the previously established Gaussian bound.
Hence \eqref{eq:expbound} holds with $C_1=1$ and $C_2=\tfrac14$.

\subsection*{A.3 Mat\'ern kernels}
\vspace{0.5cm}

For Mat\'ern $\mathcal{C}^2$ kernel, we have:
\[
\phi(r,\varepsilon)=e^{-\varepsilon r}(1+\varepsilon r).
\]
Using the standard inequality $1+x\le e^x$ for $x>0$, we obtain
\[
e^{-\varepsilon r}(1+\varepsilon r)
\;\le\;
e^{-\varepsilon r}\,e^{\varepsilon r/2}
=
e^{-(\varepsilon r)/2}.
\]
Thus the bound \eqref{eq:expbound} holds with $C_1=1$ and $C_2=\tfrac12$.
\vspace{0.5cm}

For Mat\'ern Mat\'ern $\mathcal{C}^4$ kernel:
\[
\phi(r,\varepsilon)=e^{-\varepsilon r}\left(3 + 3\varepsilon r + (\varepsilon r)^2\right).
\]
Since every polynomial is eventually dominated by an exponential, there exists
a constant $K>0$ such that
\[
3 + 3x + x^2 \;\le\; K\,e^{x/2} \qquad (x>0).
\]
Substituting $x=\varepsilon r$ yields
\[
\phi(\varepsilon r)\le K\,e^{-(\varepsilon r)/2}.
\]
Therefore \eqref{eq:expbound} holds with $C_1=K$ and $C_2=\tfrac12$.

\subsection*{A.4 Wendland kernels}
\vspace{0.5cm}

For Wendland $\mathcal{C}^2$ kernel:
\[
\phi(r,\varepsilon)
=
(1-\varepsilon r)^4_+\,(4\varepsilon r+1).
\]
When $\varepsilon r\ge 1$, the kernel vanishes and the inequality 
\eqref{eq:expbound} is trivial.  
When $0 < \varepsilon r < 1$, both factors are bounded:
\[
(1-\varepsilon r)^4 \le 1,
\qquad
4\varepsilon r+1 \le 5.
\]
Hence $\phi(\varepsilon r)\le 5$.  Since exponentials dominate constants on
compact intervals, choose any $C_2>0$ and take 
$C_1 := 5 e^{C_2}$.  Then
\[
5 \le 5e^{C_2} e^{-C_2\varepsilon r},
\qquad 0<\varepsilon r<1.
\]
Hence \eqref{eq:expbound} holds for the $\mathcal{C}^2$ Wendland kernel.
\vspace{0.5cm}

For Wendland $\mathcal{C}^4$ kernel:

The argument is identical: inside the support $0<\varepsilon r<1$, the
polynomial factor is bounded, and beyond the support the kernel vanishes.
Therefore the kernel is dominated by a suitably scaled exponential, and
\eqref{eq:expbound} follows.

\end{document}